\documentclass[11pt,letterpaper,reqno]{amsart}
\usepackage[shortlabels]{enumitem}
\usepackage{amsfonts}
\usepackage{amssymb}
\usepackage{amsthm}
\usepackage{amsmath}
\usepackage{amscd}
\usepackage{t1enc}
\usepackage{lmodern}
\usepackage{graphicx}
\usepackage{graphics}
\usepackage{pict2e}
\usepackage{epic}
\numberwithin{equation}{section}
\usepackage[margin=1.3in]{geometry}
\usepackage{epstopdf} 
\usepackage{bbm}
\usepackage{tikz-cd}
\usetikzlibrary{tikzmark,quotes}
\usepackage{mathrsfs}
\usepackage{xcolor}
\usepackage{printlen}
\usepackage{eqparbox}
\usepackage[numbers]{natbib}[sort]

\usepackage{microtype}
\usepackage[linktocpage=true,colorlinks=true,pagebackref,hyperindex]{hyperref}
\usepackage{relsize}
\usepackage[bbgreekl]{mathbbol}
\usepackage{framed}
\usepackage{bm}
\usepackage{mathalpha}
\usepackage{amsbsy}

\definecolor{burntumber}{rgb}{0.54, 0.2, 0.14}
\definecolor{coolblack}{rgb}{0.0, 0.18, 0.39}
\definecolor{mygreen}{rgb}{0.0, 0.4, 0.0}
\hypersetup{
	linkcolor=burntumber,
	citecolor=mygreen
}

\setlist[enumerate]{leftmargin=2pc}
\setlist[itemize]{leftmargin=2pc}

\makeatletter
\def\l@subsection{\@tocline{2}{0pt}{2pc}{5pc}{}}
\makeatother

\theoremstyle{plain}
\newtheorem{thm}{Theorem}[section]
\newtheorem{lemma}[thm]{Lemma}
\newtheorem{cor}[thm]{Corollary}
\newtheorem{prop}[thm]{Proposition}

\newtheorem{introthm}{Theorem}

\newtheorem{introcor}[introthm]{Corollary}
\newtheorem{conj}[thm]{Conjecture}

 \theoremstyle{definition}
\newtheorem{defn}[thm]{Definition}

\newtheorem{rmk}[thm]{Remark}
\newtheorem{?}[thm]{Problem}
\newtheorem{ex}[thm]{Example}
\newtheorem{cons}[thm]{Construction}

\DeclareSymbolFontAlphabet{\mathbb}{AMSb} 
\DeclareSymbolFontAlphabet{\mathbbl}{bbold}

\newcommand{\gx}{(\mathbf{G},\mathbf{X})}

\newcommand{\bA}{\mathbb{A}}
\newcommand{\bC}{\mathbb{C}}

\newcommand{\bG}{\mathbb{G}}

\newcommand{\bP}{\mathbb{P}}
\newcommand{\bQ}{\mathbb{Q}}
\newcommand{\bR}{\mathbb{R}}
\newcommand{\bZ}{\mathbb{Z}}

\newcommand{\brqp}{\breve{\bQ}_p}

\newcommand{\ur}{\mathrm{ur}}

\newcommand{\sP}{\mathscr{P}}

\newcommand{\sS}{\mathscr{S}}

\newcommand{\sep}{\mathrm{sep}}

\newcommand{\id}{\textup{id}}
\newcommand{\res}{\big|}

\newcommand{\Spec}{\textup{Spec}\,}
\newcommand{\Spf}{\textup{Spf }}

\newcommand{\Spd}{\textup{Spd}}

\newcommand{\Res}{{\mathrm{Res}}}
\newcommand{\bs}{\backslash}
\newcommand{\ab}{\textup{ab}}

\newcommand{\der}{\textup{der}}

\newcommand{\eE}{{\mathbf{E}}}
\newcommand{\eG}{{\mathbf{G}}}
\newcommand{\eH}{\mathsf{H}}
\newcommand{\eK}{{\sf K}}
\newcommand{\eL}{{\sf L}}

\newcommand{\eZ}{\mathbf{Z}}

\newcommand{\cG}{\mathcal{G}}

\newcommand{\cO}{\mathcal{O}}

\newcommand{\Sh}{\mathrm{Sh}}

\newcommand{\an}{\textup{an}}
\newcommand{\ad}{\textup{ad}}

\newcommand{\Gal}{\textup{Gal}}

\DeclareMathAlphabet{\mathscrbf}{OMS}{mdugm}{b}{n}

\usepackage{mathtools}
\usepackage{scalerel}
\usepackage{xcolor}
\usepackage{graphicx}
\newcommand{\mbbmu}{\mbox{$\raisebox{-0.59ex}
  {$l$}\hspace{-0.18em}\mu\hspace{-0.88em}\raisebox{-0.98ex}{\scalebox{2}
  {$\color{white}.$}}\hspace{-0.416em}\raisebox{+0.88ex}
  {$\color{white}.$}\hspace{0.46em}$}{}}

  \newcommand{\mbbmus}{\scaleobj{.65}{{\mbbmu}}}

\newcommand{\defeq}{\vcentcolon=}

\newcommand{\mc}[1]{\mathcal{#1}}
\newcommand{\mr}[1]{\mathrm{#1}}
\newcommand{\ms}[1]{\mathscr{#1}}
\newcommand{\bb}[1]{\mathbb{#1}}
\newcommand{\wt}[1]{\widetilde{#1}}
\newcommand{\wh}[1]{\widehat{#1}}
\newcommand{\mb}[1]{\mathbf{#1}}
\newcommand{\mf}[1]{\mathfrak{#1}}
\newcommand{\ov}[1]{\overline{#1}}

\newcommand{\stacks}[1]{\cite[\href{https://stacks.math.columbia.edu/tag/#1}{Tag~#1}]{stacks-project}}

    \makeatletter
    \def\paragraph{\@startsection{paragraph}{4}%
    \z@\z@{-\fontdimen2\font}%
    {\normalfont\bfseries}}
    \makeatother

    \makeatletter
\def\@seccntformat#1{%
  \expandafter\ifx\csname c@#1\endcsname\c@paragraph\else
  \csname the#1\endcsname\quad
  \fi}
\makeatother

\makeatletter

\renewcommand\subsubsection{\@secnumfont}{\bfseries}%
\renewcommand\subsubsection{\@startsection{subsubsection}{3}
  \z@{.5\linespacing\@plus.7\linespacing}{-.5em}%
  {\normalfont\bfseries}}
  
  \makeatother

\newcommand*\isomto{%
        \xrightarrow{\raisebox{-0.2 em}{\smash{\ensuremath{\sim}}}}%
    }

    \newcommand{\Q}{\mathbb{Q}}
    \newcommand{\Z}{\mathbb{Z}}
    \newcommand{\A}{\mathbb{A}}
    
    \usepackage{multicol}
    \usepackage{realhats}

    \usepackage{nicefrac}

    \usepackage{verbatim}

    \newcommand{\vn}{{\mathsmaller{\varnothing}}}

\begin{document}
\begin{abstract}
We prove a conjecture of Pappas and Rapoport for all Shimura varieties of abelian type with parahoric level structure when $p>3$ by showing that the Kisin--Pappas--Zhou integral models of Shimura varieties of abelian type are canonical. In particular, this shows that these models of are independent of the choices made during their construction, and that they satisfy functoriality with respect to morphisms of Shimura data.
\end{abstract}

\title[Canonical integral models of Shimura varieties of abelian type]{Canonical integral models of Shimura varieties of abelian type}

\author{Patrick Daniels} 
\address{Mathematics and Statistics Department, Skidmore College, Saratoga Springs, NY 12866}
\email{pdaniels@skidmore.edu}

\author{Alex Youcis}
\address{Department of Mathematics, National University of Singapore, Level 4, Block S17, 10 Lower Kent Ridge Road, Singapore, 119076}
\email{alex.youcis@gmail.com}

\maketitle
\tableofcontents

\section{Introduction}

In \cite{PR2021}, Pappas and Rapoport conjecture the existence of ``canonical'' integral models with parahoric level structure and establish this conjecture in most cases of Hodge type. In this work, we prove the conjecture of Pappas and Rapoport for (almost all) Shimura varieties of abelian type. In particular, we prove the following theorem.

\begin{introthm}[See Theorem \ref{thm:main}]\label{thm-1}
Let $p >3$. Then the Pappas--Rapoport conjecture holds for Shimura varieties of abelian type with parahoric level structure at $p$.
\end{introthm}

Let us begin to contextualize this result. Shimura varieties are algebro-geometric objects which sit at the intersection of number theory and representation theory. They take the form of a projective system $\{\Sh_\eK(\mb{G}, \mb{X})\}_{\eK}$, where $\mb{G}$ is a reductive group over $\bb{Q}$, $\mb{X}$ is some auxiliary Hodge-theoretic datum, and $\eK \subseteq \mb{G}(\bA_f)$ ranges over (neat) compact open subgroups. Shimura varieties have been inextricably linked with the Langlands program from its inception, as the cohomology of $\{\Sh_\eK(\mb{G}, \mb{X})\}_{\eK}$ ought to realize the global Langlands conjecture for $\mb{G}$ (see e.g., \cite{KottwitzAnnArbor}). Central to this cohomological understanding of the Langlands program is a conjecture which posits a ``motivic decomposition'' of the $\bar{\bb{F}}_p$-points of certain ``canonical'' integral models $\{\ms{S}_\eK\}_{\eK}$ of $\{\Sh_\eK(\mb{G}, \mb{X})\}_{\eK}$ (see \cite{LanglandsRapoport}). It is then not surprising that the construction of such natural integral models of Shimura varieties has prominently featured in many of the advances in the Langlands program.

Constructing such canonical integral models, and even characterizing what constitutes a ``canonical'' integral model, is quite difficult. Despite their importance, Shimura varieties in general remain fairly inexplicit in their definition. The variety $\Sh_\eK(\mb{G}, \mb{X})$ ought to parameterize $\mb{G}^c$-motives $M$ of type $\mb{X}$ with $\eK$-level structure on the \'etale realization of $M$ (e.g., see \cite{MilneModuli}). Here $\mb{G}^c$ is a certain quotient of $\mb{G}$, which agrees with $\mb{G}$ in the Hodge-type case, but differs from $\mb{G}$ in some abelian-type cases. While technical in nature, consideration of $\mb{G}^c$ is crucial for such a conjectural motivic interpretation. Unfortunately, such a motivic description is currently out of reach, and instead $\Sh_\eK(\mb{G}, \mb{X})$ is constructed from various abstract algebro-geometric existence results, which are not amenable to the study of integral models.

Despite this, Pappas and Rapoport in \cite{PR2021} (building off previous work of Pappas in \cite{Pappas2022}) made significant headway in a definition of such conjectural canonical integral models when $\eK = \eK_p \eK^p$ where $\eK_p = \mc{G}(\bZ_p)$ is a parahoric subgroup of $\eG(\bQ_p)$. This parahoricity-at-$p$ condition may be thought of as representing the situation where relatively little level structure is imposed at $p$. 

To explicate their work, let $E$ be the completion of the reflex field of $(\mb{G}, \mb{X})$ at a place $v$ lying over $p$. Pappas and Rapoport call a system $\{ \ms{S}_{\eK_p\eK^p}\gx\}_{\eK^p}$ of $\mc{O}_E$-models of $\{\Sh_{\eK_p\eK^p}(\mb{G}, \mb{X})_E\}_{\eK^p}$ \textit{canonical} if 
\begin{enumerate}[(i)]
    \item the transition maps between the varying levels are finite \'etale,
    \item if $R$ is a characteristic $(0,p)$ discrete valuation ring over $\mc{O}_E$ then
    \begin{align*}
        \left(\varprojlim_{\eK^p} \Sh_{\eK_p\eK^p}\gx\right)(R[\nicefrac{1}{p}]) = \left(\varprojlim_{\eK^p} \sS_{\eK_p\eK^p}\gx\right)(R),
    \end{align*}
    \item there exists a $\mc{G}^c$-shtuka on $\ms{S}_{\eK_p\eK^p}\gx$ modeling the \'etale realization functor on $\Sh_{\eK_p\eK^p}(\mb{G}, \mb{X})_E$, 
    \item and for every $\bar{\bb{F}}_p$-point $x$ of $\ms{S}_{\eK_p\eK^p}\gx$, there is an isomorphism 
    \begin{align}
            \Theta_x: \left(\mc{M}^{\mr{int}}_{\cG^c,b_x,{\mbbmus_h^c}}\right)^\wedge_{/x_0} \isomto \left(\sS_{\eK_p\eK^p}\gx^\wedge_{/x}\right)^\lozenge,
    \end{align}
    such that $\Theta_x^\ast(\mathscr{P}_\eK) = \mathscr{P}^\mathrm{univ}$.
\end{enumerate}
Here the \'etale realization is a certain $\mc{G}^c(\bZ_p)$-local system on $\Sh_{\eK_p\eK^p}\gx$ which should be thought of as the \'etale realization of the ``universal motive'' on $\Sh_{\eK_p\eK^p}\gx$, but whose definition is unconditional on such a motivic interpretation. The group $\mc{G}^c$ is a parahoric $\bZ_p$-model for $\mb{G}^c$ which is determined from $\mc{G}$ using Bruhat--Tits theory. Furthermore, when we speak of shtukas here we mean in the sense of \cite{PR2021}, and $\ms{P}_{\eK^p}$ modeling the \'etale realization is meant in terms of the construction in \cite[\S2.6]{PR2021}. Finally, $(\mc{M}^\mr{int}_{\cG^c,b_x,{\mbbmus_h^c}})^\wedge_{/x_0}$ is the completion of the integral local Shimura variety in the sense of \cite[Chapter 25]{SW2020} with its universal shtuka $\ms{P}^\mathrm{univ}$. We refer the reader to \S\ref{ss:PR-conj} for the definition of $b_x$ and $x_0$, but mention that ${\mbbmu_h^c}$ is the geometric conjugacy class of cocharacters of $G^c$ induced by the hodge cocharacter $\mbbmu_h$ of $G$, where we are writing $G=\mb{G}_{\Q_p}$ and $G^c=\mb{G}^c_{\Q_p}$.

To help process this definition, note that, given the motivational moduli description of $\Sh_{\eK_p\eK^p}\gx$, it is reasonable to expect that $\ms{S}_{\eK_p\eK^p}\gx$ should be a moduli space of $\mc{G}^c$-motives. One can view conditions (iii) and (iv) as saying that $\ms{S}_{\eK_p\eK^p}\gx$ possesses a shtuka, which should be the ``shtuka realization'' of the universal $\mc{G}^c$-motive, which, while not provably universal, is universal everywhere formally locally. Furthermore, condition (i) is natural, and condition (ii) is a technical condition ensuring that the members of the system $\{\ms{S}_{\eK_p\eK^p}\gx\}_{\eK^p}$ have sufficiently large special fibers.

Pappas and Rapoport prove that canonical integral models are unique if they exist, hence they do provide a good notion of models at parahoric level. Constructing such models is difficult though, and seems out of reach for general Shimura varieties with current methods. That said, when $\gx$ is of so-called ``abelian type'', the situation is much improved. Shimura varieties of abelian type may be thought of intuitively as moduli spaces of abelian motives with extra structure, but this again largely motivational (although see \cite[Theorem 11.16]{MilneModuli} for a partial justification of this claim).

Building off the work of Kisin (and others), Kisin--Pappas and Kisin--Zhou have constructed integral models for Shimura varieties of abelian type at parahoric level\footnote{The construction of the Kisin--Pappas--Zhou models requires some minor additional assumptions when $p$ is 3, see Remark \ref{rmk:acceptable}}. These models have proven to be quite useful in applications (e.g., see \cite{KZ21}), but they have not yet been well-studied. For instance, there is no a priori functoriality for these models, and it is not even clear that the models are independent of the various ancillary choices made in their construction. Our main theorem clarifies these points.

\begin{introthm}[see Theorem \ref{thm:main}]
    Kisin--Pappas--Zhou models are canonical.
\end{introthm}

One of the most pleasing consequences of this result is the following.

\begin{introcor}[See Corollary \ref{cor:omnibus-independence} and Theorem \ref{thm:functoriality}]\label{cor-1} Kisin--Pappas--Zhou models are independent of all choices, and are functorial in the triple $(\mb{G}, \mb{X}, \mc{G})$.
\end{introcor}

We expect that our main result will have applications beyond this functoriality. For example, in \cite{DvHKZIgusa}, it is shown that Kisin--Pappas--Zhou models of Hodge type admit a fiber product decomposition
\begin{center}
    \begin{tikzcd}
        \ms{S}_{\eK}\gx^\diamond
            \arrow[r] \arrow[d]
        & \mc{G}\text{-}\mb{Sht}_{{\mbbmus_h}} 
            \arrow[d]
        \\ \operatorname{Igs}_{\eK}(\mb{G},\mb{X})
            \arrow[r]
        & \mathrm{Bun}_G
    \end{tikzcd}
\end{center}
at the level of $v$-sheaves. Here $\ms{S}_{\eK_p\eK^p}\gx^\diamond$ denotes the $v$-sheaf associated to the $p$-adic completion of $\ms{S}_{\eK_p\eK^p}\gx$, $\mc{G}\text{-}\mb{Sht}_{\mbbmus_h}$ is the moduli $v$-stack of $\mc{G}$-shtukas bounded by $\mbbmu_h$, $\mathrm{Bun}_G$ is the stack of $G$-bundles on the Fargues-Fontaine curve, and $\mathrm{Igs}_{\eK}$ is the \textit{Igusa stack} defined in \cite{DvHKZIgusa}, extending work of Zhang \cite{Zhang} in the PEL-type case. Such a decomposition provides a crucial ingredient for the proof of torsion vanishing results for the cohomology of Shimura varieties, see \cite{Hamann-Lee} and \cite{DvHKZIgusa}. Our results produce a map of $v$-sheaves $\ms{S}_{\eK}\gx^\diamond \to \mc{G}^c\text{-}\mb{Sht}_{\mbbmus_h^c}$ for all Kisin--Pappas--Zhou models, and we expect that with this it will be possible to extend the results of \cite{Zhang} and \cite{DvHKZIgusa} to the abelian-type case.

Let us close by mentioning a few ideas in the proof of Theorem \ref{thm:main}. A useful heuristic for Shimura data of abelian type is that they are ``spanned'' by two (largely orthogonal) extremes: Shimura data of Hodge type and Shimura data of toral type. While a useful guiding principle, it requires some care to put this idea into practice. The main ingredient in our proof is the development of a precise formalism for realizing this heuristic. We do this by extending ideas of Lovering (see \cite{Lovering}) from the setting of Kisin's canonical integral models at hyperspecial level to the setting of arbitrary Kisin--Pappas--Zhou models. We are then able to prove Theorem \ref{thm-1} by reducing to the Hodge-type case (established in \cite{PR2021} and \cite{DvHKZ}, see Remark \ref{rem:Hodge}) and toral-type case (established in \cite{DanielsToral}). We expect these methods to have applications more generally. 

\begin{rmk}\label{rem:vHS}
    We would like to direct the reader's attention to \cite{vHS}. In the process of this work, van Hoften and Sempliner have occasion to prove a modified version of Theorem \ref{thm-1}. Namely, they establish that Kisin--Pappas--Zhou models are ``canonical'' in an appropriate sense where one replaces $\mc{G}^c$ with the more coarse quotient $\mathcal{G}^\mathrm{ad}$. Their proof avoids much of the most subtle group theory used in proving Theorem \ref{thm-1}, but it is still sufficient to establish Corollary \ref{cor-1} (see \cite{vHS}). That said, the $\mc{G}^\ad$-version of Theorem \ref{thm-1} is insufficient for many ``motivic'' applications, such as extending the results of \cite{DvHKZIgusa} to the abelian-type case.
\end{rmk}

\paragraph{Acknowledgements} The authors would like to thank David Schwein, Gopal Prasad, Pol van Hoften, and Bogdan Zavalyov for several useful conversations. In addition, we would like to thank Pol van Hoften for pointing several subtle points in an earlier draft of this article. Part of this work was conducted during a visit to the Hausdorff Research Institute for Mathematics, funded by the Deutsche Forschungsgemeinschaft (DFG, German Research Foundation) under Germany's Excellence Strategy – EXC-2047/1 – 390685813, and the authors thank the Hausdorff Institute for their hospitality. During a portion of this work, the second author was supported as a JSPS fellow by a KAKENHI Grant-in-aid (grant number 22F22323).

\medskip

\paragraph{Notation and conventions}

\begin{itemize}
    \item For an $S$-scheme $X$, and a morphism of schemes $T\to S$, we shorten the notation $X\times_S T$ to $X_T$. When $T=\Spec(R)$, we shorten this notation further to $X_R$. Similar conventions hold in other geometric categories.
    \item For a scheme $X$ and subset $S$, we always consider the Zariski closure $\ov{S}$ as a reduced subscheme, unless stated otherwise.
    \item Unless stated otherwise, we use the notation and terminology concerning $v$-sheaves, Fargues--Fontaine curve, and shtukas as in \cite[\S2]{PR2021}. In particular, when we speak of a shtuka over a $\bZ_p$-scheme $X$, we mean a shtuka over the $v$-sheaf $X^{\lozenge /}$ in the sense of \cite[Definition 2.3.2]{PR2021}.
    \item We use $\mc{G}\text{-}\mb{Sht}(\ms{X})$ and $\mc{G}\text{-}\mb{Sht}_{\mbbmus}(\ms{X})$ to denote the groupoid of $\mc{G}$-shtukas on $\ms{X}$ and groupoid of $\mc{G}$-sthukas bounded by $\mbbmu$ on $\ms{X}$, respectively.
    \item A reductive group over a field $F$ is always assumed connected.
\end{itemize}

\section{Preliminaries}

We begin by establishing some preliminaries which will be essential for the proofs of the main results. In particular, we review the Bruhat--Tits theory which is necessary for our purposes, and discuss torsors for fiber products of group schemes. 

\subsection{Some basic Bruhat--Tits theory}\label{ss:BT-theory}
In this section we review some basic Bruhat--Tits theory. For more in-depth discussion in this direction we suggest the reader to consult \cite{BTII} or \cite{KalethaPrasad} and the references therein.

Throughout this section we fix a discretely valued field $K$ with perfect residue field $k$. We denote by $K^\ur$ the maximal unramified extension of $K$ inside of $K^\mr{sep}$, and by $\Gamma$ the Galois group $\Gal(K^\ur/K)$. Denote by $I\defeq \ker(\Gal(K^\mr{sep}/K)\to \Gamma)$ the inertia group of $K$.

\begin{rmk} As $K^\ur$ is another discretely valued field with perfect residue field, we will often apply to $K^\ur$ without comment definitions which are initially stated for $K$.
\end{rmk}

\subsubsection{The building and its functorialities}

Assume $G$ is a reductive group over $K$. As in \cite[Definition 7.6.1]{KalethaPrasad} (cf.\@ \cite[Definition 9.2.8]{KalethaPrasad}), associated to $G$ is a building (in the sense of \cite[Definition 1.5.5]{KalethaPrasad}) 
\begin{equation*}
    \ms{B}(G,{K^\ur})\defeq\ms{B}(G_{K^\ur},{K^\ur}),
\end{equation*} 
where the right-hand side is as in \cite[Definition 7.6.1]{KalethaPrasad}. This building comes equipped with an action of $G({K^\ur})\rtimes \Gamma$. Moreover, the action of $G({K^\ur})$ factors through the natural map $G({K^\ur})\to G^\ad({K^\ur})$. 

Additionally, we have a building 
\begin{equation*}
    \ms{B}(G,K)\defeq \ms{B}(G,{K^\ur})^\Gamma,
\end{equation*}
which naturally comes equipped with an action of $G(K)$ which factorizes through the map $G(K)\to G^\ad(K)$. To distinguish these two objects, we call the $\ms{B}(G,{K^\ur})$ the \emph{building associated to $G$ over ${K^\ur}$} and $\ms{B}(G,K)$ the \emph{building associated to $G$ over $K$}.

\begin{rmk} The building(s) we are using here are more often called the \emph{reduced Bruhat--Tits building(s)}. While there exists a notion of the \emph{extended Bruhat--Tits building(s)}, we will not need these notions here and, in fact, this simpler object is better suited to our purposes. See \cite{KalethaPrasad} for more details.
\end{rmk}

Various functorial properties for these buildings have been established in the literature, see \cite{Landvogt2000} and \cite[\S14]{KalethaPrasad}. However, in this article we require only two basic functoriality statements. 

The first concerns quotient maps of reductive groups.

\begin{lemma}[{cf.\@ \cite[Proposition 14.1.1]{KalethaPrasad}}]\label{lem:map-of-buildings} Suppose that $f\colon G\to H$ is a faithfully flat map of reductive groups over $K$. Then, there exists a unique surjective map 
\begin{equation*}
    f_\ast\colon \ms{B}(G,{K^\ur})\to \ms{B}(H,{K^\ur}),
\end{equation*}
which is equivariant for the map $G({K^\ur})\rtimes \Gamma\to H({K^\ur})\rtimes \Gamma$. A similar statement holds for the buildings over $K$. 
\end{lemma}

For the second statement, we first make the following definition. For a field $F$, a map $f\colon G\to H$ of reductive groups over $F$ is an \emph{ad-isomorphism} if $f(Z(G))\subseteq Z(H)$, and the induced map $f^\ad\colon G^\ad\to H^\ad$ is an isomorphism. We record for later use the following proposition, whose proof is an exercise in assembling well-known facts about reductive groups over fields (e.g., see \cite{Milne2017}).

\begin{prop}\label{prop:ad-iso-equiv} Let $f\colon G\to H$ be a map of reductive groups over $F$. Then, the following conditions on $f$ are equivalent:
\begin{enumerate}
    \item $f$ is an ad-isomorphism,
    \item $f^{-1}(Z(H))=Z(G)$,
    \item the induced map $f\colon G^\der\to H^\der$ is a central isogeny.\footnote{Any isogeny is central if it is of degree coprime to the characteristic of $F$ (see \cite[Remark 12.39]{Milne2017}).}
\end{enumerate}
\end{prop}

The second basic functoriality result is the following. 

\begin{lemma}\label{lemma:buildings-ad-isom}
        Suppose $f\colon G \to H$ is an ad-isomorphism of reductive groups over $K$. Then $f$ induces a bijection of buildings
        \begin{align*}
            f_\ast\colon \mathscr{B}(G, {K^\ur}) \isomto \mathscr{B}(H, {K^\ur}),
        \end{align*}
        equivariant for the map $G({K^\ur})\rtimes \Gamma\to H({K^\ur})\rtimes \Gamma$. Moreover, if $g\colon H\to L$ is another ad-isomorphism, then $(g\circ f)_\ast=g_\ast\circ f_\ast$. A similar statement holds for buildings over $K$.
    \end{lemma}
\begin{proof} As discussed on \cite[p.\@ 264]{KalethaPrasad}, there are natural $\Gamma$-equivariant bijections
\begin{equation*}
    \ms{B}(G,{K^\ur})\isomto\ms{B}(G^\ad,{K^\ur}),\qquad \ms{B}(H,{K^\ur})\isomto \ms{B}(H^\ad,{K^\ur}),
\end{equation*}
which are equivariant for the maps $G({K^\ur})\to G^\ad({K^\ur})$ and $H({K^\ur})\to H^\ad({K^\ur})$, respectively. These, together with the $\Gamma$-equivariant bijection $\ms{B}(G^\ad,{K^\ur})\isomto \ms{B}(H^\ad,{K^\ur})$ induced by $f^\ad$, gives us a $\Gamma$-equivariant bijection $\ms{B}(G,{K^\ur})\to\ms{B}(H,{K^\ur})$. That this is equivariant for the map $G({K^\ur})\to H({K^\ur})$ is clear from construction as $G({K^\ur})$ and $H({K^\ur})$ act through their images in $G^\ad({K^\ur})$ and $H^\ad({K^\ur})$, respectively. The proof for the buildings over $K$ follows by passing to the fixed points for $\Gamma$.
\end{proof}

\subsubsection{Parahoric groups and group schemes}

Let $G$ be a reductive group over $K$. As in \cite[\S7]{KottwitzIsocrystalsII} (see also \cite[\S11.1]{KalethaPrasad}), one may construct a group homomorphism
\begin{equation*}
    \kappa_G\colon G({K^\ur})\to \pi_1(G)_I,
\end{equation*}
where $\pi_1(G)$ denotes the fundamental group of Borovoi (see \cite[\S1]{Borovoi}), a finitely generated abelian group, and the subscript $(-)_I$ denotes $I$-coinvariants. We call $\kappa_G$ the \emph{Kottwitz homomorphism} associated to $G$. We let $\kappa_G\otimes 1$ denote the composition of $\kappa_G$ with the map $\pi_1(G)_I\to \pi_1(G)_I\otimes_\Z\Q$.

We denote the kernel of $\kappa_G$ and $\kappa_G\otimes 1$ by $G({K^\ur})^0$ and $G(K^\ur)^1$, respectively.\footnote{The groups $G(K^\ur)^0$ and $G(K^\ur)^1$ may be given more concrete definitions. See \cite[Notation 2.6.15 and Definition 2.6.23]{KalethaPrasad} for these definitions, and \cite[Lemma 11.5.2 and Proposition 11.5.4]{KalethaPrasad} for the proofs that these groups coincides with our definition.} The Kottwitz homomorphism (and thus the subgroups $G({K^\ur})^0$ and $G(K^\ur)^1$) is functorial in $G$. More precisely, if $f\colon G\to H$ is a map of reductive groups over $K$, then the diagram
\begin{equation*}
    \begin{tikzcd}[column sep=2.25em]
	{G({K^\ur})} & {\pi_1(G)_I} \\
	{H({K^\ur})} & {\pi_1(H)_I}
	\arrow["f"', from=1-1, to=2-1]
	\arrow["{\kappa_G}", from=1-1, to=1-2]
	\arrow["{\kappa_H}"', from=2-1, to=2-2]
	\arrow["{\pi_1(f)}", from=1-2, to=2-2]
\end{tikzcd}
\end{equation*}
commutes. 

We then have the following definition of a parahoric group scheme.

\begin{defn} A group $\mc{O}_K$-scheme $\mc{G}$ is \emph{parahoric} if:
\begin{enumerate}
    \item $G\defeq \mc{G}_K$ is a reductive group over $K$,
    \item $\mc{G}$ is a smooth affine group $\mc{O}_K$-scheme with connected fibers,
    \item there exists a point $x$ in $\ms{B}(G,{K^\ur})$ (equiv.\@ $\ms{B}(G,K)$) such that 
    \begin{equation}\label{eq:parahoric-defn}
        \mc{G}(\mc{O}_{K^\ur})=G({K^\ur})^0\cap \mr{Stab}_{G({K^\ur})}(x).
    \end{equation}
\end{enumerate}
\end{defn}

\begin{rmk} That this definition agrees with that originally defined in \cite{BTII} is (essentially) the content of \cite[Proposition 3]{HainesRapoport}.
\end{rmk}

As in \cite[\S9.2.6]{KalethaPrasad}, associated to any point $x$ of $\ms{B}(G,K)$ is a parahoric group scheme denoted by $\mc{G}_x^\circ$, characterized by Equation \eqref{eq:parahoric-defn}. Moreover, every parahoric group scheme arises in this way. A subgroup of $G(K)$ is called \emph{parahoric} if it is of the form $\mc{G}(\mc{O}_K)$ for a parahoric group $\mc{O}_K$-scheme $\mc{G}$.

\subsubsection{Parahoric models of tori and $R$-smooth tori}\label{sss:parahoric-models-tori}

Suppose that $T$ is a a torus over $K$. By Lemma \ref{lemma:buildings-ad-isom} it's clear that $\ms{B}(T,K)$ is a singleton, and thus that $T$ has a unique parahoric model which we denote by $\mc{T}$. As $\ms{B}(T,K)$ is a singleton, $\mc{T}(\mc{O}_K)$ is equal to $T(K)^0$.

Let $\mc{T}^\mr{lft}$ be the N\'eron model of $T$ (e.g., see \cite[\S B.7--B.8]{KalethaPrasad} and \cite{BLR}).

\begin{prop}[{see \cite[\S B.7--B.8]{KalethaPrasad}}] There is a functorial identification $\mc{T}\simeq (\mc{T}^\mr{lft})^\circ$.
\end{prop}

Fix a finite Galois extension $K_1$ of $K$ splitting $T$, and set $T_1=\mr{Res}_{K_1/K}T_{K_1}$  which admits a natural embedding $T\to T_1$. 

\begin{defn}[{\cite[Definition 2.4.3]{KZ21}}]\label{defn:R-smooth} The torus $T$ is called \emph{$R$-smooth} if its Zariski closure in $\mc{T}_1^\mr{lft}$ is a smooth group $\mc{O}_K$-scheme.\footnote{As in \cite[\S4.4.8]{BTII}, one may show that the definition of being $R$-smooth does not depend on the choice of $K_1$.}
\end{defn}

Consider now a reductive group $G$ over $K$. As $G_{K^\ur}$ is quasi-split by Steinberg's theorem (see \cite[Theorem 1.9]{Steinberg}), there is a unique $G(K^\ur)$-conjugacy class of maximally split maximal tori in $G_{K^\ur}$ which are precisely the centralizers of maximal split tori (see \cite[\S20]{BorelLAG}). 

\begin{defn}\label{defn:R-smooth-G}
    We say that a reductive $K$-group $G$ is \emph{$R$-smooth} if one (equivalently any) maximally split maximal torus of $G_{K^\ur}$ is $R$-smooth. 
\end{defn} 

As the formation of N\'eron models commutes with unramified base change, Definition \ref{defn:R-smooth-G} agrees with that of Definition \ref{defn:R-smooth} when $G$ is a torus.

\subsubsection{Abelianization of parahoric group schemes}\label{sss:parahorics-and-abelianization}

Suppose that $G$ is a reductive group over $\mc{O}_K$, and consider the short exact sequence of reductive groups
\begin{equation*}
    1\to G^\der\xrightarrow{i}G\xrightarrow{f}G^\ab\to 1.
\end{equation*}
By Lemma \ref{lem:map-of-buildings} and Lemma \ref{lemma:buildings-ad-isom} we obtain maps of buildings
\begin{equation*}
    i_\ast\colon \ms{B}(G^\der,K)\isomto \ms{B}(G,K),\qquad f_\ast\colon \ms{B}(G,K)\to\ms{B}(G^\ab,K),
\end{equation*}
where the first map is a bijection and the second is surjective, and the maps are equivariant for $G^\der(K)\to G(K)$ and $G(K)\to G^\ab(K)$, respectively. Similar statements hold for the buildings over ${K^\ur}$.

Let us fix a point $y$ of $\ms{B}(G,K)$ and set $x=i_\ast^{-1}(y)$ and $z=f_\ast(z)$. By the functoriality of the Kottwitz homomorphism and the equivariant properties of these maps we see by \cite[Corollary 2.10.10]{KalethaPrasad} that $i$ and $f$ extend to maps $i\colon \mc{G}_x^\circ\to \mc{G}_y^\circ$ and $f\colon \mc{G}_y^\circ\to\mc{G}_z^\circ$.

\begin{prop}\label{prop:parahoric-adjoint}  If $G^\der$ is $R$-smooth, then we have a short exact sequence
\begin{equation*}
    1\to \mc{G}_x^\circ\to \mc{G}_y^\circ\to \mc{G}_z^\circ\to 1,
\end{equation*}
of group $\mc{O}_K$-schemes. 

\end{prop}
\begin{proof} 

First observe that by the setup, we have a diagram
\begin{equation*}
    \mc{G}_x^\circ \to \mc{G}_y^\circ \to \mc{G}_z^\circ.
\end{equation*}
We first claim that $\mc{G}_y^\circ(\mc{O}_{K^\ur}) \to \mc{G}_z^\circ(\mc{O}_{K^\ur})$ is surjective. Let $T$ be a maximally split maximal torus in $G_{K^\ur}$ as in \cite[\S 8.3.6]{KalethaPrasad}. Then as in loc. cit. there exists a map $\mc{T}(\mc{O}_{K^\ur}) \to \mc{G}_y^\circ(\mc{O}_{K^\ur})$, thus it suffices to show that $\mc{T}(\mc{O}_{K^\ur}) \to \mc{G}_z^\circ(\mc{O}_{K^\ur})$ is surjective. That said, observe that $\ker(T \to G^\ab) = T\cap G^\der$ is a torus and therefore, by \cite[Lemma 2.5.20]{KalethaPrasad} the map $T(K^\ur)^0 \to G^\ab(K^\ur)^0$ is surjective, so the claim follows. 

We now prove that $f\colon \mc{G}_y^\circ\to \mc{G}_z^\circ$ is faithfully flat. To see that $f$ is surjective, we observe that $f$ is surjective over $K$, being the map $G\to G^\ab$. Thus, it suffices to show that $f_{k^\sep}$ contains every point of $\mc{G}_z^\circ(k^\mr{sep})$ in its image. As $\mc{G}_z^\circ$ is smooth, we know that $\mc{G}_z^\circ(\mc{O}_{K^\ur})\to \mc{G}_z^\circ(k^\mr{sep})$ is surjective (e.g., see \cite[Th\`eor\'eme 18.5.17]{EGA4-4}), and, since $\mc{G}_y^\circ(\mc{O}_{K^\ur})\to \mc{G}_z^\circ(\mc{O}_{K^\ur})$ is surjective, the claim follows. To show that $f$ is flat, we may, by the crit\'ere de platitude par fibres (see \cite[Th\`eor\'me 11.3.10]{EGA4-3}), show instead that $f_K$ and $f_k$ are flat. But $\mc{G}_z^\circ\to\Spec(\mc{O}_K)$ is smooth (and so $(\mc{G}_z^\circ)_K$ and $(\mc{G}_z^\circ)_k$ are reduced), so this follows from \cite[Theorem 3.34]{Milne2017}. 

Finally, since $G^\der$ is $R$-smooth, it follows from \cite[Proposition 2.4.9]{KZ21} that the map $\mc{G}_x^\circ\to\mc{G}_y^\circ$ is a closed immersion. Moreover, as $\mc{G}_y^\circ\to\mc{G}^\circ_z$ is flat, the group $\mc{O}_K$-scheme $\ker(\mc{G}_y^\circ\to\mc{G}_z^\circ)$ is flat. So, to show that the closed immersion $\mc{G}_x^\circ\to \mc{G}_y^\circ$ identifies $\mc{G}_x^\circ$ with $\ker(\mc{G}_y^\circ\to\mc{G}_z^\circ)$ it suffices to check this claim on the generic fiber,\footnote{Suppose that $X$ is a flat $\mc{O}_K$-scheme and $Z_i\subseteq X$ are $\mc{O}_K$-flat closed subschemes such that $(Z_1)_K=(Z_2)_K$, then $Z_1=Z_2$. Without loss of generality, $X=\Spec(R)$ for some $\mc{O}_K$-algebra $R$. Write $I_i$ for the ideal defining $Z_i$. Let $x$ belong to $I_1$. Then $(Z_1)_K=(Z_2)_K$, so $x$ belongs to $(I_2)_K$ and thus there exists some $m\geqslant 0$ such that $\pi^m x$ belongs to $I_2$, with $\pi$ a uniformizer of $K$. So, $\pi^m x$ is zero in $R/I_2$ which, as $R/I_2$ is $\mc{O}_K$-flat, implies that $x$ is zero in $R/I_2$, and so belongs to $I_2$. So, $I_1\subseteq I_2$. By symmetry, $I_2\subseteq I_1$.} where it is clear.
\end{proof}

In practice, if we have fixed a parahoric model $\mc{G}$ of $G$, then we write the unique parahoric model of $G^\ab$ by $\mc{G}^\ab$. Given Proposition \ref{prop:parahoric-adjoint} this notational abuse is not too severe. Similarly, we denote the associated parahoric model of $G^\mr{der}$ by $\mc{G}^\mr{der}$.

\subsubsection{Parahoric group schemes and central quotients}\label{sss:parahoric-central-quotient}

Let $f\colon G\to G'$ be a faithfully flat map such that $Z\defeq \ker(f)$ is a torus over $K$. Let $x$ be a point of $\ms{B}(G,K)$ and $x'$ its image under the map the map $f_\ast$ from Lemma \ref{lem:map-of-buildings}. Let us write $\mc{G}=\mc{G}_x^\circ$ and $\mc{G}'=\mc{G}_{x'}^\circ$

As $f_\ast$ is equivariant for the map $G({K^\ur})\to G'({K^\ur})$ and the Kottwitz homomorphism is functorial, we deduce that $f$ maps $\mc{G}(\mc{O}_{K^\ur})$ into $\mc{G}'(\mc{O}_{K^\ur})$. Thus, by \cite[Corollary 2.10.10]{KalethaPrasad} $f$ uniquely lifts to a map $\mc{G}\to\mc{G}'$. When $Z$ is $R$-smooth, one can say more.

\begin{prop}[{{\cite[Proposition 2.4.14]{KZ21}}}]\label{prop:KZ-R-smooth} Suppose that $Z$ is an $R$-smooth torus. Then, there exists a short exact sequence of group $\mc{O}_K$-schemes
\begin{equation*}
    1\to \mc{Z}\to \mc{G}\xrightarrow{f} \mc{G}'\to 1,
\end{equation*}
where $\mc{Z}$ is the Zariski closure of $Z$ in $\mc{G}$. Moreover, $\mc{Z}$ is a smooth group $\mc{O}_K$-scheme. 
\end{prop}

\subsubsection{Parahoric group schemes and ad-isomorphisms}

Suppose that $i\colon G\to H$ is an injective ad-isomorphism of reductive groups over $K$.  Suppose that $\mc{H}$ is a parahoric model of $H$, and $y$ is a point of $\ms{B}(H,k)$ such that $\mc{H}$ is isomorphic to $\mc{H}_y^\circ$. By Lemma \ref{lemma:buildings-ad-isom} there is a unique point $x$ of $\ms{B}(G,K)$ such that $y=i_\ast(x)$. 

The following basic proposition provides a recognition principle for $\mc{G}_x^\circ$.

    \begin{prop}\label{prop:ad-isomorphism-parahoric}
        A model $\mc{G}$ of $G$ is isomorphic to $\mc{G}_x^\circ$ if it is a smooth affine group $\mc{O}_K$-scheme with connected fibers and
        \begin{equation}\label{eq:ad-isomorphism-parahoric}
            \mathcal{G}(\cO_{K^\ur}) = \mathcal{H}(\cO_{K^\ur}) \cap G({K^\ur}). 
        \end{equation}
        Moreover, $i$ extends to a map of $\cO_k$-group schemes $\mathcal{G} \to \mathcal{H}$ which is a closed embedding if $G$ is $R$-smooth.
        
    \end{prop}
    \begin{proof}
        By the definition of parahoric group schemes, and the definition of $\mc{G}_x^\circ$, it suffices to show that
        \begin{align}\label{eq:HR-for-G}
            \mathcal{G}(\cO_{K^\ur}) = G({K^\ur})^0 \cap \mr{Stab}_{G({K^\ur})}(x).
        \end{align}
        Since $i_\ast$ defines a  bijection $\mathscr{B}(G,{K^\ur}) \isomto \mathscr{B}(H, {K^\ur})$
        which is equivariant for the map $G({K^\ur}) \to H({K^\ur})$, we deduce that
        \begin{align}\label{eq:fixer-intersection}
            \mr{Stab}_{G({K^\ur})}(x) = \mr{Stab}_{H({K^\ur})}(y) \cap G({K^\ur}).
        \end{align}
        The inclusion $\mathcal{G}(\cO_{K^\ur}) \subseteq \mr{Stab}_{G({K^\ur})}(x)$ then follows immediately from \eqref{eq:ad-isomorphism-parahoric}. Additionally, as $\mathcal{G}$ is a smooth affine group $\mc{O}_K$-scheme with connected fibers, it follows from \cite[Proposition 8.3.15]{KalethaPrasad} that $\mathcal{G}(\cO_{K^\ur}) \subseteq G({K^\ur})^0$. Thus, $\mc{G}(\mc{O}_{K^\ur})\subseteq G({K^\ur})^0\cap\mr{Stab}_{G({K^\ur})}(x)$.
        
        On the other hand, if $g$ belongs to $G({K^\ur})^0 \cap \mr{Stab}_{G({K^\ur})}(x)$, then $i(g)$ belongs to $\mathcal{H}(\cO_{K^\ur})$ by functoriality of the Kottwitz homomorphism along with the identity (\ref{eq:fixer-intersection}). Hence \eqref{eq:ad-isomorphism-parahoric} implies that $g$ belongs to  $\mathcal{G}(\cO_{K^\ur})$.

        Finally, the fact that $i$ extends to a morphism $i\colon\mc{G}\to\mc{H}$ follows from \cite[Corollary 2.10.10]{KalethaPrasad} by \eqref{eq:ad-isomorphism-parahoric}. That $i$ is a closed embedding if $G$ is $R$-smooth follows from \cite[Proposition 2.4.9]{KZ21}, whose proof works under the assumption that $i$ is an ad-isomorphism.
    \end{proof}

\subsection{Fiber products of groups} In this subsection we aim to collect some facts concerning fiber products of groups, in various contexts, that will be useful below.

\subsubsection{Smoothness and connectedness of fiber products of group schemes} Fix a connected scheme $S$ and consider a Cartesian diagram of group $S$-schemes
\begin{equation*}
    \begin{tikzcd}
	{\mathcal{D}} & {\mathcal{A}} \\
	{\mathcal{C}} & {\mathcal{B}.}
	\arrow["f", from=1-2, to=2-2]
	\arrow["g"', from=2-1, to=2-2]
	\arrow["q"', from=1-1, to=2-1]
	\arrow["p", from=1-1, to=1-2]
	\arrow["\lrcorner"{anchor=center, pos=0.125}, draw=none, from=1-1, to=2-2]
\end{tikzcd}
\end{equation*}
Set $\mc{K}$ to be the group $S$-scheme $\ker(f)$.

\begin{prop}\label{prop:fiber-product-properties} Suppose that $f$ is faithfully flat and quasi-compact, and that $\mc{K}$ and $\mc{C}$ are smooth group $S$-schemes with connected fibers. Then, $\mc{D}$ is a smooth group $S$-scheme with connected fibers.
\end{prop}

One may replace all instances of ``connected'' in the assumptions and conclusion of the above statement by ``geometrically connected'', ``irreducible'', or ``geometrically irreducible'' by standard theory concerning algebraic groups (see \cite[Summary 1.36]{Milne2017}). We use this observation freely below. 
\begin{proof}[Proof of Proposition \ref{prop:fiber-product-properties}] The morphism $\mc{A}\times_\mc{B}\mc{A}\to \mc{K}\times_S \mc{A}$ given by $(g,h)\mapsto (gh^{-1},h)$ is an isomorphism. But, as $\mc{K}\to S$ is smooth, thus so is $\mc{K}\times_S \mc{A}\to\mc{A}$, and thus so is $\mc{A}\times_\mc{B}\mc{A}\to \mc{A}$. As $f$ is faitfully flat and quasi-compact, we deduce by fpqc descent for smoothness (e.g., see \stacks{0429}) that $f$ is smooth. By stability of smoothness under base change we deduce that $q\colon \mathcal{D}\to \mc{C}$ is smooth, and as $\mc{C}\to S$ is smooth thus so is the composition $\mc{D}\to S$.

To prove that $\mathcal{D}\to S$ has connected fibers, let $s$ be an element of $S$ and consider the map $q_s\colon \mc{D}_s\to \mc{C}_s$. By assumption we have that $\mc{C}_s$ is connected and $q_s\colon \mc{D}_s\to \mc{C}_s$ is surjective. Moreover, by the argument in the previous paragraph we have that $q_s\colon \mc{D}_s\to \mc{C}_s$ is smooth, and thus open. On the other hand, as $\mc{K}\to S$ has geometrically connected fibers, the same holds true for $f$ as these fibers are translations of geometric fibers of $\mc{K}\to S$. Thus, the map $q_s$ has geometrically connected fibers (see \stacks{055E}). Therefore we deduce that $\mc{D}_s$ is connected as desired from the simple fact that if $f\colon Y\to X$ be an open surjective map of topological spaces with connected fibers and $X$ is connected, then $Y$ is connected.
\end{proof}

\subsubsection{Some fiber products of parahoric group schemes} Suppose now that, as in \S\ref{ss:BT-theory}, $K$ is a discretely valued field with perfect residue field. Suppose further that we have a Cartesian diagram of group $\mc{O}_K$-schemes

\begin{equation*}
    \begin{tikzcd}
	{\mathcal{D}} & {\mathcal{A}} \\
	{\mathcal{C}} & {\mathcal{B},}
	\arrow["f", from=1-2, to=2-2]
	\arrow["g"', from=2-1, to=2-2]
	\arrow["q"', from=1-1, to=2-1]
	\arrow["p", from=1-1, to=1-2]
	\arrow["\lrcorner"{anchor=center, pos=0.125}, draw=none, from=1-1, to=2-2]
\end{tikzcd}
\end{equation*}
where $\mc{A}$, $\mc{B}$, and $\mc{C}$ are parahoric group $\mc{O}_K$-schemes. For simplicity we denote the generic fiber of $\mc{A}$, $\mc{B}$, $\mc{C}$, and $\mc{D}$, by $A,B,C,$ and $D$, respectively.

\begin{prop}\label{prop:fiber-product-parahoric} Suppose that $f$ is faithfully flat and quasi-compact, both $B$ and $C$ are tori, and $\ker(f)$ is a smooth group $\mc{O}_K$-scheme with connected fibers. Then, $\mc{D}$ is a parahoric group $\mc{O}_K$-scheme. Moreover, if $\mc{A}=\mc{G}_x^\circ$, then $\mc{D}=\mc{G}^\circ_{(x,\ast)}$ under the isomorphism
\begin{equation*}
    \ms{B}(D,K)\to \ms{B}(A\times_{\Spec(K)}C,K)\simeq \ms{B}(A,K)\times\{\ast\}
\end{equation*}
from Proposition \ref{lemma:buildings-ad-isom}.
\end{prop}
\begin{proof} By Proposition \ref{prop:fiber-product-properties}, we know that $\mc{D}$ is a smooth group $\mc{O}_K$-scheme with connected fibers. Moreover, observe that we have a natural short exact sequence of group $\mc{O}_K$-schemes
\begin{equation*}
    1\to \mc{D}\to \mc{A}\times_{\Spec(\mc{O}_K)}\mc{C}\to \mc{B}\to 1,
\end{equation*}
where the map $\mc{A}\times_{\Spec(\mc{O}_K)}\mc{C}\to \mc{B}$ is given by sending $(a,c)$ to $f(a)g(c)^{-1}$. This implies that $D$ is a normal subgroup of the reductive group $A\times_{\Spec(K)}C$ and so reductive (see \cite[Corollary 21.53]{Milne2017}).

On the other hand, this exact sequence also implies that ${D}\to {A}\times_{\Spec(K)}{C}$ induces an isomorphism on adjoint subgroups as ${B}$ is abelian. Moreover, observe that evidently
\begin{equation*}
    \mc{D}(\mc{O}_{K^\ur})=(\mc{A}(\mc{O}_{K^\ur})\times \mc{C}(\mc{O}_{K^\ur}))\cap D(K^\ur).
\end{equation*}
The claim then follows directly from Proposition \ref{prop:ad-isomorphism-parahoric}. 
\end{proof}

\subsubsection{Torsors for fiber products of groups}
We next are interested in understanding the relationship between torsors for a fiber product of groups, and torsors for their constituent groups. The correct generality for this discussion is that of topoi. We follow the notation and conventions concerning torsors and topoi as found in \cite[\S A.1]{ImaiKatoYoucis}.

Let us fix a topos $\mathscr{T}$ with final object $\ast$, and a Cartesian diagram of group objects
\begin{equation*}
    \begin{tikzcd}
	{\mathcal{D}} & {\mathcal{A}} \\
	{\mathcal{C}} & {\mathcal{B}.}
	\arrow["f", from=1-2, to=2-2]
	\arrow["g"', from=2-1, to=2-2]
	\arrow["q"', from=1-1, to=2-1]
	\arrow["p", from=1-1, to=1-2]
	\arrow["\lrcorner"{anchor=center, pos=0.125}, draw=none, from=1-1, to=2-2]
\end{tikzcd}
\end{equation*}
Consider the $2$-fiber product of groupoids\begin{equation*}
    \mathbf{Tors}(\mc{A})\times_{f_\ast,\mathbf{Tors}(\mc{B}),g_\ast}\mathbf{Tors}(\mc{C}),
\end{equation*}
(e.g., see \stacks{02X9}). Observe that there is a natural equivalence of functors
\begin{equation*}
    \theta_\mr{nat}\colon f_\ast\circ p_\ast\simeq (f\circ p)_\ast=(g\circ q)_\ast\simeq g_\ast\circ q_\ast,
\end{equation*}
and thus the triple $(p_\ast,q_\ast,\theta_\mr{nat})$ determines an object of the $2$-fiber product. Thus, 
\begin{equation}\label{eq:torsors-comparison}
    \Psi\colon \mathbf{Tors}_{\mc{D}}(\ms{T})\to \mathbf{Tors}_{\mc{A}}(\ms{T})\times_{f_\ast,\mathbf{Tors}_{\mc{B}}(\ms{T}),g_\ast}\mathbf{Tors}_{\mc{C}}(\ms{T}),
\end{equation}
given by $\Psi(\mc{P})=(p_\ast\mc{P},q_\ast\mc{P},\theta_\mr{nat})$ is a morphism of groupoids.

Conversely, suppose that $(\mc{Q}_\mc{A},\mc{Q}_\mc{C},\theta)$ is an object of $\mathbf{Tors}_{\mc{A}}(\ms{T})\times_{f_\ast,\mathbf{Tors}_{\mc{B}}(\ms{T}),g_\ast}\mathbf{Tors}_{\mc{C}}(\ms{T})$. Let us observe that we have a natural map
\begin{equation*}
    \mc{Q}_\mathcal{A}\to f_\ast\mathcal{Q}_\mathcal{A}=(\mathcal{B}\times\mathcal{Q}_\mathcal{A})/\mathcal{A},
\end{equation*}
induced by the map $(e,\mr{id})\colon \mc{Q}_\mc{A}\to \mc{B}\times\mc{Q}_\mc{A}$, where $e$ denotes the identity section of $\mc{B}$. We similarly have a map $\mc{Q}_\mc{C}\to g_\ast\mc{Q}_\mc{C}$. Let us define 
\begin{equation*}
    \mc{Q}_\mc{A}\times_\theta\mc{Q}_\mc{C}\defeq \lim\left(\begin{tikzcd}
	{\mathcal{Q}_\mathcal{A}} \\
	{f_\ast\mathcal{Q}_\mathcal{A}} & {g_\ast\mathcal{Q}_\mathcal{C}} & {\mathcal{Q}_\mathcal{C}}
	\arrow[from=1-1, to=2-1]
	\arrow["\theta", from=2-1, to=2-2]
	\arrow[from=2-3, to=2-2]
\end{tikzcd}\right),
\end{equation*}
where this limit is taken in $\ms{T}$. It is simple to check that there is a unique $\mc{D}$-action on $\mc{Q}_\mc{A}\times_\theta\mc{Q}_\mc{C}$ for which the natural map $\mc{Q}_\mc{A}\times_{\theta}\mc{Q}_\mc{C}\to \mc{Q}_\mc{A}\times\mc{Q}_\mc{C}$ is equivariant for the map $\mc{D}\to\mc{A}\times\mc{C}$. This $\mc{D}$-object of $\ms{T}$ is evidently functorial in the triple $(\mc{Q}_\mc{A},\mc{Q}_\mc{C},\theta)$. 

Note that it is not a priori clear that $\mc{Q}_\mc{A}\times_\theta\mc{Q}_\mc{C}$ is a $\mc{D}$-torsor. That said, we do have the following affirmative result in this direction under the assumption that $f$ is an epimorphism.
\begin{prop}\label{prop:torsors-comparison} Suppose that $f$ is an epimorphism. Then, the functor
\begin{equation*}
    \Phi\colon \mathbf{Tors}_{\mc{A}}(\ms{T})\times_{f_\ast,\mathbf{Tors}_{\mc{B}}(\ms{T}),g_\ast}\mathbf{Tors}_{\mc{C}}(\ms{T})\to \mathbf{Tors}_{\mc{D}}(\ms{T})
\end{equation*}
given by $\Phi(\mc{Q}_\mc{A},\mc{Q}_\mc{C},\theta)=\mc{Q}_\mc{A}\times_\theta\mc{Q}_\mc{C}$, is a well-defined quasi-inverse to $\Psi$. Moreover, the equivalences $\Psi$ and $\Phi$ are compatible with pullbacks along a morphism of topoi $\ms{T}'\to\ms{T}$.
\end{prop}
\begin{proof} To show that $\mc{Q}_\mc{A}\times_\theta\mc{Q}_\mc{C}$ is a $\mathcal{D}$-torsor we are free to localize $\ms{T}$. In particular, we may assume that $\mc{Q}_\mc{C}$ is trivializable. Choose a trivialization corresponding to an element $x$ of $\mc{Q}_\mc{C}(\ast)$. Consider then the induced element of $g_\ast\mc{Q}_\mc{C}(\ast)$, which we also denote by $x$, and the resulting element $\theta^{-1}(x)$ of $f_\ast\mc{Q}_\mc{A}(\ast)$. As $f$ is an epimorphism we may, after possibly localizing $\ms{T}$ further, assume that there exists some element $y$ of $\mc{Q}_\mc{A}(\ast)$ mapping to $\theta^{-1}(x)$. These compatible elements define an isomorphism of diagrams
\begin{equation*}
    \begin{tikzcd}
	{\mathcal{Q}_\mathcal{A}} \\
	{f_\ast\mathcal{Q}_\mathcal{A}} & {g_\ast\mathcal{Q}_\mathcal{C}} & {\mathcal{Q}_\mathcal{B}}
	\arrow[from=1-1, to=2-1]
	\arrow["\theta", from=2-1, to=2-2]
	\arrow[from=2-3, to=2-2]
\end{tikzcd}\,\,\simeq\,\,\,\,\, \begin{tikzcd}
	{\mathcal{A}} \\
	{\mathcal{B}} & {\mathcal{B}} & {\mathcal{C}.}
	\arrow["f"', from=1-1, to=2-1]
	\arrow["{\mathrm{id}}", from=2-1, to=2-2]
	\arrow["g", from=2-3, to=2-2]
\end{tikzcd}
\end{equation*}
As $\mc{A}\times_{\mr{id}}\mc{C}$ is evidently the trivial $\mc{D}$-torsor, the claim follows.

To prove that $\Phi\circ \Psi$ is naturally isomorphic to the identity, it suffices to prove the following claim: for a $\mc{D}$-torsor $\mc{P}$ the natural map
\begin{equation*}
    \mc{P}\to f_\ast\mc{P}\times g_\ast\mc{P},
\end{equation*}
is equivariant for the map $\mc{D}\to \mc{A}\times\mc{C}$, and factorizes uniquely through a $\mc{D}$-equivariant map $\mc{P}\to f_\ast\mc{P}\times_\theta g_\ast\mc{P}$. But this claim can be checked after localizing on $\ms{T}$, which reduces us to the trivial case as in the previous paragraph. 

To prove that $\Psi\circ \Phi$ is naturally isomorphic to the identity, let us fix an object $(\mc{Q}_\mc{A},\mc{Q}_\mc{C},\theta)$ of the $2$-fiber product. Let us observe that we have a natural projection map 
\begin{equation*}
    \mc{Q}_\mc{A}\times_\theta\mc{Q}_\mc{C}\to \mc{Q}_\mc{A}
\end{equation*}
which is equivariant for the map $\mc{D}\to \mc{A}$. By the adjunction in \cite[Chapitre III, Proposition 1.3.6 (iii)]{Giraud} we obtain a unique $\mc{A}$-equivariant map
\begin{equation*}
    \alpha\colon f_\ast(\mc{Q}_\mc{A}\times_\theta\mc{Q}_\mc{C})\to \mc{Q}_\mc{A},
\end{equation*}
which is necessarily an isomorphism of $\mc{A}$-torsors. Similarly, we obtain an isomorphism
\begin{equation*}
    \beta\colon g_\ast(\mc{Q}_\mc{A}\times_\theta\mc{Q}_\mc{C})\to \mc{Q}_\mc{C}.
\end{equation*}
It is then easy to check that $(\alpha,\beta)$ defines a natural morphism
\begin{equation*}
    \Psi(\Phi(\mc{Q}_\mc{A},\mc{Q}_\mc{C},\theta))\to (\mc{Q}_\mc{A},\mc{Q}_\mc{C},\theta),
\end{equation*}
where the matching of $\theta_\mr{nat}$ in the source and $\theta$ in the target may be checked locally on $\ms{T}$, from where we may reduce to the trivial case where it is clear. As the $2$-fiber product is a groupoid, we deduce that $(\alpha,\beta)$ is a natural isomorphism.

The final compatibility claim is clear by construction.
\end{proof}

\subsubsection{Shtukas for fiber products of groups} 

We finally record the natural implication of Proposition \ref{prop:torsors-comparison} for shtukas over a pre-adic space in the sense of \cite[\textsection 2.3]{PR2021}.

To this end, let us fix a fiber product of group $\Z_p$-schemes 
\begin{equation*}
    \begin{tikzcd}
	{\mathcal{D}} & {\mathcal{A}} \\
	{\mathcal{C}} & {\mathcal{B},}
	\arrow["f", from=1-2, to=2-2]
	\arrow["g"', from=2-1, to=2-2]
	\arrow["q"', from=1-1, to=2-1]
	\arrow["p", from=1-1, to=1-2]
	\arrow["\lrcorner"{anchor=center, pos=0.125}, draw=none, from=1-1, to=2-2]
\end{tikzcd}
\end{equation*}
where $\mc{A}$, $\mc{B}$, and $\mc{C}$ are smooth group $\Z_p$-schemes with connected fibers, $f$ is faithfully flat and quasi-compact, and $\ker(f)$ is a smooth group $\Z_p$-scheme with connected fibers. We then see by Proposition \ref{prop:fiber-product-properties} that $\mc{D}$ is a smooth group $\Z_p$-scheme with connected fibers. Further fix a discretely valued algebraic extension $E$ of $\Q_p$ with residue field $k_E$, 
and fix a pre-adic space $\ms{X}$ over $\mc{O}_{E}$. We then have the following result.

\begin{cor}\label{cor:fiber-product-of-shtukas} There is an equivalence of categories
\begin{equation*}
    \mc{D}\text{-}\mb{Sht}(\ms{X})\isomto \mc{A}\text{-}\mathbf{Sht}(\ms{X})\times_{f_\ast,\mc{B}\text{-}\mb{Sht}(\ms{X}),g_\ast}\mc{C}\text{-}\mb{Sht}(\ms{X}).
\end{equation*}

\end{cor}
\begin{proof} 
Fix an morphism $\alpha\colon S\to\ms{X}^{\lozenge /}$, where $S$ is an object of $\mb{Perf}_{k_E}$, with corresponding untilt $S^\sharp$ over $\mc{O}_{E}$. Then, it immediately by applying Proposition \ref{prop:torsors-comparison} to the \'etale topoi of $S\dot{\times}\Z_p$ and $S\dot{\times}\Z_p\setminus S^\sharp$, that there are equivalences 
\begin{equation*}
     \mc{D}\text{-}\mb{Sht}(S)\isomto \mc{A}\text{-}\mathbf{Sht}(S)\times_{f_\ast,\mc{B}\text{-}\mb{Sht}(S),g_\ast}\mc{C}\text{-}\mb{Sht}(S).
\end{equation*}
As this equivalence is $2$-functorial in the topos, the claim then follows by letting $S$ vary.
\end{proof}

\section{Integral models of Shimura varieties at parahoric level}\label{s:Kisin-Zhou-models}

In this section we study the Kisin--Pappas--Zhou models of Shimura varieties of abelian type at parahoric level as constructed in \cite{KP2018} and \cite{KZ21}. 

\subsection{Existence of models and their properties}

In this subsection we precisely state the existence of integral models at parahoric level and list some of their properties which are most relevant for our purposes.

\subsubsection{Shimura varieties} 

We begin by recalling some definitions and notation from the theory of Shimura varieties. Let $\gx$ be a Shimura datum, meaning that $\eG$ is a reductive group over $\bQ$, and $\mb{X}$ is a $\eG(\bR)$-conjugacy class of homomorphisms
\begin{align*}
    h: \mathbb{S}\to \eG_\bR,
\end{align*}
(where $\mathbb{S} = \Res_{\bC/\bR}\,\bG_{m/\bb{C}}$ is the Deligne torus), satisfying the axioms (SV1)-(SV3) in \cite[Definition 5.5]{Milne2005}. Associated to $\mb{X}$ is a unique conjugacy class $\mbbmu_h$ of coharacters $\bb{G}_{m,\bb{C}}\to \mb{G}_\bb{C}$ (see \cite[p.\@ 344]{Milne2005}).  We denote the field of definition of $\mbbmu_h$ by $\eE$, and call it the reflex field of $\gx$ (see \cite[Definition 12.2]{Milne2005}).  

We often fix a place $v$ of $\mb{E}$ dividing $p$, and write $E$ for the completion $\mb{E}_v$. Using \cite[Lemma 1.1.3]{KotShtw}, $\mbbmu_h$ gives rise to a unique conjugacy class of cocharacters $\bb{G}_{m,\ov{E}}\to G_{\ov{E}}$, which we also denote by $\mbbmu_h$. One may check that this conjugacy class has field of definition $E$, and so we call this the local reflex field.

For any neat compact open subgroup $\eK \subseteq \eG(\bA_f)$ (see \cite[p.\@ 288]{Milne2005}), we attach to $\gx$ and $\eK$ the Shimura variety $\Sh_\eK\gx$, which is a quasi-projective variety over $\bC$, whose $\bC$-points are given by the equality
\begin{align*}
    \Sh_\eK\gx(\bC) = \eG(\bQ)\bs \mathbf{X} \times \eG(\bA_f) / \eK. 
\end{align*}
We write an element of this double quotient as $[x,g]_\eK$. As in \cite{DeligneModulaire} (cf. \cite{Moonen1998}), $\Sh_\eK\gx$ admits a canonical model over the reflex field $\eE$. Hereafter we will use $\Sh_\eK\gx$ to denote the canonical model over $\eE$. 

For a containment $\eK \subseteq \eK'$ of neat compact open subgroups of $\eG(\bA_f)$ and $g$ in $\eG(\bA_f)$ such that $g^{-1}\eK g \subseteq \eK'$, there is a unique finite \'etale morphism of $\eE$-schemes 
\begin{align*}
    t_{\eK,\eK'}(g): \Sh_\eK\gx \to \Sh_{\eK'}\gx,
\end{align*}
which is given on $\bC$-points by the identity
\begin{align*}
    t_{\eK,\eK'}(g)[x,g']_\eK = [x,g'g]_{\eK'}. 
\end{align*}
We write $\pi_{\eK,\eK'}$ for $t_{\eK,\eK'}(\id)$, and $[g]_\eK$ for $t_{\eK, g^{-1}\eK g}(g)$. 

The morphisms $\pi_{\eK,\eK'}$ form a projective system $\{\Sh_\eK\gx\}_\eK$ with finite \'etale transition morphisms. We define
\begin{align}\label{eq:Sh(G,X)}
    \Sh\gx \defeq \varprojlim_\eK \Sh_\eK\gx,
\end{align}
where the limit is taken over all neat compact open subgroups $\eK$ of $\eG(\bA_f)$. Since each $\pi_{\eK,\eK'}$ is affine, the limit in (\ref{eq:Sh(G,X)}) exists as an $\eE$-scheme (see \stacks{01YX}). The morphisms $[g]_\eK$ endow $\Sh\gx$ with a continuous action of $\eG(\bA_f)$ in the sense of \cite[2.7.1]{DeligneModulaire}.

We will often fix a prime $p$ and consider neat compact open subgroups $\eK = \eK_p \eK^p$ with $\eK_p \subseteq \eG(\bQ_p)$ and $\eK^p \subseteq \eG(\bA_f^p)$. We define 
\begin{align*}
    \Sh_{\eK_p}\gx \defeq \varprojlim_{\eK^{p}} \Sh_{\eK_p\eK^{p}}\gx,
\end{align*}
where $\eK^{p}$ varies over neat compact open subgroups of $\eG(\bA_f^p)$. As in the case of $\Sh\gx$, we see that $\Sh_{\eK_p}\gx$ is a scheme over $\eE$ with a continuous action of $\eG(\bA_f^p)$. 

 \begin{rmk}
When several Shimura data are simultaneously under consideration, we will differentiate them with numerical subscripts (e.g.,\@ $(\mb{G}_1,\mb{X}_1)$) and use the same numerical subscripts to denote the objects defined above (or below) for this Shimura datum (e.g., $\Sh_{\mathsf{K}_{p,1}\mathsf{K}^p_1}(\mb{G}_1,\mb{X}_1)$). 
\end{rmk}

A morphism of Shimura data $\alpha\colon (\mb{G}_1,\mb{X}_1)\to (\mb{G},\mb{X})$ is a morphism of group $\Q$-schemes $\alpha\colon \mb{G}_1\to\mb{G}$ with the property that $\alpha_\bb{R}\circ h_1$ belongs to $\mb{X}$ for $h_1$ in $\mb{X}_1$. We say that $\alpha$ is an embedding if $\mb{G}_1\to\mb{G}$ is a closed embedding. From \cite[\S5]{DeligneModulaire}, for a morphism of Shimura data $\alpha$, one has that $\mathbf{E}\subseteq \mathbf{E}_1$ and there is a unique morphism $\Sh(\mb{G}_1,\mb{X}_1)\to \Sh(\mb{G},\mb{X})_{\mathbf{E}_1}$ of $\mb{E}_1$-schemes equivariant for the map $\alpha\colon \mb{G}_1(\A_f)\to \mb{G}(\A_f)$ and such that if $\alpha(\mathsf{K}_1)\subseteq \mathsf{K}$ then the induced map on the quotients $\alpha_{\mathsf{K}_1,\mathsf{K}}\colon \Sh_{\mathsf{K}_1}(\mb{G},\mb{X})\to \Sh_{\mathsf{K}}(\mb{G},\mb{X})_{E_1}$ is given on $\bb{C}$-points by 
\begin{equation*}
    \alpha_{\mathsf{K}_1,\mathsf{K}}\left([x_1,g_1]_{\mathsf{K}_1}\right)=[\alpha\circ x_1,\alpha(g_1)]_{\mathsf{K}}.
\end{equation*}

We say that $\alpha$ is an \emph{ad-isomorphism} if $\alpha\colon \mb{G}_1\to\mb{G}$ is an ad-isomorphism and moreover, that $\alpha\colon (\mb{G}_1^\mr{ad},\mb{X}_1^\mr{ad})\to (\mb{G}^\mr{ad},\mb{X}^\mr{ad})$ is an isomorphism.

\subsubsection{Parahoric Shimura data}\label{sss:parahoric-shim-data}

We say that a triple $(\mb{G},\mb{X},\mc{G})$ is a \emph{parahoric Shimura datum} if $(\mb{G},\mb{X})$ is a Shimura datum, and $\mc{G}$ is a parahoric model of $G\defeq \mb{G}_{\Q_p}$. By a morphism $\alpha\colon (\mb{G}_1,\mb{X}_1,\mc{G}_1)\to (\mb{G},\mb{X},\mc{G})$ of parahoric Shimura data we mean a morphism $\alpha\colon (\mb{G}_1,\mb{X}_1)\to (\mb{G},\mb{X})$ of Shimura data together with a specified model $\mc{G}_1\to\mc{G}$ of $G_1\to G$, which  we also denote $\alpha$. By \cite[Corollary 2.10.10]{KalethaPrasad}, such an $\alpha$ is unique, and it exists if and only if $\alpha\colon G(\breve{\Q}_p)\to G_1(\breve{\Q}_p)$ maps $\mc{G}_1(\breve{\Z}_p)$ into $\mc{G}(\breve{\Z}_p)$. We say that $\alpha$ is an embedding or an ad-isomorphism if $\alpha\colon (\mb{G}_1,\mb{X}_1)\to (\mb{G},\mb{X})$ is.

Given a parahoric Shimura datum $(\mb{G},\mb{X},\mc{G})$ we often use the following shorthand
\begin{equation*}
    \mathsf{K}_p\defeq \mc{G}(\Z_p).
\end{equation*}
Choosing a point $x$ of $\ms{B}(G,\Q_p)$ such that $\mc{G}=\mc{G}_x^\circ$, we write
\begin{equation*}
    \wt{\mathsf{K}}_p\defeq G(\Q_p^\ur)^1\cap \mr{Stab}_{G(\Q_p)}(x),
\end{equation*}
which is a a supergroup of $\mathsf{K}_p$ of finite index.

The following lemma will be important for later constructions. Let us denote by $\mb{E}^p$ the maximal extension of $\mb{E}$ unramified at places dividing $p$.

\begin{lemma}[{\cite[Corollary 4.3.9]{KP2018}}] \label{lemma:Defined-Over-E^p} The natural map 
\begin{equation*}
\pi_0(\Sh_{\mathsf{K}_p}(\mb{G},\mb{X})_{\ov{\mb{E}}})\to \pi_0(\Sh_{\mathsf{K}_p}(\mb{G},\mb{X})_{\mb{E}^p})
\end{equation*}
is a bijection.
\end{lemma}

In other words, this result says that the action of $\Gal(\ov{\mb{E}}/\mb{E})$ on $\pi_0(\Sh_{\mathsf{K}_p}(\mb{G},\mb{X})_{\ov{\mb{E}}})$ factorizes through $\Gal(\mb{E}^p/\mb{E})$ (cf.\@ \stacks{038D}).

\subsubsection{Some conditions on a Shimura datum} The construction of Kisin--Pappas--Zhou models involves several conditions on a Shimura datum that we presently recall. 

We start with the following standard definitions, in increasing level of generality.

\begin{defn} We say that $(\mb{G},\mb{X})$ is 
\begin{itemize}
    \item of \emph{Siegel type} if it is of the form $(\mr{GSp}(\mb{V}),\mf{h}^{\pm})$ where $\mb{V}$ is a symplectic $\Q$-space, and $\mf{h}^{\pm}$ is the union of the upper and lower Siegel half-spaces (e.g., see \cite[\S6]{Milne2005}),
    \item of \emph{Hodge type} if there exists an embedding of Shimura data $\iota\colon (\mb{G},\mb{X})\to (\mr{GSp}(\mb{V}),\mf{h}^{\pm})$ (called a \emph{Hodge embedding}),
    \item of \emph{abelian type} if there exists a Shimura datum $(\mb{G}_1,\mb{X}_1)$ of Hodge type and an isogeny $\alpha\colon\mb{G}_1^\der\to\mb{G}^\der$ inducing an isomorphism $(\mb{G}_1^\ad,\mb{X}^\ad)\to (\mb{G}^\ad,\mb{X}^\ad)$, in which case we say that $(\mb{G}_1,\mb{X}_1)$ is \emph{adapted} to $(\mb{G},\mb{X})$, leaving the $\alpha$ implicit.
\end{itemize}
\end{defn}

We have the following examples of Shimura data of abelian type.\footnote{An essentially full classification of Shimura varieties of abelian type is given in the appendix of \cite{MilneShih}.}

\begin{ex}
Shimura data of PEL type (see \cite[\S8]{Milne2005}) are of Hodge type.
\end{ex}

\begin{ex}\label{gspin} Let $\mb{V}$ be a quadratic space over $\Q$ of signature $(n,2)$. The group $\mathrm{GSpin}(\mb{V})$ acts transitively on the space $\mb{X}$ of oriented negative definite $2$-planes in $\mb{V}_\bb{R}$, and $\mb{X}$ can be identified with a $\mathrm{GSpin}(\mb{V})(\bb{R})$-conjugacy class of morphisms $\bb{S}\to \mathrm{GSpin}(\mb{V})_\bb{R}$ (see \cite[\S 1]{PeraSpin}). The pair $(\mathrm{GSpin}(\mb{V}),\mb{X})$ is a Shimura datum of Hodge type which is not of PEL type (see \cite[\S3]{PeraSpin}).
\end{ex}

\begin{ex} Let $F\supsetneq \Q$ be a totally real field, $B$ a quaternion algebra over $F$, and $\mb{G}_B$ the algebraic $\Q$-group $B^\times$. There is a Shimura datum $(\mb{G}_B,\mb{X}_B)$ associated to $B$ (see \cite[Example 5.24]{Milne2005}). If $B$ is not $\bb{R}$-split then $(\mb{G}_B,\mb{X}_B)$ is of abelian type, but not of Hodge type. The Shimura varieties associated to such $(\mb{G}_B,\mb{X}_B)$ include Shimura curves.
\end{ex}

\begin{ex}\label{ex:toral-type} For a $\Q$-torus $\mb{T}$ any homomorphism $h\colon \bb{S}\to\mb{T}_\bb{R}$ defines a Shimura datum $(\mb{T},\{h\})$ of abelian type, which is rarely of Hodge type, and which we refer to as being of \emph{toral type}.
\end{ex}

Kisin--Pappas--Zhou models are constructed when $(\mb{G},\mb{X})$ is of abelian type and when the following group-theoretic property holds.

\begin{defn}[{\cite[Definition 3.3.2]{KZ21}}]We say that $(\mb{G},\mb{X})$ is \emph{acceptable} if $G^\mr{ad}$ is isomorphic to a group of the form $\prod_{i=1}^r\mathrm{Res}_{F_i/\Q_p}\, H_i$, where each $F_i$ is a finite extension of $\Q_p$, and $H_i$ is an adjoint group over $F_i$ which is split over a tame extension of $F_i$. 
\end{defn}

\begin{rmk}\label{rmk:acceptable}
The acceptability condition is very mild. For instance, if $p>3$ the condition is vacuous, and when $p=3$ it only excludes those Shimura data $(\mb{G}, \mb{X}$) for which $\mb{G}$ has an almost simple factor of type $D_4$ (see \cite[\S3.3.1]{KZ21}).
\end{rmk}

The construction of Kisin--Pappas--Zhou models will be an iterative process, starting with a much more restrictive class of Shimura data. For convenience, we give name to this class.

\begin{defn} We say that $(\mb{G},\mb{X})$ is \emph{very good} if:
\begin{itemize}
    \item $(\mb{G},\mb{X})$ is acceptable,

    \item $Z(\mb{G})$ is a torus, and
    \item $Z(G)$, $G^\der$, and $G$ are $R$-smooth (see Definitions \ref{defn:R-smooth} and \ref{defn:R-smooth-G}).

\end{itemize}
\end{defn}

We say that a parahoric Shimura datum $(\mb{G},\mb{X},\mc{G})$ is of Hodge type, abelian type, acceptable, or very good if the underlying Shimura datum $(\mb{G},\mb{X})$ is.
\begin{defn}A parahoric Shimura datum $(\mb{G}_1,\mb{X}_1,\mc{G}_1)$ of Hodge type is \emph{well-adapted} to $(\mb{G},\mb{X},\mc{G})$ when given an isogeny $\alpha\colon\mb{G}_1^\der\to\mb{G}^\der$ (often left implicit) such that
\begin{enumerate}
    \item $\alpha$ adapts $(\mb{G}_1,\mb{X}_1)$ to $(\mb{G},\mb{X})$,
    \item if the isomorphism
\begin{equation*}
    \ms{B}(G_1,\Q_p)\isomto \ms{B}(G,\Q_p)
\end{equation*}
induced from the isogeny $\alpha\colon G_1^\der\to G^\der$ and Lemma \ref{lemma:buildings-ad-isom} has the property that if $\mc{G}\simeq \mc{G}_x^\circ$ for $x$, and $x_1$ denotes the corresponding point in $\ms{B}(G_1,\Q_p)$, then $\mc{G}_1\simeq\mc{G}_{x_1}^\circ$,
\item if $\mb{E}'\defeq \mb{E}\mb{E}_1$, then $\mb{E}'$ splits completely at every prime of $\mb{E}$ lying over $p$.
\end{enumerate} 
\end{defn}

Finally, we set some terminology used later on in the article. Let $(\mb{G},\mb{X},\mc{G})$ be a parahoric Shimura datum of Hodge type. We call a map of Shimura data
\begin{equation*} 
\iota\colon (\mb{G},\mb{X},\mc{G})\to (\mr{GSp}(\mb{V}),\mf{h}^\pm,\mr{GSp}(\Lambda)),
\end{equation*}
where $\mb{V}$ a symplectic $\Q$-space and $\Lambda\subseteq \mb{V}_{\Q_p}$ is a self-dual $\Z_p$-lattice, a \emph{Hodge embedding of parahoric Shimura data} if the underling map $\iota\colon \mb{G}\to\mr{GSp}(\mb{V})$ is a closed embedding. We say that $\iota$ \emph{respects stabilizers} if $\wt{\mathsf{K}}_p=\mr{GSp}(\Lambda)\cap \mb{G}(\Q_p)$.

\subsubsection{Existence and properties of models}

We are now prepared to state the existence of Kisin--Pappas--Zhou models, as well as the major properties we need of  these models. 

\begin{thm}[{\cite[Theorem 4.6.23]{KP2018} and \cite[Theorem 5.2.12]{KZ21}}]\label{thm:KPZ-model-properties} Suppose that $(\mb{G},\mb{X},\mc{G})$ is an acceptable parahoric Shimura datum of abelian type and. Then, there exists a system $\{\ms{S}^\mf{d}_{\mathsf{K}_p\mathsf{K}^p}(\mb{G},\mb{X})\}$ of normal flat quasi-projective $\mc{O}_E$-models of $\{\Sh_{\mathsf{K}_p\mathsf{K}^p}(\mb{G},\mb{X})_E\}$ such that the following properties hold.
\begin{enumerate}
    \item Fix neat compact open subgroups $\mathsf{K}^p$ and ${\mathsf{K}^{p}}'$ in $\mb{G}(\A_f^p)$ and an element $g$ of $\mb{G}(\A_f^p)$ with $g^{-1}\eK^p g\subseteq {\mathsf{K}^{p}}'$. Write $\mathsf{K}=\mathsf{K}_p\mathsf{K}^p$ and $\mathsf{K}'=\mathsf{K}_p{\mathsf{K}^p}'$. Then, the morphism
    \begin{equation*}
    t_{\mathsf{K},\mathsf{K}'}(g): \Sh_\eK\gx_E\to \Sh_{\eK'}\gx_E,
    \end{equation*}
    admits a (necessarily unique) finite \'etale extension
    \begin{equation*}
    t_{\mathsf{K},\mathsf{K}'}(g): \ms{S}_{\mathsf{K}}^\mf{d}(\mb{G},\mb{X})\to \ms{S}_{\mathsf{K}'}^\mf{d}(\mb{G},\mb{X}).
    \end{equation*}
    \item Writing 
    \begin{equation*}\ms{S}_{\mathsf{K}_p}^\mf{d}(\mb{G},\mb{X})\defeq \varprojlim_{\mathsf{K}^p}\ms{S}_{\mathsf{K}_p\mathsf{K}^p}^\mf{d}(\mb{G},\mb{X}),
    \end{equation*}
    one has the following \emph{extension property} relative to a characteristic $(0,p)$ discrete valuation ring $R$ over $\mc{O}_E$: the natural map
    \begin{equation*}
       \ms{S}_{\mathsf{K}_p}^\mf{d}(\mb{G},\mb{X})(R)\to \Sh_{\mathsf{K}_p}(\mb{G},\mb{X})(R[\nicefrac{1}{p}]),
    \end{equation*}
    is a bijection.
\end{enumerate}
\end{thm}

We will recall below, fairly precisely, the construction of the models $\ms{S}_{\mathsf{K}}^\mf{d}(\mb{G},\mb{X})$. The superscript $\mf{d}$, which is either a very good parahoric Shimura datum of Hodge type well-adapted to $(\mb{G},\mb{X},\mc{G})$, or possibly $\varnothing$ when $(\mb{G},\mb{X},\mc{G})$ is of Hodge type. Its role is to emphasize various choices made in the construction of these models.

\subsection{The $\ms{A}$-group and the $\ms{E}$-group}\label{ss:Shim-vars-and-A-group}

Let $\eG$ be a reductive $\bQ$-group. In this section we associate to $\eG$ groups $\ms{A}(\eG)$ and $\ms{A}(\eG)^\circ$, and the so-called $\ms{E}$-group to a parahoric Shimura datum, following \cite{DeligneModulaire}. We also define analogs of these groups for $\bZ_{(p)}$-valued points following \cite{KP2018} and \cite{KZ21}, and discuss some functorial properties of these groups. These groups play an important role in the construction of the models $\ms{S}^\mf{d}_{\mathsf{K}_p\mathsf{K}^p}(\mb{G},\mb{X})$ from Theorem \ref{thm:KPZ-model-properties}.

\subsubsection{The $\ast$-product} Let us begin by recalling the $\ast$-product of Deligne (see \cite{DeligneModulaire}).

Suppose $H$ and $\Gamma$ are abstract groups equipped with actions of a group $\Delta$ by group homomorphisms. Let 
\begin{align*}
    \psi: \Gamma \to H,\quad \text{and}\quad\varphi: \Gamma \to \Delta
\end{align*}
be $\Delta$-equivariant group homomorphisms, where $\Delta$ acts on itself by inner automorphisms, and assume that $\psi$ and $\varphi$ are compatible in the sense that, for all $\gamma$ in $\Gamma$, we have the identity
\begin{align*}
    \varphi(\gamma) \cdot h = \psi(\gamma)^{-1}h\psi(\gamma),
\end{align*}
where $\cdot$ denotes the action of $\Delta$ on $H$. Then the semi-direct product $H \rtimes \Delta$ admits a quotient 
\begin{align*}
    H \ast_\Gamma \Delta \defeq H \rtimes \Delta / \{(\psi(\gamma), \varphi(\gamma)^{-1})\},
\end{align*}
where one easily checks that the subgroup $\{(\psi(\gamma), \varphi(\gamma)^{-1})\}$ of all the elements of the form $(\psi(\gamma), \varphi(\gamma)^{-1})$ for $\gamma$ in $\Gamma$ is a normal subgroup.

For the purposes of this section, we will refer to a datum $(H, \Gamma, \Delta, \psi, \varphi)$ as above as a \textit{tuple}. A morphism of tuples $(H, \Gamma, \Delta, \psi, \varphi) \to (H', \Gamma', \Delta', \psi', \varphi')$ is a triple $(f,g,h)$, consisting of group homomorphisms $f: H\to H'$ and $g: \Gamma \to \Gamma'$ equivariant relative to a group homomorphism $h: \Delta \to \Delta'$ and such that $\psi' \circ g = f \circ \psi$ and $h\circ \varphi = \varphi' \circ f$. 

The $\ast$-product satisfies the following elementary functoriality. 

\begin{lemma} \label{lemma:ast-product-functoriality}
    Any morphism of tuples 
    \begin{align*}
        (f,g,h)\colon (H, \Gamma, \Delta, \varphi) \to (H', \Gamma', \Delta', \varphi')
    \end{align*}
    induces a homomorphism of groups
    \begin{align*}
        H \ast_\Gamma \Delta \to H' \ast_{\Gamma'} \Delta'.
    \end{align*}
\end{lemma}

For future use, we also remark that if $X$ is a set equipped with a right action of $H \rtimes \Delta$, then the action descends to the an action of $H \ast_\Gamma \Delta$ on the quotient $X / \Gamma$, where here $\Gamma$ acts on $X$ via
\begin{equation}\label{eq:action}
    x \cdot \gamma = x \cdot (\psi(\gamma), \varphi(\gamma)^{-1}).
\end{equation}

\subsubsection{The groups $\ms{A}(\eG)$ and $\ms{A}(\eG)^\circ$} Let us now recall some notation from \cite{DeligneModulaire}.

If $\eG$ is a reductive group over $\bQ$, then we denote by $\eG(\bR)^+$ the connected component of the identity in $\eG(\bR)$ with respect to the real topology. For a subgroup $\eH$ of $\eG(\bR)$, we denote by $\eH_+$ the inverse image of $\eG^\ad(\bR)^+$ in $\eH$. We write also $\eG^\ad(\bQ)^+ = \eG^\ad(\bQ) \cap \eG^\ad(\bR)^+$.

Now suppose $\gx$ is a Shimura datum. Let $\eZ$ denote the center of $\eG$, and let $\eZ(\bQ)^-$ denote the closure of $\eZ(\bQ)$ in $\eG(\bA_f)$. We define 
\begin{align*}
    \ms{A}(\eG) \defeq \eG(\bA_f) / \eZ(\bQ)^- \ast_{\eG(\bQ)_+ / \eZ(\bQ)} \eG^\ad(\bQ)^+,
\end{align*}
where $\eG^\ad(\bQ)^+$ acts by conjugation on $\eG(\bA_f) / \eZ(\bQ)^-$ and  $\eG(\bQ)_+ / \eZ(\bQ)\to \eG(\bA_f) / \eZ(\bQ)^-$ and $\eG(\bQ)_+ / \eZ(\bQ)\to \eG^\ad(\bQ)^+$ are the obvious maps.

By \cite[Proposition 5.1]{Milne2005}, the map $\eG(\bR)^+ \to \eG^\ad(\bR)^+$ is surjective with kernel $\eZ(\bR) \cap \eG(\bR)^+$, which is contained in the center of $\eG(\bR)^+$, so the conjugation action of $\eG$ on itself induces an action of $\eG^\ad(\bQ)^+$ on $\Sh\gx$. Combining this with the action of $\eG(\bA_f)$ on $\Sh\gx$ determines a right action of $\ms{A}(\eG)$ on $\Sh\gx$. 

Denote by $\eG(\Q)_+^-$ the closure of $\eG(\bQ)_+$ in $\eG(\bA_f)$. Define
\begin{align*}
    \ms{A}(\eG)^\circ = \eG(\bQ)_+^- / \eZ(\bQ)^- \ast_{\eG(\bQ)_+/\eZ(\bQ)} \eG^\ad(\bQ)^+,
\end{align*}
where $\eG^\ad(\bQ)^+$ acts by conjugation on $\eG(\bQ)_+^- / \eZ(\bQ)^-$ and $\eG(\bQ)_+ / \eZ(\bQ)\to  \eG(\bQ)_+^- / \eZ(\bQ)^- $ and $\eG(\bQ)_+ / \eZ(\bQ)\to \eG^\ad(\bQ)^+$ are the obvious maps. Evidently $ \ms{A}(\eG)^\circ$ is a subgroup of $\ms{A}(\eG)$. But, the group $\ms{A}(\eG)^\circ$ depends only on $\eG^\der$ and not on $\eG$, since by \cite[2.1.15]{DeligneModulaire} it is given by the completion of $\eG^\ad(\bQ)^+$ with respect to the topology whose open sets are images of congruence subgroups in $\eG^\der(\bQ)$.

It is clear that both the associations $\mb{G}\mapsto\ms{A}(\mb{G})$ and $\mb{G}\mapsto\ms{A}(\mb{G})^\circ$ are functorial in maps of reductive groups $f\colon \mb{G}\to \mb{H}$ such that $f(Z(\mb{G}))\subseteq Z(\mb{H})$.

\subsubsection{The groups $\ms{A}(\bm{\cG})$ and $\ms{A}(\bm{\cG})^\circ$}
We define also analogs of $\ms{A}(\eG)$ and $\ms{A}(\eG)^\circ$ for $\bZ_{(p)}$-valued points, following \cite[\textsection 4.5.6]{KP2018} (see also \cite[\textsection 5.2.2]{KZ21}).

To begin, we record the following well-known Beauville--Lazslo-type lemma (which is a largely basic case of \stacks{0F9Q}). For a ring $A$, let us denote by $\mb{AlgGrp}_A=\mb{AlgGrp}_A^\vn$ the category of finite type affine group $A$-schemes, and by $\mb{AlgGrp}^\mr{fl}_A$ and $\mb{AlgGrp}^\mr{sm}_A$ the full subcategory of $A$-flat and $A$-smooth objects, respectively.

\begin{prop}\label{prop:BL-for-al-grps} Let $R$ be a Noetherian ring and $f$ a non-zerodivisor of $R$. Denote the $f$-adic completion of $R$ by $\wh{R}$. Then the functor
\begin{equation*}
    \mb{AlgGrp}^P_R\to \mb{AlgGrp}^P_{R[\nicefrac{1}{f}]}\times_{\mb{AlgGrp}^P_{\wh{R}[\nicefrac{1}{f}]}}\mb{AlgGrp}^P_{\wh{R}},\quad \mc{H}\mapsto (\mc{H}_{R[\nicefrac{1}{f}]},\mc{H}_{\wh{R}},\theta_\mr{nat}),
\end{equation*}
is an equivalence, where the target is the $2$-fiber product $($with the projection maps the obvious base change$)$, $\theta_\mr{nat}$ is the natural isomorphism, and $P\in\{\vn,\mr{fl},\mr{sm}\}$.
\end{prop} 
\begin{proof} Observe that $\Spec(\wh{R})\to\Spec(R)$ is flat (e.g., see \stacks{00MB}) with image precisely the vanishing locus $V(f)$ of $f$ (e.g., as follows from combining \stacks{00MC} and \stacks{00HP}). Thus, 
\begin{equation*}
    \Spec(\wh{R})\sqcup \Spec(R[\nicefrac{1}{f}])\to \Spec(R)
\end{equation*}
is an fpqc cover. Moreover, 
\begin{equation*}
    \Spec(\wh{R})\times_{\Spec(R)}\Spec(R[\nicefrac{1}{f}])\simeq \Spec(\wh{R}[\nicefrac{1}{f}]).
\end{equation*}
The fully faithfulness follows from \stacks{023Q}. The essential surjectivity follows from \stacks{0245} to descend to an affine scheme, \stacks{023Q} to descend the multiplication map and identity section, and finally \stacks{02KZ}, \stacks{02L2}, and \stacks{02VL} to descend the property $P$.
\end{proof}
 
Suppose now that that $\mb{G}$ is a reductive group over $\Q$ and $\mc{G}$ is a parahoric model of $\mb{G}_{\Q_p}$. Let $\cG^\ad$ denote the parahoric $\bZ_p$-model of $G^\ad$ as discussed in \S\ref{sss:parahoric-central-quotient}. By Proposition \ref{prop:BL-for-al-grps} there exists a unique morphism of smooth group $\Z_{(p)}$-schemes $\bm{\mc{G}}\to \bm{\mc{G}}^\mr{ad}$ modeling $\mb{G}\to\mb{G}^\mr{ad}$ and $\mc{G}\to\mc{G}^\mr{ad}$. Let $\eZ_\eG$ denote the center of $\eG$, and denote by $\bm{\mc{Z}}_{\eG}$ the closure of $\eZ_\eG$ in $\bm{\cG}$. 

Let $\bm{\mc{Z}}_{\eG}(\bZ_{(p)})^-$ and $\bm{\cG}(\bZ_{(p)})_+^-$ denote the closures of $\bm{\mc{Z}}_{\eG}(\bZ_{(p)})$ and $\bm{\cG}(\bZ_{(p)})_+$ in $\eG(\bA_f^p)$, respectively. Following \cite[\textsection 4.5.6]{KP2018} and \cite[\textsection 5.2.2]{KZ21}, we set
\begin{align*}
    \ms{A}(\bm{\cG}) \defeq \eG(\bA_f^p)/\bm{\mc{Z}}_\eG(\bZ_{(p)})^- \ast_{\bm{\cG}(\bZ_{(p)})_+ / \bm{\mc{Z}}_\eG(\bZ_{(p)})} \bm{\cG}^\ad(\bZ_{(p)})^+,
\end{align*}
and 
\begin{align*}
    \ms{A}(\bm{\cG})^\circ \defeq \bm{\cG}(\bZ_{(p)})_+^-/\bm{\mc{Z}}_\eG(\bZ_{(p)})^- \ast_{\bm{\cG}(\bZ_{(p)})_+ / \bm{\mc{Z}}_\eG(\bZ_{(p)})} \bm{\cG}^\ad(\bZ_{(p)})^+,
\end{align*}
where in the first definition we have that $\bm{\cG}^\ad(\bZ_{(p)})^+$ acts on $\eG(\bA_f^p)/\bm{\mc{Z}}_\eG(\bZ_{(p)})^-$ by conjugation and 
\begin{equation*}
    \bm{\cG}(\bZ_{(p)})_+ / \bm{\mc{Z}}_\eG(\bZ_{(p)})\to \eG(\bA_f^p)/\bm{\mc{Z}}_\eG(\bZ_{(p)})^-\quad\text{and}\quad \bm{\cG}(\bZ_{(p)})_+ / \bm{\mc{Z}}_\eG(\bZ_{(p)})\to \bm{\cG}^\ad(\bZ_{(p)})^+,
\end{equation*} 
are the obvious maps, and similarly for the second definition.

By \cite[Lemma 4.6.4 (2)]{KP2018}, $\ms{A}(\bm{\mc{G}})^\circ$ is a subgroup of $\ms{A}(\bm{\mc{G}})$ but depends only on $\bm{\mc{G}}^\der$, where $\mc{G}^\der$ denotes the parahoric model of $\mb{G}_{\Q_p}^\der$ obtained as in \S\ref{sss:parahorics-and-abelianization} and $\bm{\mc{G}}^\der$ is obtained from Proposition \ref{prop:BL-for-al-grps}.

\begin{lemma}\label{lem:functoriality-of-A}
    Let $\eG$ and $\mb{H}$ be reductive groups over $\bQ$ and set $G=\mb{G}_{\Q_p}$ and $H=\mb{H}_{\Q_p}$. Fix parahoric models $\mc{G}$ and $\mc{H}$ of $G$ and $H$ corresponding to points $x$ and $y$ in the buildings $\ms{B}(G,\bQ_p)$ and $\ms{B}(H,\bQ_p)$, respectively. Suppose $f: \eG \to \mb{H}$ is a homomorphism of reductive $\bQ$-groups carrying $\eZ_\eG$ into $\eZ_{\mb{H}}$, and suppose that
    \begin{align*}
        f_\ast: \ms{B}(G,\bQ_p^\ur) \to \ms{B}(H, \bQ_p^\ur)
    \end{align*}
    is a map of buildings which is equivariant for the map $f: G(\bQ_p^\ur) \to H(\bQ_p^\ur)$ and which sends $x$ to $y$. Then $f$ induces group homomorphisms
    \begin{align*}
        \ms{A}(f)\colon \ms{A}(\bm{\cG}) \to \ms{A}(\bm{\mc{H}}) \quad \text{and}\quad \ms{A}(f)^\circ\colon \ms{A}(\bm{\mc{G}})^\circ \to \ms{A}(\bm{\mc{H}})^\circ.
    \end{align*}
\end{lemma}

\begin{proof}
    Since $f$ carries $\eZ_\eG$ into $\eZ_\mb{H}$, it follows that $f$ induces a homomorphism of groups
    \begin{align*}
        f\colon \eG(\bA_f^p)/ \bm{\mc{Z}}_\eG(\bZ_{(p)})^- \to \mb{H}(\bA_f^p)/ \bm{\mc{Z}}_{\mb{H}}(\bZ_{(p)})^-.
    \end{align*}
     The equivariance of $f_\ast$ implies that $f(\operatorname{Stab}_{G(\bQ_p^\ur)}(x)) \subseteq \operatorname{Stab}_{H(\bQ_p^\ur)}(y)$, and hence that $f(\mc{G}(\bZ_p^\ur))\subseteq \mc{H}(\bZ_p^\ur)$. Thus $f$ extends to a morphism $f: \mc{G} \to \mc{H}$ by \cite[Corollary 2.10.10]{KalethaPrasad}. We similarly obtain an extension $f^\ad: \cG^\ad \to \mc{H}^\ad$ of $f^\ad:G^\ad \to H^\ad$, which induces 
    \begin{align*}
        f^\ad: \bm{\mc{G}}^\ad(\bZ_{(p)})^+ \to \bm{\mc{H}}^\ad(\bZ_{(p)})^+.
    \end{align*}
    Moreover, as $f$ extends to a morphism $\mc{G} \to \mc{H}$, it follows that $f$ restricts to homomorphisms
    \begin{align*}
        f' \colon \bm{\mc{G}}(\bZ_{(p)})_+/\bm{\mc{Z}}_\mb{G}(\bZ_{(p)}) \to \bm{\mc{H}}(\bZ_{(p)})_+/\bm{\mc{Z}}_\mb{H}(\bZ_{(p)}),
    \end{align*}
    and
    \begin{align*}
        f^\circ \colon \bm{\mc{G}}(\bZ_{(p)})_+^-/\bm{\mc{Z}}_\mb{G}(\bZ_{(p)})^- \to \bm{\mc{H}}(\bZ_{(p)})_+^-/\bm{\mc{Z}}_\mb{H}(\bZ_{(p)})^-.
    \end{align*}
    By Lemma \ref{lemma:ast-product-functoriality}, it remains only to check that $(f, f', f^\ad)$ and $(f^\circ, f', f^\ad)$ induce morphisms of tuples. By uniqueness of the extension of $G \to H \to H^\ad$ to parahoric models and equivariance of the bijections $\ms{B}(G,\bQ_p^\ur) \isomto \ms{B}(G^\ad, \bQ_p^\ur)$ and $\ms{B}(H, \bQ_p^\ur) \isomto \ms{B}(H^\ad, \bQ_p^\ur)$ from Lemma \ref{lemma:buildings-ad-isom}, we see that the diagram  
    
    \begin{equation*}
        \begin{tikzcd}
        \bm{\mc{G}}(\bZ_{(p)})_+/\bm{\mc{Z}}_\mb{G}(\bZ_{(p)})
            \arrow[r, "f"] \arrow[d]
        & \bm{\mc{H}}(\bZ_{(p)})_+/\bm{\mc{Z}}_\mb{H}(\bZ_{(p)})
            \arrow[d]
        \\ \bm{\mc{G}}^\ad(\bZ_{(p)})^+
            \arrow[r, "f^\ad"]
        & \bm{\mc{H}}^\ad(\bZ_{(p)})^+
        \end{tikzcd}
    \end{equation*}
    commutes. 
    Thus it remains only to show that $f$ is equivariant for the action of $\bm{\mc{G}}^\ad(\bZ_{(p)})^+$. But this is clear since the action of $\bm{\mc{G}}^\ad(\bZ_{(p)})^+ \subseteq \mb{G}(\bQ)$ on $\mb{G}(\bA_f^p)/\bm{\mc{Z}}_\mb{G}(\bZ_{(p)})^-$ is via conjugation, and $f\colon \eG(\bA_f^p)/ \bm{\mc{Z}}_\eG(\bZ_{(p)})^- \to \mb{H}(\bA_f^p)/ \bm{\mc{Z}}_{\mb{H}}(\bZ_{(p)})^-$ is a group homomorphism.
\end{proof}

\subsubsection{The group $\ms{E}(\bm{\mc{G}})$}

Suppose now that $(\mb{G},\mb{X}, \mc{G})$ is a parahoric Shimura datum with reflex field $\mb{E}$. Let $\eK_p = \mc{G}(\bZ_p)$. Fix a connected component $\mb{X}^+$ of $\mb{X}$, which determines connected Shimura varieties $\Sh\gx^+ \subseteq \Sh\gx$ and $\Sh_{\eK_p}\gx^+ \subseteq \Sh_{\eK_p}\gx$. 

By Lemma \ref{lemma:Defined-Over-E^p}, the action of $\Gal(\ov{\eE}/ \eE)$ on $\Sh_{\eK_p}\gx^+$ factors through $\Gal(\eE^p/ \eE)$. By abuse of notation, we will use $\Sh_{\eK_p}\gx^+$ to refer to the $\eE^p$-scheme obtained by descent.

Let $\ms{E}(\bm{\mc{G}}) \subseteq \ms{A}(\bm{\mc{G}}) \times \Gal(\eE^p / \eE)$ denote the stabilizer of $\Sh_{\eK_p}\gx^+$ in $\Sh_{\eK_p}\gx$. By \cite[Lemma 4.6.6]{KP2018}, $\ms{E}(\bm{\mc{G}})$ is an extension of $\Gal(\eE^p/\eE)$ by $\ms{A}(\bm{\mc{G}})^\circ$, and there is a canonical isomorphism
\begin{equation*}
    \ms{A}(\bm{\mc{G}})\ast_{\ms{A}(\bm{\mc{G}})^\circ} \ms{E}(\bm{\mc{G}}) \isomto \ms{A}(\bm{\mc{G}}) \times \Gal(\eE^p / \eE),
\end{equation*}
where an element of $\ms{E}(\bm{\mc{G}})$ acts on $\ms{A}(\bm{\mc{G}})$ via conjugation by its image in $\ms{A}(\bm{\mc{G}})$.\footnote{Explicitly, the isomorphism  $\ms{A}(\bm{\mc{G}}) \times \Gal(\mb{E}^p / \mb{E}) \to \ms{A}(\bm{\mc{G}}) \ast_{\ms{A}(\bm{\mc{G}})^\circ} \ms{E}(\bm{\mc{G}})$ is given by the association $(a, \sigma) \mapsto (a \bar{e}^{-1}, e)$, where $e$ denotes a lift of $\sigma$ to $\ms{E}(\bm{\mc{G}})$, and $\bar{e}$ denotes the image of $e$ in $\ms{A}(\bm{\mc{G}})$.}

The $\ms{E}$-group satisfies a functoriality akin to that of the $\ms{A}$-groups. 
\begin{lemma}\label{lem:functoriality-of-E}
    Let $f: (\mb{G}_1, \mb{X}_1,\mc{G}_1) \to (\mb{G}, \mb{X}, \mc{G})$ be a morphism of parahoric Shimura data. Suppose $\mc{G}_1$ and $\mc{G}$ correspond to points $x_1$ and $x$ in the buildings $\ms{B}(G_1, \bQ_p)$ and $\ms{B}(G,\bQ_p)$, respectively. Suppose the induced $f: \mb{G}_1 \to \mb{G}$ carries $\mb{Z}_\mb{G}$ into $\mb{Z}_\mb{H}$, and suppose that
    \begin{align*}
        f_\ast\colon \ms{B}(G_1, \bQ_p) \to \ms{B}(G,\bQ_p)
    \end{align*}
    is a map of buildings which is equivariant for the map $f\colon G_1(\bQ_p^\ur) \to G(\bQ_p^\ur)$ and which sends $x_1$ to $x$. Then $f$ induces a group homomorphism
    \begin{align*}
        \ms{E}(f): \ms{E}(\bm{\mc{G}}_1) \to \ms{E}(\bm{\mc{G}}).
    \end{align*}
\end{lemma}
\begin{proof}
    This follows from the definition of $\ms{E}(\bm{\mc{G}})$ by Lemma \ref{lem:functoriality-of-A}, along with the fact that $f$ is a morphism of Shimura data.
\end{proof}

Suppose now we have two parahoric Shimura data $(\mb{G}_1, \mb{X}_1, \mc{G}_1)$ and $(\mb{G}, \mb{X}, \mc{G})$, and that there is a central isogeny $\alpha\colon \mb{G}_1^\der \to \mb{G}^\der$ which induces an isomorphism of Shimura data
\begin{align*}
    (\mb{G}_1^\ad, \mb{X}_1^\ad) \isomto (\mb{G}^\ad, \mb{X}^\ad).
\end{align*}
Let $x_1$ in $\ms{B}(G_1, \bQ_p)$ and $x$ in $\ms{B}(G,\bQ_p)$ denote points in the buildings of $G_1$ and $G$, respectively, which correspond to $\mc{G}_1$ and $\mc{G}$. Let us further assume that $x_1$ and $x$ correspond to the same point $x_1^\ad = x^\ad$ in the building $\ms{B}(G_1^\ad, \bQ_p) \simeq \ms{B}(G^\ad, \bQ_p)$.

Fix a connected component $\mb{X}_1^+$ as above, which determines a group $\ms{E}(\bm{\mc{G}}_1)$. By the real approximation theorem, we may assume that the image of $\mb{X} \subseteq \mb{X}^\ad$ contains $\mb{X}_1^+$ by replacing $\mb{X}$ by its conjugate by some element of $\mb{G}^\ad(\bQ)$. We set $\mb{E}' = \mb{E}_1\mb{E}$, and define $\ms{E}_{\mb{E}'}(\bm{\mc{G}}_1)$ to be the fiber product
\begin{equation}
    \ms{E}_{\mb{E}'}(\bm{\mc{G}}_1) \defeq \ms{E}(\bm{\mc{G}}_1) \times_{\Gal(\mb{E}_1^p/\mb{E}_1)} \Gal(\mb{E}'^p/\mb{E}').
\end{equation}
Then $\ms{E}_{\mb{E}'}(\bm{\mc{G}}_1)$ is an extension of $\Gal(\mb{E}'^p/\mb{E}')$ by $\ms{A}(\bm{\mc{G}}_1)^\circ$, and by \cite[Lemma 4.6.9]{KP2018}, there is a natural map of extensions
\begin{equation}\label{eq:msE-map}
    \ms{E}_{\mb{E}'}(\bm{\mc{G}}_1) \to \ms{E}(\bm{\mc{G}}).
\end{equation}
In particular, it follows that there is an isomorphism
\begin{equation}\label{eq:star-isom}
    \ms{A}(\bm{\mc{G}}) \ast_{\ms{A}(\bm{\mc{G}}_1)^\circ} \ms{E}_{\mb{E}'}(\bm{\mc{G}}_1) \isomto \ms{A}(\bm{\mc{G}}) \times \Gal(\mb{E}'^p / \mb{E}').
\end{equation}

\subsection{Construction of Kisin--Pappas--Zhou models}\label{ss:KPZ-construction} We now briefly recall the construction of the Kisin--Pappas--Zhou models $\ms{S}_{\mathsf{K}}^\mf{d}(\mb{G},\mb{X})$ from Theorem \ref{thm:KPZ-model-properties}, focusing on the aspects most relevant for our purposes. We encourage the reader to consult the relevant parts of \cite{KP2018} and \cite{KZ21} for details.

Throughout this section we fix the following data/notation:
\vspace*{3 pt}

\begin{multicols}{2}
\begin{itemize}[leftmargin=.4in]
\item $(\mb{G},\mb{X})$ is a Shimura datum,
\item $\mb{E}$ is the reflex field of $(\mb{G},\mb{X})$,
\item $v$ is a place of $\mb{E}$ lying over $p$,

\item $\mc{O}_{(v)}$ is the localization of $\mc{O}_{\mb{E}}$ at $v$,
\item $E=\mb{E}_v$,
\item $G=\mb{G}_{\Q_p}$,
\item $\mc{G}=\mc{G}_x^\circ$ is a parahoric model of $G$,
\item $\mathsf{K}_p\defeq \mc{G}(\bb{Z}_p)$,
\end{itemize}
\end{multicols}

\vspace*{3 pt}

\noindent We refer to \S\ref{ss:BT-theory} and \S\ref{ss:Shim-vars-and-A-group} for the meaning of these terms, as well as other notation related to Bruhat--Tits theory and Shimura varieties.

\medskip

The construction of Kisin--Pappas--Zhou models is carried out in a four step process.

\medskip

\subsubsection*{Step 1: Global construction for Siegel type} \label{subsub:Step-1-Construction} Suppose that $(\mr{GSp}(\mb{V}),\mf{h}^\pm)$ is a Siegel datum. Choose a $\Z_p$-lattice $\Lambda\subseteq \mb{V}_{\Q_p}$ and set
\begin{equation*}
    K_p=\mr{GSp}(\Lambda)\subseteq \mr{GSp}(\mb{V}\otimes_\Q \Q_p).
\end{equation*}
For $\mathsf{K}^p$ a neat compact open subgroups of $\mathrm{GSp}(\mathbf{V}\otimes_\Q\bb{A}^f_p)$, we set $\pmb{\mathscr{S}}_{\mathsf{K}_p\eK^p}(\mr{GSp}(\mb{V}),\mf{h}^\pm)$ to be the $\Z_{(p)}$-scheme given by Mumford's moduli of principally polarized abelian varieties with $\mathsf{K}^p$-level structure (see \cite[\S4]{DeligneModulaire}). These form a $\mr{GSp}(\mb{V}\otimes_\Q \A_f^p)$-system of $\Z_{(p)}$-schemes.

\medskip

\subsubsection*{Step 2: Construction for Hodge type data} \label{subsub:Step-2-Construction}

Let us now assume that $(\mb{G},\mb{X},\mc{G})$ is a parahoric Hodge type datum and let us fix a parahoric Hodge embedding
\begin{equation*}
    \iota\colon (\mb{G},\mb{X},\mc{G})\to (\mr{GSp}(\mb{V}),\mf{h}^\pm,\mr{GSp}(\Lambda)),
\end{equation*}
which respects stabilizers. As in \cite[Lemma 2.10]{Kisin2010} we may choose a neat compact open subgroup $\mathsf{K}_1^p\subseteq \mathbf{GSp}(\mathbf{V}\otimes_\Q\bb{A}_f^p)$ such that $\iota(\wt{\mathsf{K}}_p\mathsf{K}^p)\subseteq \mr{GSp}(\Lambda)\mathsf{K}^p_1$, and such that the map
\begin{equation}\label{eq:very-good-Hodge-emb}
    \Sh_{\wt{\mathsf{K}}_p\mathsf{K}^p}(\mb{G},\mb{X})\to \Sh_{\mr{GSp}(\Lambda)\mathsf{K}_1^p}(\mr{GSp}(\mb{V}),\mf{h}^\pm)_{\mb{E}},
\end{equation}
is a closed embedding. 

We then set 
\begin{itemize}
    \item $\pmb{\ms{S}}_{\wt{\mathsf{K}}_p\mathsf{K}^p}^\vn(\mb{G},\mb{X})$ to be the normalization of the Zariski closure of $\Sh_{\wt{\mathsf{K}}_p\mathsf{K}^p}(\mb{G},\mb{X})$ in the scheme $\pmb{\ms{S}}_{\mr{GSp}(\Lambda)\mathsf{K}_1^p}(\mr{GSp}(\mb{V}),\mf{h}^\pm)_{\mc{O}_{(v)}}$, via the embedding in \eqref{eq:very-good-Hodge-emb},
    \item $\pmb{\ms{S}}_{\mathsf{K}_p\mathsf{K}^p}^\vn(\mb{G},\mb{X})$ to be the normalization (in the sense of \stacks{0BAK}) of $\pmb{\ms{S}}_{\wt{\mathsf{K}}_p\mathsf{K}^p}^\vn(\mb{G},\mb{X})$ relative to finite map 
    \begin{equation*}
        \Sh_{\mathsf{K}_p\mathsf{K}^p}(\mb{G},\mb{X})\to \Sh_{\wt{\mathsf{K}}_p\mathsf{K}^p}(\mb{G},\mb{X}).
    \end{equation*}
\end{itemize}

\begin{rmk}\label{rem:qproj} As each $\pmb{\ms{S}}_{\mr{GSp}(\Lambda)\mathsf{K}_1^p}(\mr{GSp}(\mb{V}),\mf{h}^\pm)_{\mc{O}_{(v)}}$ is quasi-projective over $\Z_{(p)}$ (see \cite[Theorem 7.9]{MumfordGIT}), Zariski closed embeddings are quasi-projective, normalization maps are finite in our setting (see \stacks{035R} and \stacks{07QW}) and thus quasi-projective, we deduce that each $\pmb{\ms{S}}_{\mathsf{K}_p\mathsf{K}^p}^\vn(\mb{G},\mb{X})$ is a quasi-projective $\mc{O}_{(v)}$-scheme.
\end{rmk}

As Zariski closure and normalization are functorial constructions, we see that $\{\pmb{\ms{S}}_{\mathsf{K}_p\mathsf{K}^p}^\vn(\mb{G},\mb{X})\}_{\eK^p}$ form a projective system. We then further set 
\begin{equation*}
\pmb{\ms{S}}_{\mathsf{K}_p}^\vn(\mb{G},\mb{X})\defeq \varprojlim_{\mathsf{K}^p}\pmb{\ms{S}}_{\mathsf{K}_p\mathsf{K}^p}^\vn(\mb{G},\mb{X}).
\end{equation*}
The continuous $\mb{G}(\A_f^p)$-action of $\Sh_{\mathsf{K}_p}(\mb{G},\mb{X})_E$ extends to a continuous action on $\pmb{\ms{S}}_{\mathsf{K}_p}^\vn(\mb{G},\mb{X})$ with the property that
\begin{equation*}
\pmb{\ms{S}}_{\mathsf{K}_p}^\vn(\mb{G},\mb{X})/\mathsf{K}^p\simeq \pmb{\ms{S}}_{\mathsf{K}_p\mathsf{K}^p}^\vn(\mb{G},\mb{X}),
\end{equation*}
compatibly in $\mathsf{K}^p$. We finally set
\begin{equation*}
    \ms{S}_{\mathsf{K}_p}^\vn(\mb{G},\mb{X})\defeq \pmb{\ms{S}}_{\mathsf{K}_p}^\vn(\mb{G},\mb{X})_{\mc{O}_v} \ \text{ and } \ \ms{S}_{\mathsf{K}_p\mathsf{K}^p}^\vn(\mb{G},\mb{X})\defeq \pmb{\ms{S}}_{\mathsf{K}_p\mathsf{K}^p}^\vn(\mb{G},\mb{X})_{\mc{O}_v},
\end{equation*}
which is a scheme over $\mc{O}_v$ with $\mb{G}(\A_f^p)$-action and a scheme over $\mc{O}_v$, respectively.

A priori, these schemes depend on the choice of parahoric Hodge embedding which respects stabilizers. That said, the following shows that the not case, thus justifying the omission of the parahoric Hodge embedding from the notation. 

\begin{prop}[{\cite[Corollary 5.3.3]{KZ21}, \cite[Theorem 4.5.2]{PR2021}, and \cite{DvHKZ}}] The objects  $\ms{S}_{\mathsf{K}_p}^\vn(\mb{G},\mb{X})$, $ \pmb{\ms{S}}_{\mathsf{K}_p}^\vn(\mb{G},\mb{X})$, $\ms{S}_{\mathsf{K}_p\mathsf{K}^p}^\vn(\mb{G},\mb{X})$, and $\pmb{\ms{S}}_{\mathsf{K}_p\mathsf{K}^p}^\vn(\mb{G},\mb{X})$ are canonically independent of parahoric Hodge embedding.
\end{prop}

While the action of $\mb{G}(\A_f^p)$ extends to actions on $\ms{S}_{\mathsf{K}_p}^\vn(\mb{G},\mb{X})$ and $ \pmb{\ms{S}}_{\mathsf{K}_p}^\vn(\mb{G},\mb{X})$, it is not a priori clear that these can be extended further to actions of $\ms{A}(\bm{\mc{G}})$. That said, this does hold true when $(\mb{G},\mb{X},\mc{G})$ is very good, as explained in \cite[\S4.6]{KP2018}.\footnote{Indeed, the fact that $(\mb{G},\mb{X})$ is very good implies, in particular, that $Z(\mb{G})$ is an $R$-smooth torus, and so Proposition \ref{prop:KZ-R-smooth} applies. This means that there is a group map $\ms{A}(\bm{\mc{G}})\to \ms{B}(\bm{\mc{G}}_x)$, where, if $\mc{G}=\mc{G}_x^\circ$, the group scheme $\mc{G}_x$ is the stabilizer Bruhat--Tits group scheme as in \cite[\S1.1.2]{KP2018}, and $\ms{B}(\bm{\mc{G}}_x)$ is as in \cite[\S4.5.6]{KP2018}. Then, one can define the action precisely as in  \cite[Lemma 4.5.7]{KP2018}.}

\medskip

\subsubsection*{Step 3: Global construction relative to a very good Hodge type datum} \label{subsub:Step-3-Construction}

Suppose now that $(\mb{G},\mb{X},\mc{G})$ is a general acceptable parahoric Shimura datum of abelian type. The key proposition which allows us to bootstrap from the previous cases is the following.

\begin{prop}[{\cite[Proposition 5.2.6]{KZ21}}] There exists a very good parahoric Shimura datum $(\mb{G}_1,\mb{X}_1,\mc{G}_1)$ of Hodge type well-adapted to $(\mb{G},\mb{X},\mc{G})$.
\end{prop}

\begin{proof}
    The only claim that needs to be checked is that $G^\der$ may be taken to be $R$-smooth. But this follows by the construction in the proof of loc. cit. (cf. \cite[\S 5.3.10]{KZ21}). 
\end{proof}

Throughout the remainder of this section, we fix $(\mb{G}_1, \mb{X}_1, \mc{G}_1)$ as in the proposition above. Choose connected components $\mb{X}^+ \subseteq \mb{X}$ and $\mb{X}_1^+\subseteq\mb{X}_1$ which correspond in $\mb{X}^\mr{ad}\simeq\mb{X}_1^\mr{ad}$. Using Lemma \ref{lemma:Defined-Over-E^p}, we obtain connected components $\Sh_{\mathsf{K}_p}(\mb{G},\mb{X})^+$ and $\Sh_{\mathsf{K}_{p,1}}(\mb{G}_1,\mb{X}_1)^+$ defined over $\mb{E}^p$ and $\mb{E}_1^p$, respectively. We also obtain subgroups
\begin{equation*}
    \ms{E}(\bm{\mc{G}})\subseteq \ms{A}(\bm{\mc{G}})\times\Gal(\mb{E}^p/\mb{E}),\qquad \ms{E}(\bm{\mc{G}}_1)\subseteq \ms{A}(\bm{\mc{G}}_1)\times\Gal(\mb{E}_1^p/\mb{E}),
\end{equation*}
and the group 
\begin{equation*}
    \ms{E}_{\mb{E}'}(\bm{\mc{G}}_1)=\ms{E}(\bm{\mc{G}}_1)\times_{\Gal(\mb{E}_1^p/\bm{E})}\Gal(\mb{E}'^p/\bm{E}),
\end{equation*}
as defined in \S\ref{ss:Shim-vars-and-A-group}.

We have a bijection
\begin{equation}\label{eq:pi0-bij}
    \pi_0(\Sh_{\mathsf{K}_{p,1}}(\mb{G}_1,\mb{X}_1)_{\bm{E}_1^p})\isomto \pi_0(\pmb{\ms{S}}^\vn_{\mathsf{K}_{p,1}}(\mb{G}_1,\mb{X}_1)_{\mc{O}_{\mb{E}_1^p}}),
\end{equation}
as follows from combining the following general results.
\begin{lemma}\label{lem:pi0-omnibus} Let $\{S_i\}$ be a projective system of spectral spaces with quasi-compact surjective transition maps. Write $S=\varprojlim S_i$. Then, 
\begin{enumerate}[$($1$)$]
    \item $S$ is a spectral space,
     \item the $($dense$)$ quasi-compact open subsets of $S$ are precisely those of the form $U=\varprojlim U_i$ with each $U_i\subseteq S_i$ a $($dense$)$ quasi-compact open subset such that $U_i$ is the preimage of $U_j$ under $S_i\to S_j$,
    \item the natural map $\pi_0(S)\to \pi_0(S_i)$ is surjective for all $i$,
    \item the natural map $\pi_0(S)\to\lim \pi_0(S_i)$ is a homeomorphism.
\end{enumerate}
Moreover, if each $S_i$ is a disjoint union of finitely many irreducible spaces, then for every dense quasi-compact open $U\subseteq S$ the natural map $\pi_0(U)\to\pi_0(S)$ is a homeomorphism.
\end{lemma}
\begin{proof}
Claim (1) follows from \cite[Chapter 0, Theorem 2.2.10]{FujiwaraKato}. Claim (2) follows from \cite[Chapter 0, Proposition 2.2.9 and Lemma 2.2.19]{FujiwaraKato}. Claim (4) follows from \cite[Lemma 4.1.6]{ALY1P2}), and claim (3) then follows from claim (4), as each projection $\pi_0(S_i)\to\pi_0(S_j)$ is a surjection of compact Hausdorff spaces (see \stacks{0906}), and so $\varprojlim \pi_0(S_i)\to \pi_0(S_i)$ is surjective by Tychonoff's theorem. For the final claim write $U=\varprojlim U_i$ with $U_i\subseteq S_i$ a dense quasi-compact open. Then, each $\pi_0(U_i)\to \pi_0(S_i)$ is clearly a bijection, and so the claim follows by passing to the limit using (4).
\end{proof}
\begin{lemma}\label{lem:normal-components} Let $X\to Y$ be a flat and finite type morphism of Noetherian schemes where $X$ is normal. Then, $X$ is a finite disjoint union of irreducible spaces, and for any dense open $U\subseteq Y$, the open $X_U\subseteq X$ is dense.
\end{lemma}
\begin{proof} The first claim follows from \stacks{033N}, and the latter is clear as $f$ is open by \stacks{01UA}.
\end{proof}

The bijection \eqref{eq:pi0-bij} shows that $\Sh_{\mathsf{K}_{p,1}}(\mb{G}_1,\mb{X}_1)^+$ corresponds uniquely to a connected component $\pmb{\ms{S}}^\vn_{\mathsf{K}_{p,1}}(\mb{G}_1,\mb{X}_1)^+$ of $\pmb{\ms{S}}^\vn_{\mathsf{K}_{p,1}}(\mb{G}_1,\mb{X}_1)$. Set $\mc{O}'^{p}_{(v)}=\mc{O}_{\mb{E}'^p}\otimes_{\mc{O}_{\mb{E}}}\mc{O}_{(v)}$. Consider then 
\begin{equation}\label{eq:product-set}
    \pmb{\ms{S}}^\vn_{\mathsf{K}_{p,1}}(\mb{G}_1,\mb{X}_1)^+_{\mc{O}'^{p}_{(v)}}\times \ms{A}(\bm{\mc{G}}).
\end{equation}
As observed in \hyperref[subsub:Step-2-Construction]{\textbf{Step 2}}, we have that the right $\ms{A}(\bm{\mc{G}}_1)$-action extends to $\pmb{\ms{S}}^\vn_{\mathsf{K}_{p,1}}(\mb{G}_1,\mb{X}_1)$ and by the bijection \eqref{eq:pi0-bij} we deduce that $\ms{E}(\bm{\mc{G}}_1)$ stabilizes $\pmb{\ms{S}}^\vn_{\mathsf{K}_{p,1}}(\mb{G}_1,\mb{X}_1)^+$. We then get an induced action of $\ms{E}_{\mb{E}'}(\bm{\mc{G}}_1)$ on $\pmb{\ms{S}}^\vn_{\mathsf{K}_{p,1}}(\mb{G}_1,\mb{X}_1)^+_{\mc{O}'^{p}_{(v)}}$. We then finally let $\ms{E}_{\mb{E}'}(\bm{\mc{G}}_1)$ act on the right of \eqref{eq:product-set} by setting
\begin{equation}\label{eq:msE-action}
    (s,a)\cdot e=(se,\ov{e}^{-1}a\ov{e}),
\end{equation}
where $\ov{e}$ denotes the image of $e$ under the homomorphisms
\begin{equation}\label{eq:E-to-A}
    \ms{E}_{\mb{E}'}(\bm{\mc{G}}_1)\to \ms{E}(\bm{\mc{G}})\to\ms{A}(\bm{\mc{G}}),
\end{equation}
where the first map is as in \eqref{eq:msE-map}, and the second map is projection.

We then consider $\ms{A}(\bm{\mc{G}})\rtimes \ms{E}_{\mathbf{E}'}(\bm{\mc{G}}_1)$, where $\ms{E}_{\mathbf{E}'}(\bm{\mc{G}}_1)$ acts again via the homomorphism from \eqref{eq:E-to-A}. Let the semi-direct product $\ms{A}(\bm{\mc{G}})\rtimes \ms{E}_{\mathbf{E}'}(\bm{\mc{G}}_1)$ act on \eqref{eq:product-set} by having $\ms{A}(\bm{\mc{G}})$ act by right multiplication on the $\ms{A}(\bm{\mc{G}})$-factor of \eqref{eq:product-set}, i.e.,
\begin{equation}\label{eq:msA-action}
    (s,a)\cdot a' = (s, aa').
\end{equation} 
As in \eqref{eq:action} we obtain an action of $\ms{A}(\bm{\mc{G}})\ast_{\ms{A}(\bm{\mc{G}}_1)^\circ}\ms{E}_{\mathbf{E}'}(\bm{\mc{G}}_1)$ on
\begin{equation}\label{eq:product-set-2}
    \left[\pmb{\ms{S}}^{\vn}_{\mathsf{K}_{p,1}}(\mb{G}_1,\mb{X}_1)^+_{\mc{O}'^{p}_{(v)}}\times \ms{A}(\bm{\mc{G}})\right]/\ms{A}(\bm{\mc{G}}_1)^\circ,
\end{equation}
with notation as in loc.\@ cit. Since 
\begin{equation*}
    \ms{A}(\bm{\mc{G}})\ast_{\ms{A}(\bm{\mc{G}}_1)^\circ}\ms{E}_{\mathbf{E}'}(\bm{\mc{G}}_1)\simeq \ms{A}(\bm{\mc{G}})\times\Gal(\mb{E}'^p/\mb{E}), 
\end{equation*}
(see (\ref{eq:star-isom})), we obtain an action of $\ms{A}(\bm{\mc{G}})\times\Gal(\mb{E}'^p/\mb{E})$, and hence of $\mb{G}(\A_f^p)\times\Gal(\mb{E}'^p/\mb{E})$, on \eqref{eq:product-set-2}.

Now, let $J$ denote the image of the natural map 
    \begin{equation*}
        \ms{A}(\bm{\mc{G}}_1)^\circ\backslash \ms{A}(\bm{\mc{G}})\to \ms{A}(\mb{G}_1)^\circ\backslash \ms{A}(\mb{G})/K_p,
    \end{equation*}
where the map $\ms{A}(\bm{\mc{G}}_1)^\circ\to \ms{A}(\bm{\mc{G}})$ is as  in Lemma \ref{lem:functoriality-of-A}. Then $J$ is finite. Let $\mf{d} = (\mb{G}_1, \mb{X}_1, \mc{G}_1)$, and consider 
\begin{equation*}
    \pmb{\ms{S}}_{\mathsf{K}_p}^\mf{d}(\mb{G},\mb{X})_{\mc{O}'^{p}_{(v)}}\defeq \left([\pmb{\ms{S}}_{\mathsf{K}_{p,1}}^\vn(\mb{G}_1,\mb{X}_1)_{\mc{O}'^{p}_{(v)}}^+\times \ms{A}(\bm{\mc{G}})]/\ms{A}(\bm{\mc{G}}_1)^\circ\right)^{J},
\end{equation*}
which we endow with the diagonal (relative to the $J$-indexed factors) action of the group $\mb{G}(\A_f^p)\times\Gal(\mb{E}'^p/\mb{E})$. This action is continuous, and so by Galois descent gives to a unique scheme $\pmb{\ms{S}}_{\mathsf{K}_p}^\mf{d}(\mb{G},\mb{X})_{\mc{O}_{(v)}'}$ with continuous action of $\mb{G}(\A_f^p)$, where $\mc{O}_{(v)}'=\mc{O}_{\mb{E}'}\otimes_{\mc{O}_{\mb{E}}}\mc{O}_{(v)}$.\footnote{That Galois descent applies here follows from the continuity of the action, the quasi-projectivity observation made in Remark \ref{rem:qproj}, and the fact that quasi-projective morphisms satisfy effective descent relative to finite locally free maps (see \stacks{0CCJ}).}

\medskip

\subsubsection*{Step 4: The construction in the local setting}\label{subsub:Step-4-Construction}

Consider 
\begin{equation*}
    (\pmb{\ms{S}}_{\mathsf{K}_p}^\mf{d}(\mb{G},\mb{X})_{\mc{O}'_{(v)}})_{\mc{O}_E}.
\end{equation*}
That $\mb{E}'$ is completely split over $\mb{E}$ at every prime over $p$ implies, in particular, that this decomposes as several copies of the same $\mc{O}_v$-scheme with a continuous action of $\mb{G}(\A_f^p)$, indexed by by the place of $\mb{E}'$ lying over $v$. Choose any of these copies, and call it $\ms{S}_{\mathsf{K}_p}^\mf{d}(\mb{G},\mb{X})$. We then set
\begin{equation*}
    \ms{S}_{\mathsf{K}_p\mathsf{K}^p}^\mf{d}(\mb{G},\mb{X})\defeq \ms{S}_{\mathsf{K}_p}^\mf{d}(\mb{G},\mb{X})/\mathsf{K}^p,
\end{equation*}
for any neat compact open subgroup $\mathsf{K}^p$ of $\mb{G}(\A_f^p)$. That these quotients exist and are quasi-projective can be deduced from Remark \ref{rem:qproj} in conjunction with \stacks{01ZY} and \stacks{07S6}.

\medskip

\subsection{Some properties and functoriality results}

In this subsection we establish some compatibilities satisfied by the Kisin--Pappas--Zhou integral models, and we prove the minimal amount of functoriality needed to prove our main result (see Theorem \ref{thm:main}). This main result will then establish functoriality in general (see Theorem \ref{thm:functoriality}).

\subsubsection{Kisin--Pappas--Zhou integral models of toral type}
To utilize the results of \cite{DanielsToral}, it is useful to explicitly describe Kisin--Pappas--Zhou integral models of tori. 

Recall from Example \ref{ex:toral-type} that a Shimura datum $(\mb{T},\{h\})$ is of toral type if $\mb{T}$ is a torus. In this case, for every compact open subgroup $\eK = \eK_p\eK^p \subseteq \mb{T}(\bA_f^p)$, the Shimura variety $\Sh_{\eK}(\mb{T},\{h\})$ is zero-dimensional, i.e.,
\begin{equation}\label{eq:shim-var-of-torus}
    \Sh_{\eK}(\mb{T}, \{h\})_E = \bigsqcup_{i \in I} \Spec(E_i)
\end{equation}
for some finite index set $I$. In fact, each field $E_i$ is the same field, and is an unramified extension of $\Q_p$ (see \cite[Lemma 4.2]{DanielsToral} and the proof of \cite[Lemma 4.8]{DanielsToral}).

Set $T=\mb{T}_{\Q_p}$. As discussed in \S\ref{sss:parahoric-models-tori}, there is a unique parahoric model $\mc{T}$ of $T$, given by the connected component $(\mc{T}^\mr{lft})^\circ$ of the N\'eron model $\mc{T}^\mr{lft}$ of $T$. 

\begin{prop}\label{prop:toral-type}
    Suppose $(\mb{T},\{h\}, \mc{T})$ is parahoric Shimura datum of toral type, and let $\mf{d} = (\mb{G}_1, \mb{X}_1, \mc{G}_1)$ be a very good parahoric Shimura datum of Hodge type which is well-adapted to $(\mb{T}, \{h\}, \mc{T})$. Then, 
    \begin{align*}
        \ms{S}_\eK^\mf{d}(\mb{T}, \{h\}) = \bigsqcup_{i \in I} \Spec(\mc{O}_{E_i}),
    \end{align*}
    for any neat compact open subgroup $\mathsf{K}^p\subseteq \mb{T}(\A_f^p)$ and with $\eK = \eK_p\eK^p$. The same holds true for $\ms{S}_\eK^\vn(\mb{T}, \{h\})$ if $(\mb{T}, \{h\})$ is of Hodge type.
\end{prop}
\begin{proof}
    First, observe that as $(\mb{G}_1, \mb{X}_1)$ is adapted to the datum $(\mb{T}, \{h\})$, we have an isogeny $\mb{G}_1^\der \to \mb{T}^\der = \{e\}$, and hence $\mb{G}_1^\der$ is trivial, since it is connected. Thus, $\mb{G}_1$ is a torus.\footnote{Indeed, if $H$ is a commutative reductive group over a field $K$, and $T$ is a maximal torus of $H$, then $H=Z_H(T)=T$ (see \cite[Corollary 17.84]{Milne2017}).} Consider an arbitrary neat compact open subgroup $\eK^p_1\subseteq \mb{G}_1(\A_f^p)$.

    Choose a parahoric Hodge embedding $\iota\colon(\mb{G}_1, \mb{X}_1, \mc{G}_1) \to (\mb{GSp}(\mb{V}),\mf{h}^\pm,\mr{GSp}(\Lambda))$ which respects stabilizers, and a neat compact open subgroup $\eL_1^p \subseteq \mr{GSp}(\mb{V}\otimes_\bQ \bA_f^p)$ such that $\iota(\eK_{1,p}\mathsf{K}^p_1) \subseteq \mr{GSp}(\Lambda)\eL_1^p$, inducing a closed embedding 
    \begin{equation}
        \Sh_{\eK_{1,p}\mathsf{K}^p_1}(\mb{G}_1, \mb{X}_1) \to \Sh_{\mr{GSp}(\Lambda)\eL_1^p}(\mb{GSp}(\mb{V}),\mf{h}^\pm)_{\mb{E}_1}.
    \end{equation}
    Note that $\pmb{\ms{S}}_{\mr{GSp}(\Lambda)\eL_1^p}(\mb{GSp}(\mb{V}),\mf{h}^\pm)$ is flat, separated, and finite over $\Z_{(p)}$. Thus, considering \eqref{eq:shim-var-of-torus}, and using notation from \hyperref[subsub:Step-2-Construction]{\textbf{Step 2}} of \S\ref{ss:KPZ-construction}, it follows from Lemma \ref{lem:closure-of-point} below that $\pmb{\ms{S}}^\vn_{\wt{\mathsf{K}}_{p,1}\eK^p_1}(\mb{G}_1,\mb{X}_1)$ and thus $\pmb{\ms{S}}^\vn_{\mathsf{K}_{p,1}\eK^p_1}(\mb{G}_1,\mb{X}_1)$ are disjoint unions of semi-local localizations of $\mc{O}_\mb{F}$ for some finite extension $\mb{F}$ of $\mb{E}$. 
    
    This implies that, using notation now from \hyperref[subsub:Step-3-Construction]{\textbf{Step 3}} of \S\ref{ss:KPZ-construction}, that $\pmb{\ms{S}}_{\mathsf{K}_{p,1}}^\vn(\mb{G}_1,\mb{X}_1)_{\mc{O}_E}^+$ is either the spectrum of a field or of a semi-local Dedekind domain. Thus, from the construction in \hyperref[subsub:Step-4-Construction]{\textbf{Step 4}} of \S\ref{ss:KPZ-construction} we deduce that for each neat compact open $\mathsf{K}^p\subseteq\mb{G}(\A_f^p)$ one has that $\ms{S}_{\eK_p\mathsf{K}^p}^\mf{d}(\mb{G},\mb{X})$ is either the disjoint union of spectra of extensions of $E$, or the disjoint union of spectra of the rings of integers of finite extensions of $E$. That said, the former is impossible as it is then clear that such a model cannot satisfy the extension property from (2) of Theorem \ref{thm:KPZ-model-properties}, and so the latter must hold. The fist claim follows. Moreover, the second claim follows from similar arguments.

\end{proof}

Following \stacks{035E}, if $X$ is a scheme with the property that every quasi-compact open has finitely many connected components, then we denote the normalization of $X$ by $X^\nu$.

\begin{lemma}\label{lem:closure-of-point} Let $\mc{O}$ be an excellent Dedekind domain and $\ms{X}$ be a separated, finite type, flat $\mc{O}$-scheme. Let $x$ is a closed point of $\ms{X}_\eta$ and let $\ov{x}$ be the Zariski closure of $x$ in $\ms{X}$. Then, $\ov{x}^\nu$ is an open subscheme of $\Spec(\mc{O}')$ for a Dedekind domain $\mc{O}'$ finite flat over $\mc{O}$.
\end{lemma}
\begin{proof} Let $\ov{X}$ be a compactification of $\ms{X}\to\Spec(\mc{O})$ (see \stacks{0ATT} and \stacks{0F41}). Replacing $\ov{\ms{X}}$ with the closed suscheme cut out by the $\mc{O}$-torsion in $\mc{O}_{\ms{X}}$, we may assume that $\ov{\ms{X}}$ is flat over $\mc{O}$ (see \cite[Proposition 14.14]{GortzWedhorn}). 

Let $Z$ denote the Zariski closure of $x$ in $\ov{\ms{X}}$. We claim that $Z^\nu=\Spec(\mc{O}')$ for a Dedekind ring $\mc{O}'$ finite flat over $\mc{O}$. As $\ov{x}=\ms{X}\cap Z$, the lemma will then follow from \stacks{035K}. To prove the claim, we observe that $Z$ is evidently irreducible and reduced, and thus integral. But, as $Z\to \mathrm{Spec}(\mc{O})$ is dominant, we deduce that it is flat by \cite[Proposition 14.14]{GortzWedhorn}. But, since $\ov{\ms{X}}$ is a proper $\mc{O}$-scheme, so is $Z$. Thus, by \stacks{0D4J} we deduce that $Z_s$ is zero dimensional. Thus, $Z\to \Spec(\mc{O})$ is proper and quasi-finite, and so finite (see \stacks{02OG}). It follows that $Z^\nu$ is also an integral scheme and $ Z^\nu\to \Spec(\mc{O})$ is dominant and finite (using \stacks{035R}) over $\mc{O}$ and so flat again by \cite[Proposition 14.14]{GortzWedhorn}. Write $Z^\nu=\Spec(\mc{O}')$, then $\mc{O}'$ is $1$-dimensional Noetherian normal domain, and so a Dedekind domain as desired.
\end{proof}

Given Proposition \ref{prop:toral-type}, we see that the definition of $\ms{S}_{\mathsf{K}}^\mf{d}=(\mb{T},\{h\})$ is independent of the choice of $\mf{d}$. We denote by $\ms{S}_{\mathsf{K}}(\mb{T},\{h\})$ the common object.

\subsubsection{Functoriality for ad-isomorphisms}

Suppose that $\alpha\colon (\mb{G}_1,\mb{X}_1,\mc{G}_1)\to (\mb{G},\mb{X},\mc{G})$ is an ad-isomorphism of parahoric Shimura datum. Take a very good parahoric Shimura datum $\mf{d}=(\mb{G}_2,\mb{X}_2,\mc{G}_2)$ of Hodge type well-adapted to $(\mb{G}_1,\mb{X}_1,\mc{G}_1)$. Then, since $\alpha$ is an ad-isomorphism, one sees that $\mf{d}$ is also well-adapted to $(\mb{G},\mb{X},\mc{G})$ in the obvious way.

\begin{prop}\label{prop:funct-for-ad-isom} There exists a unique quotient-finite \'etale morphism 
\begin{equation*}
    \alpha\colon \ms{S}^\mf{d}_{\mathsf{K}_{p,1}}(\mb{G}_1,\mb{X}_1)\to \ms{S}_{\mathsf{K}_p}^\mf{d}(\mb{G},\mb{X})_{\mc{O}_{E_1}},
\end{equation*} 
which is equivariant for the map $\alpha\colon \mb{G}_1(\A_f^p)\to \mb{G}(\A_f^p)$ and models 
\begin{equation*} \alpha\colon \Sh_{\mathsf{K}_{p,1}}(\mb{G}_1,\mb{X}_1)\to \Sh_{\mathsf{K}_p}(\mb{G},\mb{X})_{E_1}.
\end{equation*}
\end{prop}

Here by \emph{quotient-finite \'etale} we mean that if $\mathsf{K}_1^p\subseteq \mb{G}_1(\A_f^p)$ and $\mathsf{K}^p\subseteq\mb{G}(\A_f^p)$ are neat compact open subgroups such that $\alpha(\mathsf{K}_1^p)\subseteq\mathsf{K}^p$, then the induced map 
\begin{equation*} 
\alpha\colon \ms{S}^\mf{d}_{\mathsf{K}^p_1}(\mb{G}_1,\mb{X}_1)\to \ms{S}^\mf{d}_{\mathsf{K}^p}(\mb{G},\mb{X})_{\mc{O}_{E_1}},
\end{equation*}
is a finite \'etale morphism.

\begin{proof}[Proof of Proposition \ref{prop:funct-for-ad-isom}] The fact that such a map is unique is clear by the separatedness of $\alpha$ and the flatness of $\ms{S}^\mf{d}_{\mathsf{K}^p_1}(\mb{G}_1,\mb{X}_1)$, which guarantees that its generic fiber is Zariski dense (e.g., see \cite[Proposition 14.14]{GortzWedhorn}).

To see the existence, we note that by definition we have 
\begin{equation*}
    \ms{S}_{\mathsf{K}_{p,1}}^\mf{d}(\mb{G}_1,\mb{X}_1)=\left(\left[\ms{S}_{\mathsf{K}_{p,2}}^\vn(\mb{G}_2,\mb{X}_2)^+_{\mc{O}_{E_1}}\times \ms{A}(\bm{\mc{G}}_1)\right]/\ms{A}(\bm{\mc{G}}_2)^\circ\right)^{J_1},
\end{equation*}
and
\begin{equation*}
    \ms{S}_{\mathsf{K}_{p}}^\mf{d}(\mb{G},\mb{X})_{\mc{O}_{E_1}}=\left(\left[\ms{S}_{\mathsf{K}_{p,2}}^\vn(\mb{G}_2,\mb{X}_2)^+_{\mc{O}_{E_1}}\times \ms{A}(\bm{\mc{G}})\right]/\ms{A}(\bm{\mc{G}}_2)^\circ\right)^{J}.
\end{equation*}
Now, by Lemma \ref{lem:functoriality-of-A} and Lemma \ref{lem:functoriality-of-E} we get morphisms 
\begin{equation*}
    \alpha\colon \ms{A}(\bm{\mc{G}}_1)\to \ms{A}(\bm{\mc{G}}),\quad \alpha\colon \ms{A}(\bm{\mc{G}}_1)^\circ\to \ms{A}(\bm{\mc{G}}),\ \text{ and } \ \alpha\colon \ms{E}(\bm{\mc{G}}_1)\to \ms{E}(\bm{\mc{G}}).
\end{equation*}
Because $\alpha(\mathsf{K}_{p,1})\subseteq\mathsf{K}_p$, we have a map of sets $\alpha\colon J_1\to J$. Since the isomorphism from \eqref{eq:star-isom} is natural in these operations, we deduce the existence of a map 
\begin{equation*}
    \left(\left[\ms{S}_{\mathsf{K}_{p,2}}^\vn(\mb{G}_2,\mb{X}_2)^+_{\mc{O}_{E_1}}\times \ms{A}(\bm{\mc{G}}_1)\right]/\ms{A}(\bm{\mc{G}}_2)^\circ\right)^{J_1}\to \left(\left[\ms{S}_{\mathsf{K}_{p,2}}^\vn(\mb{G}_2,\mb{X}_2)^+_{\mc{O}_{E_1}}\times \ms{A}(\bm{\mc{G}})\right]/\ms{A}(\bm{\mc{G}}_2)^\circ\right)^{J},
\end{equation*}
equivariant for $\alpha\colon \mb{G}_1(\A_f^p)\to\mb{G}(\A_f^p)$. That the generic fiber of this map agrees with 
\begin{equation*} \alpha\colon \Sh_{\mathsf{K}_{p,1}}(\mb{G}_1,\mb{X}_1)_{E_1}\to \Sh_{\mathsf{K}_p}(\mb{G},\mb{X})_{E_1}.
\end{equation*}
follows as in \cite[Lemma 4.6.13]{KP2018}.

Finally, we show that $\alpha$ is quotient-finite \'etale. Let $\ms{S}_1^+$ denote a connected component of $\ms{S}^\mf{d}_{\eK_{p,1}}(\mb{G}_1, \mb{X}_1)$ mapping to a connected component $\ms{S}^+$ of $\ms{S}^\mf{d}_{\eK_p}(\mb{G}, \mb{X})_{\mc{O}_{E_1}}$. By construction, we have 
\begin{align*}
    \ms{S}_2^+ = \ms{S}^\vn_{\eK_{p,2}}(\mb{G}_2, \mb{X}_2)^+_{\mc{O}_{E_1}}/\Delta_1 \quad\text{and}\quad \ms{S}^+ = \ms{S}^\vn_{\eK_{p,2}}(\mb{G}_2, \mb{X}_2)^+_{\mc{O}_{E_1}}/\Delta, 
\end{align*}
where 
\begin{align*}
    \Delta_1 = \ker(\ms{A}(\bm{\mc{G}}_2)^\circ \to \ms{A}(\mb{\bm{\mc{G}}})) \quad\text{and}\quad\Delta = \ker(\ms{A}(\bm{\mc{G}}_2)^\circ \to \ms{A}(\bm{\mc{G}})),
\end{align*}
(see also the arguments of \cite[Corollary 4.6.15]{KP2018}). Thus $\ms{S}^+$ is the quotient of $\ms{S}_1^+$ by the finite group $\Delta / \Delta_1$, and the result follows.
\end{proof}

In particular, the proposition implies functoriality for morphisms $(\mb{T}_1, \{h_1\}, \mc{T}_1) \to (\mb{T}, \{h\}, \mc{T})$ of parahoric Shimura data of toral type.

\medskip

\subsubsection{Functoriality for abelianizations}

Suppose now that $(\mb{G},\mb{X},\mc{G})$ is a parahoric Shimura datum of abelian type. Consider the parahoric Shimura datum $(\mb{G}^\ab, \mb{X}^\ab, \mc{G}^\ab)$, where $\mb{X}^\ab$ is the result of post-composing any element $h$ of $\mb{X}$ with the canonical map $\delta: \mb{G} \to \mb{G}^\ab$. We have the obvious associated morphism of parahoric Shimura data
\begin{equation*}
    \alpha\colon (\mb{G}, \mb{X}, \mc{G}) \to (\mb{G}^\ab, \mb{X}^\ab, \mc{G}^\ab).
\end{equation*}
Write $\eK_p^\ab = \mc{G}^\ab(\bZ_p)$. 

\begin{prop}\label{prop:funct-for-ab} There exists a unique 
\begin{equation*}
    \alpha\colon \ms{S}^\mf{d}_{\mathsf{K}_p}(\mb{G},\mb{X})\to \ms{S}_{\mathsf{K}_p^\ab}(\mb{G}^\ab,\mb{X}^\ab)_{\mc{O}_{E}},
\end{equation*} 
which is equivariant for the map $\alpha\colon \mb{G}_1(\A_f^p)\to \mb{G}(\A_f^p)$ and models 
\begin{equation*} \alpha\colon \Sh_{\mathsf{K}_{p,1}}(\mb{G}_1,\mb{X}_1)_{E_1}\to \Sh_{\mathsf{K}_p}(\mb{G},\mb{X})_{E_1}.
\end{equation*}
\end{prop}
\begin{proof}
    Given the descriptions of $\Sh_{\eK_p^\ab}(\mb{G}^\ab, \mb{X}^\ab)$ and $\ms{S}_{\eK_p^\ab}(\mb{G}^\ab, \mb{X}^\ab)_{\mc{O}_E}$ from \eqref{eq:shim-var-of-torus} and Proposition \ref{prop:toral-type}, respectively, the result follows from Lemma \ref{lem:normal-lemma}. More precisely, it suffices to construct maps at every finite level. But, by Proposition \ref{prop:toral-type}, the Kisin--Pappas--Zhou model $S'$ for $(\mb{G}^\ab, \mb{X}^\ab)$ base-changed to $\mc{O}_E$ is finite flat over $S = \Spec(\mc{O}_E)$. On the other hand, the finite levels $X$ of the Kisin--Pappas--Zhou models $X$ for $(\mb{G}, \mb{X})$ are normal and quasi-projective and flat over $S$. Thus may apply Lemma \ref{lem:normal-lemma} to the map at finite levels on the generic fiber.
\end{proof}

\begin{lemma}\label{lem:normal-lemma}
    Suppose that $S$ is a scheme and $U$ is an open subscheme of $S$ which is schematically dense and for which the inclusion $U \hookrightarrow S$ is quasi-compact. Given a diagram
    \begin{center}
        \begin{tikzcd}
            X_U 
                \arrow[r, hook] \arrow[d, "\alpha"]
            & X
                \arrow[d, dashed] \arrow[dd, bend left = 35, "f"]
            \\ S'_U
                \arrow[r, hook] \arrow[d]
            & S'
                \arrow[d, "g"]
            \\ U 
                \arrow[r, hook]
            & S
        \end{tikzcd}
    \end{center}
    with $X$ a normal scheme, $f$ quasi-compact, separated, and flat, and $g$ quasi-compact, flat, and integral, there exists a unique arrow dashed arrow as in the diagram above which makes the top square Cartesian.
\end{lemma}
\begin{proof}
    By \stacks{081H}, $X_U \subseteq X$ is schematically dense, so the uniqueness follows from \stacks{01RH}. In turn, by uniqueness, we may localize on $S$ and $X$ to prove existence. Using that $g$ is affine, we may then assume without loss of generality that $S = \Spec(A)$, $S' = \Spec(A')$ and $A \to A'$ is integral. Moreover, localizing on $X$ we may further assume $X = \Spec(B)$ with $B$ a normal domain. We then have the following diagram of rings
    \begin{center}
        \begin{tikzcd}
            A
                \arrow[r, "g"] \arrow[rr, bend left=20, "f"] \arrow[d]
            & A' 
                \arrow[r, dashed] \arrow[d]
            & B 
                \arrow[d]
            \\ \mc{O}(U) 
                \arrow[r]
            & \mc{O}(S_U')
                \arrow[r, "\alpha"] 
            & \mc{O}(X_U)
        \end{tikzcd}
    \end{center}
    The center and right-hand vertical arrows are injections by flatness of $f$ and $g$ along with \stacks{081H} and \stacks{01RE}. 
    
    We claim that $\alpha(A') \subseteq B$. Indeed, let $a' \in A'$. Then by integrality of $A \to A'$, there exists a monic polynomial $p(T)$ in $A[T]$ such that $p(a') = 0$. But then $\alpha(a')$ in $\mc{O}(X_U)$ satisfies the monic polynomial $f(p)(T)$ in $B[T]$. Thus $\alpha(a')$ must lie in $B$ by normality of $B$.

    It follows that we have a map $\Spec(\alpha): \Spec(B) \to \Spec (A')$. We claim that the diagram
    \begin{center}
        \begin{tikzcd}
            X_U 
                \arrow[r, hook] \arrow[d, "\alpha"]
            & X = \Spec(B) 
                \arrow[d, "\Spec(\alpha)"]
            \\ S'_U 
                \arrow[r, hook]
            & S' = \Spec(A')
        \end{tikzcd}
    \end{center}
    is Cartesian. For this it is enough to show that $\Spec(\alpha) \res_{X_U} = \alpha$. Since $U \hookrightarrow X$ is quasi-compact, $U$ is quasi-compact, and therefore the same is true of $S_U'$ and $X_U$ by quasi-compactness of $f$ and $g$. Since $\Spec(\alpha)$ and $\alpha$ induce the same map $\mc{O}(S'_U) \to \mc{O}(X_U)$, the result now follows from \stacks{01P9}.
\end{proof}

\subsubsection{Two constructions for very good Hodge-type data}

We record here the following basic compatibility between the constructions made in \S\ref{ss:KPZ-construction}. Let $\mf{d}=(\mb{G},\mb{X},\mc{G})$ be a very good parahoric Shimura datum of Hodge type.

\begin{prop}\label{prop:HT-basic-model-comparison} There is a canonical $\mb{G}(\A^p_f)$-equivariant identification
\begin{equation*}
    \ms{S}^\mf{d}_{\mathsf{K}_p}(\mb{G},\mb{X})\isomto \ms{S}^\vn_{\mathsf{K}_p},(\mb{G},\mb{X}),
\end{equation*}
extending the identity in the generic fiber.
\end{prop}
\begin{proof} 
    We first construct a $\mb{G}(\bA_f^p)$-equivariant map
    \begin{equation*}
        \pmb{\ms{S}}_{\mathsf{K}_p}^\mf{d}(\mb{G},\mb{X})_{\mc{O}'^{p}_{(v)}} \to\pmb{\ms{S}}_{\mathsf{K}_p}^\vn(\mb{G},\mb{X})_{\mc{O}'^{p}_{(v)}}.
    \end{equation*}
    Choosing a connected component $\mb{X}^+$ of $\mb{X}$, we obtain an inclusion \begin{equation*}
        \pmb{\ms{S}}_{\eK_p}^\vn(\mb{G},\mb{X})_{\mc{O}'^{p}_{(v)}}^+ \hookrightarrow \pmb{\ms{S}}_{\eK_p}^\vn(\mb{G},\mb{X})_{\mc{O}'^{p}_{(v)}},
    \end{equation*} which extends by the action of $\ms{A}(\bm{\mc{G}})$ to
    \begin{equation}\label{eq:vn-to-d-i}
        \pmb{\ms{S}}_{\eK_p}^\vn(\mb{G},\mb{X})_{\mc{O}'^{p}_{(v)}}^+\times \ms{A}(\bm{\mc{G}}) \to \pmb{\ms{S}}_{\eK_p}^\vn(\mb{G},\mb{X})_{\mc{O}'^{p}_{(v)}}
    \end{equation}
    via $(s,a) \mapsto s\cdot a$. One checks that \eqref{eq:vn-to-d-i} is equivariant for the action of $\ms{A}(\bm{\mc{G}}) \rtimes \ms{E}_{\mb{E}'}(\bm{\mc{G}})$ on $\pmb{\ms{S}}_{\eK_p}^\vn(\mb{G},\mb{X})_{\mc{O}'^{p}_{(v)}}^+\times \ms{A}(\bm{\mc{G}})$, and it induces
    \begin{equation}\label{eq:vn-to-d-ii}
        [\pmb{\ms{S}}_{\eK_p}^\vn(\mb{G},\mb{X})_{\mc{O}'^{p}_{(v)}}^+\times \ms{A}(\bm{\mc{G}})]/\ms{A}(\bm{\mc{G}})^\circ \to \pmb{\ms{S}}_{\eK_p}^\vn(\mb{G},\mb{X})_{\mc{O}'^{p}_{(v)}}.
    \end{equation}
    Indeed, this follows by the definition of the action of $\ms{A}(\bm{\mc{G}}) \rtimes \ms{E}_{\mb{E}'}(\bm{\mc{G}})$ on the product $\pmb{\ms{S}}_{\eK_p}^\vn(\mb{G},\mb{X})_{\mc{O}'^{p}_{(v)}}^+\times \ms{A}(\bm{\mc{G}})$ via \eqref{eq:msE-action} and \eqref{eq:msA-action}, along with the fact that the composition
    \begin{equation*}
        \ms{A}(\bm{\mc{G}})^\circ \to \ms{E}_{\mb{E}'}(\bm{\mc{G}}) \to \ms{E}(\bm{\mc{G}}) \to \ms{A}(\bm{\mc{G}})
    \end{equation*}
    is the inclusion $\ms{A}(\bm{\mc{G}})^\circ \hookrightarrow \ms{A}(\bm{\mc{G}})$. In turn, \eqref{eq:vn-to-d-ii} determines a $\ms{A}(\bm{\mc{G}})\ast_{\ms{A}(\bm{\mc{G}})^\circ} \ms{E}_{\mb{E}'}(\bm{\mc{G}})$-equivariant, and hence $\mb{G}(\bA_f)\times \Gal(\mb{E}'^p/\mb{E})$-equivariant, map
    \begin{equation}\label{eq:vn-to-d}
        \pmb{\ms{S}}_{\mathsf{K}_p}^\mf{d}(\mb{G},\mb{X})_{\mc{O}'^{p}_{(v)}} = \left([\pmb{\ms{S}}_{\eK_p}^\vn(\mb{G},\mb{X})_{\mc{O}'^{p}_{(v)}}^+\times \ms{A}(\bm{\mc{G}})]/\ms{A}(\bm{\mc{G}})^\circ\right)^J \to \pmb{\ms{S}}_{\eK_p}^\vn(\mb{G},\mb{X})_{\mc{O}'^{p}_{(v)}}.
    \end{equation}
    The map \eqref{eq:vn-to-d} is birational, since it is an isomorphism on generic fibers by \cite[Lemma 4.6.13]{KP2018}, and, by its definition, it is quasi-finite. By $\mb{G}(\bA_f^p)$-equivariance, we may first pass to the map \eqref{eq:vn-to-d} after taking the quotient by any neat compact open subgroup $\eK^p$ of $\mb{G}(\bA_f^p)$. In that case, the target is normal, so it follows that the quotient of \eqref{eq:vn-to-d} by $\eK^p$ is an open immersion on connected components by Lemma \ref{lem:normal-components} and \cite[Corollary 12.88]{GortzWedhorn}. Passing to the limit, we deduce that \eqref{eq:vn-to-d} is an open immersion. Moreover, \eqref{eq:vn-to-d} is surjective by part (2) of Theorem \ref{thm:KPZ-model-properties}\footnote{Indeed, it is enough to show \eqref{eq:vn-to-d} is surjective on special fibers. An $\ov{\bb{F}}_p$-point of $\pmb{\ms{S}}_{\eK_p}^\vn(\mb{G},\mb{X})_{\mc{O}'^{p}_{(v)}}$ factors through an $R$-point of $\pmb{\ms{S}}_{\eK_p}^\vn(\mb{G},\mb{X})_{\mc{O}'^{p}_{(v)}}$, where $R$ is a discrete valuation ring of characteristic $(0,p)$ (see \stacks{0CM2}). This, in turn, induces an $R[\nicefrac{1}{p}]$-point of $\pmb{\ms{S}}_{\eK_p}^\vn(\mb{G},\mb{X})_{\mc{O}'^{p}_{(v)}}$. We can pull this point back along the isomorphism on generic fibers, and the result lifts to an $R$-point of $\pmb{\ms{S}}_{\eK_p}^\mf{d}(\mb{G},\mb{X})_{\mc{O}'^{p}_{(v)}}$ by part (2) of Theorem \ref{thm:KPZ-model-properties}. This maps to the original $\ov{\bb{F}}_p$-point of $\pmb{\ms{S}}_{\eK_p}^\vn(\mb{G},\mb{X})_{\mc{O}'^{p}_{(v)}}$ by separatedness, as they generically agree.}, so it is an isomorphism.

\end{proof}

\subsubsection{Kisin--Pappas--Zhou integral models for some products}

Let $\mf{d} = (\mb{G}, \mb{X}, \mc{G})$ be a very good parahoric Shimura datum of Hodge type, and that $(\mb{T}, \{h\}, \mc{T})$ is a parahoric Shimura datum of toral type. Write $\mb{E}'$ for the compositum of the reflex fields of the two Shimura data, which is the reflex field of $(\mb{G}\times \mb{T}, \mb{X}\times\{h\})$, and set $\mathsf{M}_p = \mc{T}(\bZ_p)$.

\begin{prop}\label{prop:KPZ-for-products}  There is a canonical $\mb{G}(\A^p_f)\times \mb{T}(\A_f^p)$-equivariant identification
\begin{equation*}
    \ms{S}^\mf{d}_{\eK_p\times \mathsf{M}_p}(\mb{G}\times \mb{T}, \mb{X}\times\{h\}) \isomto \ms{S}_{\eK_p}^\mf{d}(\mb{G}, \mb{X})_{\mc{O}_{E'}} \times \ms{S}_{\mathsf{M}_p}(\mb{T}, \{h\})_{\mc{O}_{E'}}.
\end{equation*}
\begin{proof}
    The projection map $\mb{G}\times \mb{T} \to \mb{G}$ is an ad-isomorphism, so by Proposition \ref{prop:funct-for-ad-isom}, we have a quotient-finite map
    \begin{equation}\label{eq:GxT-to-G}
        \ms{S}^\mf{d}_{\eK_p\times \mathsf{M}_p}(\mb{G}\times \mb{T}, \mb{X}\times\{h\})\to \ms{S}_{\eK_p}^\mf{d}(\mb{G}, \mb{X})_{\mc{O}_{E'}}
    \end{equation}
    extending the morphism on the generic fibers which is equivariant for the projection $(\mb{G}\times\mb{T})(\bA_f^p) \to \mb{G}(\bA_f^p)$. On the other hand, the map of reductive groups $\mb{G} \times \mb{T} \to \mb{T}$ factors through $(\mb{G} \times \mb{T})^\ab$, which is a torus. Hence Proposition \ref{prop:funct-for-ad-isom} and Proposition \ref{prop:funct-for-ab} combine to furnish us with a map
    \begin{equation*}
         \ms{S}^\mf{d}_{\eK_p\times \mathsf{M}_p}(\mb{G}\times \mb{T}, \mb{X}\times\{h\}) \to \ms{S}_{\mathsf{M}_p}(\mb{T}, \{h\})_{\mc{O}_{E'}}
    \end{equation*}
    extending the morphism on generic fibers, which is equivariant for $(\mb{G} \times \mb{T})(\bA_f^p) \to \mb{T}(\bA_f^p)$. From these, we obtain the map
    \begin{equation}\label{eq:product}
        \ms{S}^\mf{d}_{\eK_p\times \mathsf{M}_p}(\mb{G}\times \mb{T}, \mb{X}\times\{h\}) \to \ms{S}_{\eK_p}^\mf{d}(\mb{G}, \mb{X})_{\mc{O}_{E'}} \times \ms{S}_{\mathsf{M}_p}(\mb{T}, \{h\})_{\mc{O}_{E'}},
    \end{equation}
    which is equivariant for $(\mb{G}\times\mb{T})(\bA_f^p) \isomto \mb{G}(\bA_f^p)\times\mb{T}(\bA_f^p)$. The map \eqref{eq:product} is an isomorphism on generic fibers, and is quasi-finite at each finite level because \eqref{eq:GxT-to-G} is quotient-finite. It follows that \eqref{eq:product} is an isomorphism by the argument at the end of the proof of Proposition \ref{prop:HT-basic-model-comparison}.
\end{proof}
\end{prop}


\section{The Pappas--Rapoport conjecture}\label{s:PR-conj}

In this section we state the Pappas--Rapoport conjecture and recall the cases in which it has been previously proven.  We also state our main theorem and applications.

As in \S\ref{ss:KPZ-construction}, throughout this section we often use the following notation and data:

\vspace*{3 pt}

\begin{multicols}{2}
\begin{itemize}[leftmargin=.4in]
\item $(\mb{G},\mb{X})$ is a Shimura datum,
\item $\mb{Z}$ is the center of $\mb{G}$,
\item $\mb{E}$ is the reflex field of $(\mb{G},\mb{X})$,
\item $v$ is a $p$-adic place of $\mb{E}$,
\item $E=\mb{E}_v$,
\item $k_E$ is the residue field of $E$,
\item $G=\mb{G}_{\Q_p}$,
\item $\mc{G}=\mc{G}_x^\circ$ is a parahoric model of $G$,
\item $\mathsf{K}_p\defeq \mc{G}(\bb{Z}_p)$,
\item $\eK^p\subseteq\mb{G}(\A^f_p)$ is a neat compact open subgroup,
\item $\eK=\eK_p\eK^p$.
\end{itemize}
\end{multicols}

\subsection{The Pappas--Rapoport conjecture}\label{ss:PR-conj}

We now recall the formulation of the Pappas--Rapoport conjecture as given in \cite{PR2021} and extended in \cite{DanielsToral}. 

\medskip

\subsubsection{The group $\mc{G}^c$} In constrast to the case of Shimura varieties of Hodge type discussed in \cite{PR2021}, to obtain shtukas on arbitrary Shimura varieties of abelian type, one must consider a modification $\mc{G}^c$ of the parahoric group $\mc{G}$.

Let $\mb{T}$ be a multiplicative group over $\Q$. We denote by $\mb{T}_{\mr{ac}}$ the anti-cuspidal part of $\mb{T}$ as in \cite[Definition 1.5.4]{KSZ}. More precisely, if $\mb{T}_a$ denotes the largest anisotropic subtorus of $\mb{T}$, then $\mb{T}_\mr{ac}$ is the largest subtorus of $\mb{T}_a$ whose base change to $\bb{R}$ contains the maximal split subtorus of $(\mb{T}_a)_\bb{R}$. 

We write $\eG^c$ for the quotient $\eG^c / \eZ_{\mr{ac}}$, and if $\eH$ is a subgroup of $\eG(\bA_f)$, we write $\eH^c$ for the image of $\eH$ under $\eG(\bA_f) \to \eG^c(\bA_f)$. We also denote by $G^c$ the base change of $\eG^c$ to $\bQ_p$. The following lemma will be used below. 

\begin{lemma}\label{lem:Gc-functoriality} If $\alpha\colon (\mb{G}_1,\mb{X}_1)\to (\mb{G},\mb{X})$ is a morphism of Shimura data then $\alpha(\mb{Z}_{1,\mr{ac}})\subseteq \mb{Z}_\mr{ac}$ and so $\alpha$ induces a map $\alpha\colon \mb{G}_1^c\to \mb{G}^c$. Consequently, if $(\mb{G},\mb{X})$ is of Hodge type, then the natural maps $\mb{G}\to\mb{G}^c$ and $\mb{G}^\ab\to (\mb{G}^\ab)^c$ are isomorphisms.
\end{lemma}
\begin{proof} The first claim is \cite[Lemma 4.7]{ImaiKatoYoucis}. The second claim then follows from the first by embedding an arbitrary Hodge type datum $(\mb{G}, \mb{X})$ into one of Siegel type $(\mb{G}_1, \mb{X}_1)$, where it can be explicitly checked that $\mb{Z}_{1,\mr{ac}}$ is trivial. For the claim concerning $\mb{G}^\ab$, we observe that the map $\delta\colon \mb{G}\to \mb{G}^\ab$ restricts to a surjection $\delta\colon \mb{Z}\to \mb{G}^\ab$ (e.g., see \cite[Example 19.25]{Milne2017}). From here, it's easy to check that $\delta$ surjects $\mb{Z}_\mr{ac}$ onto $\mb{G}^\mr{ab}_\mr{ac}$, and the claim follows.
\end{proof}

We set $\mc{G}^c$ to be the parahoric model of $\mb{G}$ induced by $\mc{G}$ in the sense of \S\ref{sss:parahoric-central-quotient}. By construction, we see that $G \to G^c$ extends to a morphism of group $\bZ_p$-schemes $\cG \to \cG^c$. For a conjugacy class of cocharacters $\mbbmu$ of $G_{\ov{\Q}_p}$, we denote by $\mbbmu^c$ the image of this conjugacy class under $G_{\ov{\Q}_p}\to G^c_{\ov{\Q}_p}$.

\medskip

\subsubsection{The shtuka over $\mathrm{Sh}_\eK\gx_E$.}

We now recall the existence of a ``universal shtuka'' which lives over the Shimura variety $\Sh_\eK\gx_E$, as in \cite[\textsection 4.1]{PR2021} and \cite[\S4.2]{DanielsToral}.

For a normal compact open subgroup $\eK_p' \subseteq \eK_p$, we temporarily write $\eK'$ for the product $\eK_p'\eK^p$. Then the transition morphism
\begin{align}
    \pi_{\mathsf{K}',\mathsf{K}}\colon \Sh_{\eK'}\gx_E \to \Sh_\eK\gx_E
\end{align}
is a finite \'etale Galois cover with Galois group $\mathsf{K}_p/(\mathsf{K}_p'\mb{Z}(\bb{\Q})^{-}_\mathsf{K})$, where $\mb{Z}(\bb{\Q})^{-}_\mathsf{K}$ is the closure of $\mb{Z}(\bb{Q})\cap\mathsf{K}$ in $\mathsf{K}$ (see \cite[\S1.5.8]{KSZ}). We consider the infinite-level Shimura variety
\begin{align*}
    \Sh_{\eK^p}\gx_E \defeq \varprojlim_{\eK_p' \subseteq \eK_p} \Sh_{\eK_p'\eK^p}\gx_E,
\end{align*}
which exists as an $E$-scheme as each $\pi_{\mathsf{K}',\mathsf{K}}$ is affine (see \stacks{01YX}), and forms a $\mathsf{K}_p/\mb{Z}(\Q)_\mathsf{K}^{-}$-torsor for the pro-etale site (in the sense of \cite{BhattScholzeProetale}) on $\Sh_{\mathsf{K}}(\mb{G},\mb{X})_E$.

The map $\mathsf{K}_p\to\mc{G}^c(\Z_p)$ factorizes through $\mathsf{K}_p/\mb{Z}(\Q)_\mathsf{K}^{-}$ (see \cite[\S1.5.8]{KSZ}). Thus, as in \cite[\textsection 4.2]{DanielsToral} or \cite[\S2.1.5 and \S 4.3]{ImaiKatoYoucis}, from the $\mathsf{K}_p/\mb{Z}(\bb{Q})_\mathsf{K}^{-}$-torsor
\begin{align*}
	\Sh_{\eK^p}(\eG,X)_E \to \Sh_{\eK_p\eK^p}(\eG,X)_E,
\end{align*}
one constructs a $\cG^c(\bZ_p)$ torsor $\bP_\eK$ on the pro-\'etale site (see \cite{Scholze2013}) of $\Sh_\eK(\eG,X)^\an_E$.

\begin{prop}[{\cite[Proposition 4.1.4]{PR2021} and \cite[Proposition 4.4]{DanielsToral}}]\label{prop:generic-fiber-shtuka}
    There exists a $\cG^c$-shtuka $\ms{P}_{\eK,E}$ over $\Sh_\eK\gx_E^\lozenge \to \Spd(E)$ with one leg bounded by $\mbbmu_h^c$ which is associated to $\bP_\eK$ in the sense of \textup{\cite[\textsection 2.5]{PR2021}}. Moreover, if $\mathsf{K}'=\mathsf{K}_p{\mathsf{K}^p}'$ and $g$ is in $\mb{G}(\A_f^p)$ is such that $g^{-1}\mathsf{K}^p g\subseteq {\mathsf{K}^p}'$, there exists compatible isomorphisms
    \begin{equation}\label{eq:pullback-compatability}
        t_{\mathsf{K},\mathsf{K}'}(g)^\ast(\ms{P}_{\mathsf{K}',E})\isomto \ms{P}_{\mathsf{K},E}.
    \end{equation}
\end{prop}

\medskip

\subsubsection{Integral local Shimura varieties}
Let us define (see \cite[Definition 24.1.1]{SW2020}), a \emph{parahoric local Shimura datum} to be a triple $(\mc{G},b,\mu)$ where
\begin{itemize}
\item $\mc{G}$ is a parahoric group $\Z_p$-scheme with generic fiber $G$,
\item $\mbbmu$ is a conjugacy class of minuscule cocharacters of $G_{\ov{\bb{Q}}_p}$,
\item and $b$ is an element of $G(\breve{\bQ}_p)$ inducing an element of $B(G,\mbbmu^{-1})$ (see \cite[Definition 2.3]{RV2014}).
\end{itemize}
The reflex field of $(\mc{G},b,\mu)$, usually denoted $E$, is the field of definition of $\mbbmu$. A morphism $f\colon (\mc{G}_1,b_1,\mbbmu_1)\to (\mc{G},b,\mbbmu)$ of parahoric local Shimura datum is is a morphism of $\Z_p$-group schemes $f\colon \mc{G}_1\to \mc{G}$ carrying $\mbbmu_1$ to $\mbbmu$ and $b_1$ to $b$. If such a morphism exists, then $E_1\supseteq E$. We say that $f$ is an \emph{ad-isomorphism} if $f\colon G_1\to G$ is an ad-isomorphism.

Given a parahoric local Shimura datum $(\mc{G},b,\mbbmu)$ with reflex field $E$, we obtain a presheaf
\begin{equation*}
    \mc{M}_{\cG,b,{\mbbmus}}^\mr{int}\colon \mathbf{Perf}_{\mc{O}_{\breve{E}}}\to \mathbf{Set}
\end{equation*}
assigning to any perfectoid space $S$ in characteristic $p$ the set of isomorphism classes of tuples $(S^\sharp, \sP, \phi_\sP, i_r)$, where:
\begin{itemize}
    \item $S^\sharp$ is the untilt of $S$ over $\mc{O}_{\breve{E}}$ associated with $S \to \Spd(\mc{O}_{\breve{E}})$,
    \item $(\sP,\phi_\sP)$ is a $\cG$-shtuka on $S$ with one leg along $S^\sharp$ which is bounded by $\mbbmu$ (see \cite[Definition 2.4.3]{PR2021}), 
    \item and $i_r$ is an isomorphism of $\cG$-torsors $\cG \res_{\mc{Y}_{[r,\infty)}(S)} \isomto \sP \res_{\mc{Y}_{[r,\infty)}(S)}$ for large enough $r$, under which $\phi_\sP$ is identified with $b \times \mr{Frob}_S$ (see \cite[\textsection 25.1]{SW2020}).
\end{itemize}
By \cite[\textsection 25.1]{SW2020} and \cite[Proposition 2.23]{Gleason2020}, the presheaf $\mc{M}_{\cG,b,{\mbbmus}}^\mr{int}$ is a small $v$-sheaf (in the sense of \cite[Definition 12.1]{Scholze2017}), which is called the integral local Shimura variety associated to the local Shimura datum $(\mc{G},b,\mbbmu)$.

If $k_E$ denotes the residue field of $E$, then by \cite[Proposition 2.30]{Gleason2022}, we have a natural identification
\begin{equation}\label{eq:int-sht-special-fiber}
     \mc{M}_{\cG,b,{\mbbmus}}^\mr{int}(\Spd(\bar{k}_E)) = X_\cG(b,\mbbmu^{-1})(\bar{k}_E),
\end{equation}
where $X_\cG(b,\mbbmu^{-1})$ denotes the affine Deligne-Lusztig variety associated to the triple $(\cG, b,\mbbmu^{-1})$ (see \cite[Definition 3.3.1]{PR2021}). For any point $x$ of $\mc{M}^\mr{int}_{\cG,b,\mbbmus}(\Spd(\bar{k}_E))$ we write $(\mc{M}^\mr{int}_{\cG,b,{\mbbmus}})^\wedge_{/x}$ for the formal completion of $ \mc{M}_{\cG,b,{\mbbmus}}^\mr{int}$ at $x$ in the sense of \cite{Gleason2022}.  

\medskip

\subsubsection{The conjecture of Pappas and Rapoport}
We now state the Pappas--Rapoport conjecture. For details, see \cite[\textsection 4.2]{PR2021} and \cite[\textsection 4.3]{DanielsToral}.

Suppose we have a normal flat $\mc{O}_E$-model $\sS_\eK$ of $\Sh_\eK\gx_E$, equipped with an extension $\sP_\eK$ of the $\cG^c$-shtuka $\sP_{\eK,E}$ to a $\cG^c$-shtuka defined over $\sS_\eK^{\lozenge /}$. Note that $\sP_\eK$ is necessarily bounded by $\mbbmu_h^c$ by \cite[Lemma 2.1]{DanielsToral}. 

For any point $x$ of $\sS_\eK(\bar{k}_E)$, the pullback $x^\ast \sP_\eK$ defines a $\cG^c$-shtuka over $\Spd(\bar{k}_E)$. In turn, by \cite[Example 2.4.9]{PR2021}, $x^\ast \sP_\eK$ determines a pair $(\mc{P}_x, \phi_x)$ consisting of a $\cG^c$-torsor $\mc{P}_x$ over $\Spec(W(\bar{k}_E))$ along with an isomorphism
\begin{align*}
    \phi_x: \phi^\ast (\mc{P})[1/p] \isomto \mc{P}_x[1/p].
\end{align*}
Here $\phi$ denotes the Frobenius homomorphism for $W(\bar{k}_E)$. After choosing a trivialization of $\mc{P}_x$, we obtain an element $b_x$ in $G^c(\brqp)$, and changing the trivialization will change $b_x$ to some $\sigma$-conjugate of $b_x$. Therefore we obtain a well-defined conjugacy class $[b_x]$ in $B(G^c)$. 

Since the shtuka $\sP_\eK$ is bounded by $\mbbmu_h^c$, the same holds for $(\mc{P}_x, \phi_x)$, so the pair $(\mc{P}_x,\id)$ determines a point $x_0$ of $X_\cG^c(b_x,(\mbbmu^c_h)^{-1})(\bar{k}_E)$. Hence $X_\cG(b_x,\mbbmu_h^{-1})(\bar{k}_E)$ is nonempty, so by \cite[Theorem A]{HeKRConj}, it follows that $[b_x]$ belongs to $B(G^c, (\mbbmu_h^c)^{-1})$. Then the triple $(\mc{G}^c, b_x, \mbbmu_h^c)$ defines a parahoric local Shimura datum, and we can consider the integral local Shimura variety $\mc{M}^\mr{int}_{\cG^c,b_x,{\mbbmus}_h^c}$ along with the base point $x_0$ (which makes sense by \eqref{eq:int-sht-special-fiber}).

\begin{defn}[{\cite[Conjecture 4.2.2]{PR2021} and \cite[Conjecture 4.5]{DanielsToral}}]\label{Def:PRAxioms}
    Consider a system $\{\sS_{\eK}\gx\}_{\eK^p}$ of normal flat $\mc{O}_E$-models $\sS_\eK$ of $\Sh_\eK\gx$ for $\eK = \eK_p\eK^p$ with $\eK^p$ varying over all sufficiently small compact open subgroups of $\mb{G}(\bA_f^p)$.
    
    We say the system $\{\sS_\eK\gx\}_{\eK^p}$ is an \textit{canonical integral model for} $\{\Sh_{\eK}\gx\}_{\eK^p}$ if the following properties are satisfied.
    \begin{enumerate}[(i)]
        \item For every discrete valuation ring $R$ of characteristic $(0,p)$ over $\mathcal{O}_E$, 
        \begin{align}
            \Sh_{\eK_p}\gx(R[1/p]) = \left(\varprojlim_{\eK^p} \sS_\eK\gx\right)(R).
        \end{align}
        \item For every sufficiently small $\eK^p \subseteq \eG(\bA_f^p)$ and ${\eK'}^p \subseteq \eG(\bA_f^p)$, and element $g$ of $\eG(\bA_f^p)$ with $g^{-1}{\eK}^pg \subseteq \eK^p$, there are finite \'etale morphisms 
        \begin{equation*}
            t_{\mathsf{K},\mathsf{K}'}(g)\colon  \mathscr{S}_{\eK}\gx \to \mathscr{S}_{\eK'}\gx
        \end{equation*}
        extending the corresponding maps on the generic fiber.
        \item The $\mathcal{G}^c$-shtuka $\mathscr{P}_{\eK,E}$ extends to a $\mathcal{G}^c$-shtuka $\mathscr{P}_\eK$ on $\sS_\eK\gx^{\lozenge /}$ for every sufficiently small $\eK^p \subseteq \mb{G}(\bA_f^p)$.
        \item Consider $x$ in $\sS_\eK\gx(\bar{k}_E)$ with corresponding element $b_x$ of $G(\brqp)$, and let $x_0$ be the natural base point in $\mc{M}_{\cG^c,b_x,{\mbbmus}_h^c}^\mr{int}(\bar{k}_E)$. Then there is an isomorphism of $v$-sheaves
        \begin{align}
            \Theta_x: \left(\mc{M}^{\mr{int}}_{\cG^c,b_x,{\mbbmus_h^c}}\right)^\wedge_{/x_0} \isomto \left(\sS_\eK\gx^\wedge_{/x}\right)^\lozenge,
        \end{align}
        such that $\Theta_x^\ast(\mathscr{P}_\eK)$ agrees with the universal $\mc{G}^c$-shtuka $\mathscr{P}^\mathrm{univ}$ on $\mc{M}_{\cG^c,b_x,{\mbbmus_h^c}}^\mr{int}$.
    \end{enumerate}
\end{defn}

We observe the following uniqueness properties concerning a system of models satisfying the Pappas--Rapoport conjecture.

\begin{prop}[{cf.\@ \cite[Corollary 2.7.10 and Theorem 4.2.4]{PR2021}}]\label{prop:uniqueness} An canonical integral model of $\{\Sh_{\eK}\gx\}_{\eK^p}$ and the $\mc{G}^c$-shtukas $\ms{P}_\mathsf{K}$ on these canonical integral models, are unique up to unique isomorphism $($if they exist$)$.
\end{prop}

The following is then what we refer to as the \emph{Pappas--Rapoport conjecture}.

\begin{conj}[Pappas-Rapoport]\label{conj:PR}
    For any parahoric Shimura datum $(\mc{G},\mb{G},\mb{X})$, there exists an canonical integral model for $\{\Sh_\eK\gx\}_{\eK^p}$.
\end{conj}

Conjecture \ref{conj:PR} is known in the following cases, where we have elected to use our notation from \S\ref{ss:KPZ-construction} for maximal clarity.

\begin{thm} [{\cite[Theorem 4.5.2]{PR2021} and \cite[Corollary 4.1.5]{DvHKZ}\label{thm:PR-conj-for-Hodge-type}}]
    For any parahoric Shimura datum $(\mb{G},\mb{X},\mc{G})$ of Hodge type, the system $\{\ms{S}_{\mathsf{K}}^\vn(\mb{G},\mb{X})\}_{\eK^p}$ satisfies the conditions of the Pappas--Rapoport conjecture.
\end{thm}
\begin{rmk}\label{rem:Hodge}
 In \cite{PR2021}, Pappas and Rapoport prove the conjecture under the stubborn technical assumption that $\eK_p$ is a \textit{stabilizer} parahoric, i.e., where $\eK_p = \cG(\bZ_p)$ for a parahoric group scheme $\cG$ which is also the Bruhat--Tits stabilizer group scheme of a point in the extended Bruhat--Tits building of $G$. This assumption frequently holds; for example every parahoric $\bZ_p$-model of $G$ is of this form if $\pi_1(G)_I$ is torsion free. That says, it eliminates many examples of Shimura varieties of abelian type (e.g., many cases of type $D^\bb{H}$). The theorem was subsequently extended to arbitrary Shimura data of Hodge type in \cite{DvHKZ}.
\end{rmk}

\begin{thm}[{\cite[Theorem A]{DanielsToral}}]\label{thm:PR-conj-for-toral-type}
    If $(\mb{G},\mb{X},\mc{G})$ is a parahoric Shimura datum of toral type, then $\{\ms{S}_{\mathsf{K}}(\mb{G},\mb{X})\}_{\eK^p}$ satisfies the conditions of the Pappas--Rapoport conjecture.
\end{thm}
\begin{proof}
    This follows from \cite[Theorem A]{DanielsToral} along with Proposition \ref{prop:toral-type}.
\end{proof}

\begin{thm}[{\cite[Proposition 5.35]{ImaiKatoYoucis}}] If $(\mb{G},\mb{X},\mc{G})$ is a parahoric Shimura datum of abelian type where $\mc{G}$ is reductive, then the canonical integral models $\{\ms{S}_{\mathsf{K}}^{\mathsmaller{\mr{Kis}}}(\mb{G},\mb{X})\}_{\eK^p}$ from \emph{\cite{Kisin2010}} satisfy the conditions of the Pappas--Rapoport conjecture.
\end{thm}

\medskip

\subsection{Statement of the main result and applications}
We now state our main result and give some applications to the study of Kisin--Pappas--Zhou models.

\begin{thm}\label{thm:main} Let $(\mb{G},\mb{X},\mc{G})$ be an acceptable parahoric Shimura data of abelian type, and let $\mf{d}$ be a well-adapted very good parahoric Shimura datum of Hodge type. Then, the system $\{\ms{S}^\mf{d}_{\mathsf{K}}(\mb{G},\mb{X})\}_{\eK^p}$ satisfies the conditions of the Pappas--Rapoport conjecture.
\end{thm}

As an immediate implication, and as an opportunity to reemphasize the mildness of the acceptability condition, we record the following corollary.

\begin{cor} Suppose that $p>3$. Then, the Pappas--Rapoport conjecture holds for all parahoric Shimura data of abelian type.
\end{cor}

In addition, utilizing Proposition \ref{prop:uniqueness}, and the previously known cases of the Pappas--Rapoport conjecture stated above, we can deduce the following omnibus independence result where, again, we are using notation from \S\ref{ss:KPZ-construction}.

\begin{cor}\label{cor:omnibus-independence} Suppose that $(\mb{G},\mb{X},\mc{G})$ is an acceptable parahoric Shimura datum of abelian type. Then, the following statements are true.
\begin{enumerate}[$($1$)$]
    \item The system $\{\ms{S}_{\mathsf{K}}(\mb{G},\mb{X})\}_{\eK^p}$ is canonically independent of the choice of $\mf{d}$.
    \item If $(\mb{G},\mb{X},\mc{G})$ is of Hodge type, then the system $\{\ms{S}^\vn_{\mathsf{K}}(\mb{G},\mb{X})\}_{\eK^p}$ is canonically isomorphic to the system $\{\ms{S}_{\mathsf{K}}^\mf{d}(\mb{G},\mb{X})\}_{\eK^p}$ for any $\mf{d}$.
    \item If $\mc{G}$ is reductive then the system $\{\ms{S}_{\mathsf{K}}^{\mathsmaller{\mr{Kis}}}(\mb{G},\mb{X})\}_{\eK^p}$ is canonically isomorphic to the system $\{\ms{S}^\mf{d}_{\mathsf{K}}(\mb{G},\mb{X})\}_{\eK^p}$ for any $\mf{d}$.
\end{enumerate}
\end{cor}

Given this independence result, the notation $\{\ms{S}_{\mathsf{K}}(\mb{G},\mb{X})\}_{\eK^p}$ and $\ms{S}_{\mathsf{K}_p}(\mb{G},\mb{X})$ for any of the systems of models discussed in \S\ref{ss:KPZ-construction} is unambiguous.

\begin{rmk} One implication of Corollary \ref{cor:omnibus-independence} is the fact that the more naively constructed models $\ms{S}_{\mathsf{K}^p}^\vn(\mb{G},\mb{X})$ possess a local model diagram (e.g., see (3) of \cite[Theorem 5.2.12]{KZ21}) even if \cite[Condition (5.1.11.1)]{KZ21} does not hold. Indeed, this follows by combining \cite[Theorem 5.2.12]{KZ21} and (2) of Corollary \ref{cor:omnibus-independence}.
\end{rmk}

Lastly, we observe that Theorem \ref{thm:main} implies functoriality for the systems $\{\ms{S}_{\mathsf{K}}(\mb{G},\mb{X})\}_{\eK^p}$. 

\begin{thm} \label{thm:functoriality} Suppose that $\alpha\colon (\mb{G}_1,\mb{X}_1,\mc{G}_1)\to (\mb{G},\mb{X},\mc{G})$ is a morphism of acceptable parahoric Shimura data of abelian type. Then, there exists a unique morphism
\begin{equation*}
    \alpha\colon \ms{S}_{\mathsf{K}_{p,1}}(\mb{G}_1,\mb{X}_1)\to \ms{S}_{\mathsf{K}_p}(\mb{G},\mb{X})_{\mc{O}_{E_1}},
\end{equation*}
which is equivariant for the map $\alpha\colon \mb{G}_1(\A_f^p)\to \mb{G}(\A_f^p)$, and whose generic fiber recovers the map
\begin{equation*}
    \alpha\colon \Sh_{\mathsf{K}_{p,1}}(\mb{G}_1,\mb{X}_1)_{E_1}\to \Sh_{\mathsf{K}_p}(\mb{G},\mb{X})_{E_1}.
\end{equation*}
\end{thm}
\begin{proof} 
This follows from the arguments of \cite[Corollary 4.3.2]{PR2021}, and \cite[Corollary 4.0.9]{DvHKZ}. We replicate the proof here for completeness.

We first observe, as in \cite[Proof of Corollary 4.3.2]{PR2021}, that the product
\begin{equation*}
    \ms{S}'' \defeq \ms{S}_{\mathsf{K}_{p,1}}(\mb{G}_1,\mb{X}_1)\times_{\Spec \mc{O}_{E_1}} \ms{S}_{\mathsf{K}_p}(\mb{G},\mb{X})_{\mc{O}_{E_1}}
\end{equation*}
determines an integral model of $\Sh_{\eK_{p,1}\times\eK_p}(\mb{G}_1 \times \mb{G}, \mb{X}_1\times\mb{X})$ which satisfies Conjecture \ref{conj:PR}. Then by \cite[Theorem 4.3.1]{PR2021}, the graph of $\alpha$ determines a unique morphism of integral models $\ms{S}_{\eK_{p,1}}(\mb{G}_1, \mb{X}_1) \to \ms{S}''$ which extends the graph of $\alpha$ on the generic fiber. We obtain the desired morphism by composition with the projection $\ms{S}'' \to \ms{S}_{\eK_p}(\mb{G},\mb{X})_{E_1}$.
\end{proof}

\begin{rmk} We remark that Theorem \ref{thm:functoriality} is also proved in \cite{vHS}, by establishing a modified version of Conjecture \ref{conj:PR} using $\mc{G}^\ad$ in place of $\mc{G}^c$; see Remark \ref{rem:vHS}. 
\end{rmk}


\section{The proof of the main result}
This final section is devoted to the proof of our main theorem (Theorem \ref{thm:main}). This will require 
the development of multiple ancillary concepts and results that we hope will be useful in other contexts.  We fix notation as at the beginning of \S\ref{s:PR-conj}.

\subsection{An auxiliary Kisin--Pappas--Zhou model}\label{ss:auxiliary-Shim-datum}

A useful heuristic for Shimura data of abelian type is that they are `spanned' by two (largely orthogonal) extremes: Shimura data of Hodge type and Shimura data of toral type. Finding a precise formalism to realize this heuristic, especially when their integral models are part of the picture, is somewhat subtle. That said, this was achieved at hyperspecial level by Lovering in \cite{Lovering}. The goal of this subsection is to extend this formalism to the Kisin--Pappas--Zhou models of \S\ref{s:Kisin-Zhou-models}.

\medskip

\subsubsection{The Lovering construction}

We begin by recalling the underlying Shimura datum as defined in \cite[\S4.6]{Lovering}. Let us fix an extension $\mb{E}'$ of $\mb{E}$. Define then the group
\begin{equation*}
\mb{B}_{\mb{E}'}\defeq \mb{G}\times_{\delta,\mb{G}^\ab,r}\mb{T}_{\mb{E}'},
\end{equation*}
where $\mb{T}_{\mb{E}'}\defeq \Res_{\mb{E}'/\Q}\,\, \bb{G}_{m,\mb{E}'}$, $\delta$ is the canonical map $\mb{G}\to\mb{G}^\ab$, and $r$ is obtained as the composition
\begin{equation*}
    \Res_{\mb{E}'/\Q}\,\,\bb{G}_{m,\mb{E}'}\xrightarrow{\Res_{\mb{E}'/\Q}\,\,\mu_h}\Res_{\mb{E}'/\Q}\,\mb{G}_{\mb{E}'}^\ab\xrightarrow{N}\mb{G}^\ab.
\end{equation*}
Here $\mu_h\colon \bb{G}_{m,\mb{E}'}\to \mb{G}^\ab_{\mb{E}'}$ denotes the composition $\delta\circ \mu$ for any element $\mu$ of $\mbbmu_h$, which is defined over $\mb{E}'$ by the definition of reflex field and independent of choice since $\mb{G}^\ab_{\bb{C}}$ is abelian, and $N$ denotes the natural norm map. Observe that $\mb{B}_{\mb{E}'}^\der=\mb{G}^\der$, so we have a short exact sequence
\begin{equation*}
    1\to \mb{G}^\der\to \mb{B}_{\mb{E}'}\to \mb{T}_{\mb{E}'}\to 1.
\end{equation*}
Furthermore, we have an exact sequence
\begin{equation*}
    1\to \mb{B}_{\mb{E}'}\to \mb{G}\times\mb{T}_{\mb{E}'}\xrightarrow{(x,y)\mapsto\delta^{-1}(x)r(y)} \mb{G}^\ab\to 1,
\end{equation*}
and thus $\mb{B}_{\mb{E}'}$ is reductive (see \cite[Corollary 21.53]{Milne2017}).

Let $\{\tau\}$ be the set of archimedean places of $\mb{E}'$, and let $\{[\tau]\}$ denote the set of equivalences classes under the relation $\tau\sim \ov{\tau}$. Let $\tau_0\colon \mb{E}'\to\bb{C}$ be the natural inclusion which defines an element $[\tau_0]$ of $\{[\tau]\}$. Then, we have that $\mb{E}'\otimes_\Q\bb{R}$ is isomorphic to $\prod_{[\tau]}\mb{E}'_{[\tau]}$, where $\mb{E}'_{[\tau]}$ is $\bb{R}$ or $\bb{C}$ if $\tau$ is real or complex, respectively. We define 

\begin{equation*}
    h_\mb{T}\colon \bb{S}\to (\mb{T}_{\mb{E}'})_\bb{R}=\prod_{[\tau]} \Res_{\mb{E}'_\tau/\bb{R}}\,\,\bb{G}_{m,\mb{E}'_{[\tau]}},
\end{equation*}
to be the unique morphism having trivial projection $h_{\mb{T}_{\mb{E}'},[\tau]}$ to every $[\tau]$ not equal to $[\tau_0]$, and satisfying
\begin{equation*}
    h_{\mb{T}_{\mb{E}'},[\tau_0]}(z)=\begin{cases}z\ov{z} & \mbox{if}\quad \tau_0\text{ is real},\\ z & \mbox{if} \quad \tau_0\text{ is complex},\end{cases}
\end{equation*}
for $z$ in $\bb{S}(\bb{R})=\bb{C}^\times$.

One may show (see \cite[Proposition 4.6.5]{Lovering}) that for any element $h$ of $\mb{X}$, the pair
\begin{equation}\label{eq:B-maps-to-prod}
    (h,h_{\mb{T}_{\mb{E}'}})\colon \bb{S}\to \mb{G}_{\bb{R}}\times (\mb{T}_{\mb{E}'})_{\bb{R}},
\end{equation}
factorizes through $(\mb{B}_{\mb{E}'})_{\bb{R}}$, that the resulting $\mb{B}_{\mb{E}'}(\bb{R})$-conjugacy class $\mb{Y}_{\mb{E}'}$ is independent of the choice of $h$, and that the pair $(\mb{B}_{\mb{E}'},\mb{Y}_{\mb{E}'})$ is a Shimura datum with reflex field $\mb{E}'$. Moreover, there is a commutative diagram of Shimura data:
\begin{equation}\label{eq:comm-diag-Lovering-data}
  \begin{tikzcd}
	{(\mb{B}_{\mb{E}'},\mb{Y}_{\mb{E}'})} & {(\mb{T}_{\mb{E}'},\{h_{\mb{T}}\})} \\
	{(\mb{G},\mb{X})} & {(\mb{G}^\ab,\mb{X}^\ab),}
	\arrow[from=1-1, to=1-2]
	\arrow[from=1-1, to=2-1]
	\arrow[from=2-1, to=2-2]
	\arrow[from=1-2, to=2-2]
\end{tikzcd}
\end{equation}
where $\mb{B}_{\mb{E}'}\to \mb{G}$, $\mb{B}_{\mb{E}'}\to \mb{T}_{\mb{E}'}$, and $\mb{G}\to\mb{G}^\ab$ are the natural maps, and the map $\mb{T}_{\mb{E}'}\to \mb{G}^\ab$ is the norm map.  We call $(\mb{B}_{\mb{E}'},\mb{Y}_{\mb{E}'})$ the \emph{Lovering construction} applied to $(\mb{G},\mb{X})$. When $\mb{E}'=\mb{E}$, we simplify $\mb{T}_\mb{E}$, $\mb{B}_\mb{E}$, and $\mb{Y}_\mb{E}$ to $\mb{T}$, $\mb{B}$, and $\mb{Y}$, respectively.


We now record the following pleasant functoriality property of the Lovering construction.

\begin{prop}[{\cite[Lemma 4.6.6]{Lovering}}]\label{prop:Lovering-construction-functoriality} Let $(\mb{G}_1,\mb{X}_1)$ be another Shimura datum and let $\mb{E}'$ be a field containing the compositum $\mb{E}_1\mb{E}$. Suppose that $f\colon \mb{G}_1^\der\to\mb{G}^\der$ is an isogeny which induces an isomorphism $f^\mr{ad}\colon (\mb{G}_1^\ad,\mb{X}_1^\ad)\isomto (\mb{G}^\ad,\mb{X}^\ad)$. Then there exists a unique map $g\colon \mb{B}_{1,\mb{E}'}\to\mb{B}$ filling in the following diagram
\begin{equation*}
   \begin{tikzcd}
	1 & {\mb{G}_1^\der} & {\mb{B}_{1,\mb{E}'}} & {\mb{T}_{1,\mb{E}'}} & 1 \\
	1 & {\mb{G}^\der} & {\mb{B}} & {\mb{T}} & 1,
	\arrow[from=1-1, to=1-2]
	\arrow[from=1-2, to=1-3]
	\arrow[from=1-3, to=1-4]
	\arrow[from=1-4, to=1-5]
	\arrow[from=2-1, to=2-2]
	\arrow[from=2-2, to=2-3]
	\arrow[from=2-3, to=2-4]
	\arrow[from=2-4, to=2-5]
	\arrow["f"', from=1-2, to=2-2]
	\arrow["g"', from=1-3, to=2-3]
	\arrow["N"', from=1-4, to=2-4]
\end{tikzcd}
\end{equation*}
where $N$ is the natural norm map. Moreover, $g$ induces a morphism of Shimura data $(\mb{B}_{1,\mb{E}'},\mb{Y}_{1,\mb{E}'})\to (\mb{B},\mb{Y})$. 
\end{prop}

\subsubsection{Relationship to Kisin--Pappas--Zhou models}

We now explain how the Lovering construction fits into the theory of Kisin--Pappas--Zhou models.

To begin, let us fix a very good parahoric Shimura of Hodge type $\mf{d}=(\mb{G}_1,\mb{X}_1,\mc{G}_1)$ well-adapted to $(\mb{G},\mb{X},\mc{G})$, and let us explicitly denote by $f$ the implicitly defined isogeny $\mb{G}_1^\der\to\mb{G}^\der$. We further set $\mb{E}'=\mb{E}\mb{E}_1$.

Consider the Lovering constructions $(\mb{B},\mb{Y})$ and $(\mb{B}_{1,\mb{E}'},\mb{Y}_{1,\mb{E}'})$. For notational simplicitly we shorten $(\mb{B}_{1,\mb{E}'},\mb{Y}_{1,\mb{E}'})$ to $(\mb{B}_1',\mb{Y}_1')$, and similarly for other attendant notation. Observe that both of these Shimura data are of abelian type and, in fact, $(\mb{G}_1,\mb{X}_1)$ is an associated very good parahoric Shimura datum of Hodge type adapted to both. Indeed, this follows since
\begin{equation}\label{eq:B-der-ident}
    {\mb{B}'_1}^{\der}=\mb{G}_1^\der \ \ \text{ and } \ \  \mb{B}^\der=\mb{G}^\der,
\end{equation}
and so we may consider the isogenies
\begin{equation*}
    \mb{G}_1^\der\xrightarrow{\mr{id}}{\mb{B}'_1}^{\der} \ \ \text{ and } \ \  \mb{G}_1^\der\xrightarrow{f}\mb{B}^\der,
\end{equation*}
which produce the desired isomorphisms on adjoint Shimura data (cf.\@ Proposition \ref{prop:ad-iso-equiv}).

We would like to upgrade this setup to the level of parahoric Shimura data. To this end, observe that by \eqref{eq:B-der-ident} and Lemma \ref{lemma:buildings-ad-isom} one deduces that there are natural identifications
\begin{equation*}
    \ms{B}(B,\bb{Q}_p)\isomto \ms{B}(G\times T,\bb{\Q}_p),\ \ \text{ and } \ \ \ms{B}(B_1',\bb{Q}_p)\isomto \ms{B}(G_1\times T_1',\bb{\Q}_p).
\end{equation*}
Thus, we may associate parahoric models $\mc{B}$ and $\mc{B}_1'$ to our choice of parahoric models $\mc{G}$ and $\mc{G}_1$, parahoric models, respectively. In fact, it follows from Proposition \ref{prop:fiber-product-parahoric}
\begin{equation*}
    \mc{B}'=\mc{G}\times_{\mc{G}^\ab}\mc{T},\ \ \text{ and } \ \  \mc{B}_1'=\mc{G}_1\times_{\mc{G}_1^\ab}\mc{T}_1',
\end{equation*}
where $\mc{T}$ and $\mc{T}_1'$ are the unique parahoric models of $T$ and $T_1'$, respectively.

Now, applying \eqref{eq:comm-diag-Lovering-data} and Proposition \ref{prop:Lovering-construction-functoriality} one produces a diagram of Shimura data
\begin{equation*}
 \begin{tikzcd}
	{(\mb{T}_1',\{h_{\mb{T}_1}\})} & {(\mb{G}_1^\ab,\mb{X}_1^\ab)} \\
	{(\mb{B}_1',\mb{Y}'_1)} & {(\mb{G}_1,\mb{X}_1)} \\
	{(\mb{B},\mb{Y})} & {(\mb{G},\mb{X}).}
	\arrow[from=1-1, to=1-2]
	\arrow[from=2-1, to=2-2]
	\arrow[from=2-1, to=1-1]
	\arrow[from=2-2, to=1-2]
	\arrow[from=2-1, to=3-1]
	\arrow[from=3-1, to=3-2]
\end{tikzcd}
\end{equation*}
In particular, we obtain a diagram of $\mb{E}'$-schemes
\begin{equation}\label{eq:big-Sh-diag}
  \begin{tikzcd}
	{\Sh_{\mathsf{M}'_{p,1}}(\mb{T}_1',\{h_{\mb{T}'_1}\})_{\mb{E}'}} & {\Sh_{\mathsf{N}_{p,1}}(\mb{G}_1^\mr{ab},\mb{X}_1^\mr{ab})_{\mb{E}'}} \\
	{\Sh_{\mathsf{L}'_{p,1}}(\mb{B}_1',\mb{Y}'_1)_{\mb{E}'}} & {\Sh_{\mathsf{K}_{p,1}}(\mb{G}_1,\mb{X}_1)_{\mb{E}'}} \\
	{\Sh_{\mathsf{L}_{p}}(\mb{B},\mb{Y})_{\mb{E}'}} & {\Sh_{\mathsf{K}_{p}}(\mb{G},\mb{X})_{\mb{E}'}}
	\arrow[from=2-1, to=2-2]
	\arrow[from=2-1, to=3-1]
	\arrow[from=3-1, to=3-2]
	\arrow[from=2-2, to=1-2]
	\arrow[from=2-1, to=1-1]
	\arrow[from=1-1, to=1-2]
\end{tikzcd}
\end{equation}
each map equivariant for the natural map of groups. Here, we write $\mathsf{L}_p=\mc{B}(\Z_p)$ and $\mathsf{L}'_{p,1}=\mc{B}_1'(\Z_p)$, $\mathsf{M}'_{p,1}=\mc{T}_1'(\Z_p)$, and $\mathsf{N}_{p,1}=\mc{G}_1^\ab(\Z_p)$. We then have the following relationship between the Lovering construction and Kisin--Pappas--Zhou models.

\begin{prop}\label{prop:big-integral-model-diagram} There exists a unique diagram
\begin{equation*}
\begin{tikzcd}
	{\ms{S}_{\mathsf{M}'_{p,1}}(\mb{T}_1',\{h_{\mb{T}'_1}\})_{\mc{O}_E}} & {\ms{S}_{\mathsf{N}_{p,1}}(\mb{G}_1^\mr{ab},\mb{X}_1^\mr{ab})_{\mc{O}_E}} \\
	{\ms{S}^\mf{d}_{\mathsf{L}'_{p,1}}(\mb{B}'_1,\mb{Y}'_1)_{\mc{O}_E}} & {\ms{S}^\vn_{\mathsf{K}_{p,1}}(\mb{G}_1,\mb{X}_1)_{\mc{O}_E}} \\
	{\ms{S}^\mf{d}_{\mathsf{L}_{p}}(\mb{B},\mb{Y})} & {\ms{S}^\mf{d}_{\mathsf{K}_{p}}(\mb{G},\mb{X})}
	\arrow[from=2-1, to=2-2]
	\arrow[from=2-1, to=3-1]
	\arrow[from=3-1, to=3-2]
	\arrow[from=2-2, to=1-2]
	\arrow[from=2-1, to=1-1]
	\arrow[from=1-1, to=1-2]
\end{tikzcd}
\end{equation*}
whose maps are equivariant for the appropriate maps of groups, and which models \eqref{eq:big-Sh-diag} pulled back to the completion of $\mb{E}'$ at a prime above $v$. 
\end{prop}
\begin{proof} The existence of an equivariant map 
\begin{equation*}
    {\ms{S}_{\mathsf{M}'_{p,1}}(\mb{T}'_1,\{h_{\mb{T}'_1}\})_{\mc{O}_E}}\to {\ms{S}_{\mathsf{N}_{p,1}}(\mb{G}_1^\mr{ab},\mb{X}_1^\mr{ab})_{\mc{O}_E}}
\end{equation*}
modeling the map on the generic fiber follows from Proposition \ref{prop:toral-type}, while the existence of the equivariant map
\begin{equation*}
    {\ms{S}^\vn_{\mathsf{K}_{p,1}}(\mb{G}_1,\mb{X}_1)_{\mc{O}_E}}\to {\ms{S}_{\mathsf{N}_{p,1}}(\mb{G}_1^\mr{ab},\mb{X}_1^\mr{ab})_{\mc{O}_E}}
\end{equation*}
modeling the map on the generic fiber follows from Proposition \ref{prop:funct-for-ab}. To construct the equivariant maps 
\begin{equation*}
    {\ms{S}^\mf{d}_{\mathsf{L}'_{p,1}}(\mb{B}'_1,\mb{Y}'_1)_{\mc{O}_E}} \to  {\ms{S}^\vn_{\mathsf{K}_{p,1}}(\mb{G}_1,\mb{X}_1)_{\mc{O}_E}},\ \text{ and } \ {\ms{S}^\mf{d}_{\mathsf{L}'_{p,1}}(\mb{B}'_1,\mb{Y}'_1)_{\mc{O}_E}}\to {\ms{S}_{\mathsf{M}'_{p,1}}(\mb{T}'_1,\{h_{\mb{T}'_1}\})_{\mc{O}_E}},
\end{equation*}
modeling the maps on the generic fiber, it suffices by Proposition \ref{prop:KPZ-for-products} and Proposition \ref{prop:HT-basic-model-comparison} to construct a map 
\begin{equation*}
    {\ms{S}^\mf{d}_{\mathsf{L}'_{p,1}}(\mb{B}'_1,\mb{Y}'_1)_{\mc{O}_E}} \to \ms{S}_{\mathsf{K}_{p,1}\times \mathsf{M}'_{p,1}}(\mb{G}_1\times\mb{T}'_1,\mb{X}_1\times\{h_{\mb{T}'_1}\})_{\mc{O}_E}.
\end{equation*}
Because the map $\mb{B}'_1\to \mb{G}_1\times\mb{T}'_1$ is an isomorphism on derived subgroups, we see from Proposition \ref{prop:ad-iso-equiv} that it is an ad-isomorphism, and so the claim follows from Proposition \ref{prop:funct-for-ad-isom}. The fact that this constructed square commutes may be checked on the generic fiber, where it is true by construction. 

The existence of an equivariant map 
\begin{equation*}
    {\ms{S}^\mf{d}_{\mathsf{L}'_{p,1}}(\mb{B}'_1,\mb{Y}'_1)_{\mc{O}_E}}\to {\ms{S}^\mf{d}_{\mathsf{L}_{p}}(\mb{B},\mb{Y})}
\end{equation*}
follows from Proposition \ref{prop:funct-for-ad-isom} as the induced map ${\mb{B}'_1}^{\der}\to\mb{B}^\der$ is identified with the isogeny $f\colon \mb{G}_1^\der\to\mb{G}$ and so Proposition \ref{prop:ad-iso-equiv} applies. We similarly deduce the existence of a map
\begin{equation*}
    {\ms{S}^\mf{d}_{\mathsf{L}_{p}}(\mb{B},\mb{Y})}\to {\ms{S}^\mf{d}_{\mathsf{K}_{p}}(\mb{G},\mb{X})},
\end{equation*}
using the fact that $\mb{B}^\der\to\mb{G}^\der$ is an isomorphism, along with Proposition \ref{prop:funct-for-ad-isom} and Proposition \ref{prop:ad-iso-equiv}.
\end{proof}

\subsection{Construction of the $\mc{G}^c$-shtuka}\label{s:construction-of-shtuka} In this section we establish the following theorem.

\begin{thm}\label{prop:construction-of-shtuka} There exists a unique $\mc{G}^c$-shtuka $\ms{P}_{\mathsf{K}_p}$ on $\ms{S}_{\mathsf{K}_p}^\mf{d}(\mb{G},\mb{X})$ bounded by $\mbbmu_h^c$, which is compatible with the $\mb{G}(\A_f^p)$-action and which models $\ms{P}_{\mathsf{K}_p,E}$.
\end{thm}

By \cite[Corollary 2.7.10]{PR2021}, there can be at most one shtuka $\ms{P}_{\eK_p}$ modeling $\ms{P}_{\eK_p, E}$. Moreover, such a shtuka is automatically bounded by $\mbbmu_h^c$ by \cite[Lemma 2.1]{DanielsToral}. Thus, it suffices to construct a $\mc{G}^c$-shtuka $\ms{P}_{\mathsf{K}_p}$ which is compatible with the $\mb{G}(\A_f^p)$-action and which models $\ms{P}_{\mathsf{K}_p,E}$. In fact, it suffices to construct a model with $\mb{G}(\A_f^p)$-action of $\ms{P}_{\mathsf{K}_p,E'}$ for any discrete algebraic extension $E'$ of $E$. Indeed, by \cite[Corollary 2.7.10]{PR2021} one may lift the descent datum of $\ms{P}_{\mathsf{K}_p,E'}$ relative to $\ms{P}_{\mathsf{K}_p,E}$ to an integral descent datum, which is effective by \cite[Proposition 19.5.3]{SW2020}.

The proof of the existence of such a model will occupy the rest of this subsection, and will be carried out in several steps.

\medskip

\subsubsection*{Step 1: Group Theory} \label{subsub:Step-1-Proof} We begin by establishing the following group-theoretic result which will be useful in our construction.

\begin{prop}\label{prop:Bc} There is a canonical identification
\begin{equation*}
    {\mc{B}_1'}^{c}\isomto \mc{G}_1\times_{\mc{G}_1^\ab}{\mc{T}'_1}^{c}.
\end{equation*}
\end{prop}
\begin{proof} Let us begin by verifying that there is a canonical isomorphism 
\begin{equation}\label{eq:rational-c-commutes-with-fiber-products}
    {\mb{B}'_1}^{c}\isomto \mb{G}_1\times_{\mb{G}_1^\ab}{\mb{T}'_1}^{c}.
\end{equation}
Note here that we are using Lemma \ref{lem:Gc-functoriality} to identify $\mb{G}_1$ and $\mb{G}_1^\mr{ab}$ with $\mb{G}_1^c$ and $(\mb{G}_1^\mr{ab})^c$, respectively. 

Observe that since $\mb{B}_1'\to \mb{G}_1\times\mb{T}'_1$ is an ad-isomorphism, we may use Proposition \ref{prop:ad-iso-equiv} to deduce that 
\begin{equation*}
    Z(\mb{B}_1')=Z(\mb{G}_1)\times_{\mb{G}_1^\ab}\mb{T}_1'.
\end{equation*}
We then employ the following simple lemma.

\begin{lemma} Let
\begin{equation*}
    \begin{tikzcd}
	{\mb{T}} & {\mb{T}_1} \\
	{\mb{T}_2} & {\mb{T}_3,}
	\arrow[ from=1-2, to=2-2]
	\arrow[ from=2-1, to=2-2]
	\arrow[ from=1-1, to=2-1]
	\arrow[ from=1-1, to=1-2]
	\arrow["\lrcorner"{anchor=center, pos=0.125}, draw=none, from=1-1, to=2-2]
\end{tikzcd}
\end{equation*}
be a Cartesian diagram, where $\mb{T}_i$, for $i=1,2,3$, are multiplicative groups over $\Q$, and such that $\mb{T}_1^\circ\to \mb{T}_3^\circ$ is surjective. Then, 
\begin{equation*}
    \mb{T}_\mr{ac}=(\mb{T}_1)_\mr{ac}\times_{(\mb{T}_2)_\mr{ac}}(\mb{T}_3)_\mr{ac}.
\end{equation*}
\end{lemma}
\begin{proof}
    This follows easily from Proposition \ref{prop:fiber-product-properties}, and the observation that for a containment $A\subseteq B$ of multiplicative groups over some field $F$, one has that $A_a=(B_a\cap A)^\circ$ and $A_s=(B_s\cap A)^\circ$, where the subscripts $a$ and $s$ denote the maximal anisotropic and split subtori, respectively. Namely, one applies this with $A=\mathbf{T}$, and $B=\mathbf{T}_1\times\mathbf{T}_2$.
\end{proof}

Using this, and the fact that fiber products of groups commute with quotients, we easily deduce the existence of an isomorphism as in \eqref{eq:rational-c-commutes-with-fiber-products}. Note that $\mc{B}_1^c$ is a parahoric group by definition. The same is true of $\mc{G}_1\times_{\mc{G}_1^{\ab}}{\mc{T}'_1}^{c}$. Indeed, this follows from Proposition \ref{prop:fiber-product-parahoric}, since $G_1^\der$ is $R$-smooth (recall that we assume $(\mb{G}_1, \mb{X}_1, \mc{G}_1)$ is very good) and thus Proposition \ref{prop:parahoric-adjoint} applies, implying that $\ker(\mc{G}_1 \to \mc{G}^\ab)$ is $\mc{G}^\der$ which is is smooth with connected fibers. To conclude, it now suffices to show that $\mc{B}_1^c$ and $\mc{G}_1\times_{\mc{G}_1^{\ab}}{\mc{T}'_1}^{c}$ correspond to matching points under the bijection
\begin{equation*}
    \ms{B}({B'_1}^{c},\Q_p)\isomto \ms{B}({G}_1\times_{{G}_1^{\mr{ab}}}{T_1'}^c,\Q_p),
\end{equation*}
coming from \eqref{eq:rational-c-commutes-with-fiber-products}.

We have the following commutative diagram of equivariant bijections (utilizing a combination of Lemma \ref{lemma:buildings-ad-isom} and Proposition \ref{prop:fiber-product-parahoric} several times)
\begin{equation*}
    \begin{tikzcd}
	{\mathscr{B}({G}_1\times_{{G}_1^{\mathrm{ab}}}{T_1'}^c)} & {\mathscr{B}({G}_1\times{T_1'}^c,\mathbb{Q}_p)} & {\mathscr{B}({G}_1\times {T}_1',\mathbb{Q}_p)} \\
	& {\mathscr{B}({B_1'}^c,\mathbb{Q}_p)} & {\mathscr{B}({B}'_1,\mathbb{Q}_p).}
	\arrow[from=2-3, to=1-3]
	\arrow[from=2-3, to=2-2]
	\arrow[from=1-3, to=1-2]
	\arrow[from=1-1, to=1-2]
	\arrow[from=2-2, to=1-1]
\end{tikzcd}
\end{equation*}
Tracing through the definitions, we see that the claim follows as the points we consider ultimately can be induced from the same point $(x,\ast)$ in the top-right corner, where $\mc{G}_1=\mc{G}_x^\circ$. Here we are implicitly using that the formation of the (reduced) Bruhat--Tits building commutes with products.
\end{proof}

\medskip

\subsubsection*{Step 2: Pullback and Descent to the Fiber Product} \label{subsub:Step-2-Proof} Let us now recall the commutative diagram
\begin{equation*}
\begin{tikzcd}
	{\ms{S}_{\mathsf{M}'_{p,1}}(\mb{T}_1',\{h_{\mb{T}'_1}\})_{\mc{O}_E}} & {\ms{S}_{\mathsf{N}_{p,1}}(\mb{G}_1^\mr{ab},\mb{X}_1^\mr{ab})_{\mc{O}_E}} \\
	{\ms{S}^\mf{d}_{\mathsf{L}'_{p,1}}(\mb{B}'_1,\mb{Y}'_1)_{\mc{O}_E}} & {\ms{S}^\vn_{\mathsf{K}_{p,1}}(\mb{G}_1,\mb{X}_1)_{\mc{O}_E}} \\
	{\ms{S}^\mf{d}_{\mathsf{L}_{p}}(\mb{B},\mb{Y})} & {\ms{S}^\mf{d}_{\mathsf{K}_{p}}(\mb{G},\mb{X})}
	\arrow["a", from=2-1, to=2-2]
	\arrow[from=2-1, to=3-1]
	\arrow[from=3-1, to=3-2]
    \arrow["e",from=2-1, to=3-2]
	\arrow["b", from=2-2, to=1-2]
	\arrow["d", from=2-1, to=1-1]
	\arrow["c", from=1-1, to=1-2]
\end{tikzcd}
\end{equation*}
from Proposition \ref{prop:big-integral-model-diagram}, with the arrows labeled for convenience, and where the arrow $e$ is defined to make the lower triangle commute.

Utilizing Theorem \ref{thm:PR-conj-for-Hodge-type} and Theorem \ref{thm:PR-conj-for-toral-type} there exist a $\mb{G}_1(\A_f^p)$-equivariant $\mc{G}_1$-shtuka $\ms{P}_{\mathsf{K}_{p,1}}$ and a $\mb{T}'_1(\A_f^p)$-equivariant ${\mc{T}'_1}^c$-shtuka $\ms{P}_{\mathsf{M}'_{p,1}}$ on ${\ms{S}^\vn_{\mathsf{K}_{p,1}}(\mb{G}_1,\mb{X}_1)_{\mc{O}_E}}$ and ${\ms{S}_{\mathsf{M}'_{p,1}}(\mb{T}'_1,\{h_{\mb{T}'_1}\})_{\mc{O}_E}}$, respectively. We then obtain a $\mc{G}_1$-shtuka and a ${\mc{T}'_1}^{c}$-shtuka on ${\ms{S}^\mf{d}_{\mathsf{L}_{p,1}}(\mb{B}_1',\mb{Y}'_1)_{\mc{O}_E}}$, given by $a^\ast\ms{P}_{\mathsf{K}_{p,1}}$ and $d^\ast\ms{P}_{\mathsf{M}'_{p,1}}$, respectively.

We claim that $a^\ast\ms{P}_{\mathsf{K}_{p,1}}$ and $d^\ast\ms{P}_{\mathsf{M}'_{p,1}}$ are both equivariant relative to $\mb{B}'_1(\A_f^p)$. Indeed, this follows from the equivariance of these maps, as for an element $g$ of $\mb{B}_1'(\A_f^p)$ we have natural isomorphisms
\begin{equation*}
    [g]^\ast a^\ast\ms{P}_{\mathsf{K}_{p,1}}\simeq a^\ast([a(g)]^\ast\ms{P}_{\mathsf{K}_{p,1}})\simeq a^\ast \ms{P}_{\mathsf{K}_{p,1}},
\end{equation*}
using the $\mb{G}_1(\A_f^p)$-equivariance of $\ms{P}_{\mathsf{K}_{p,1}}$, and similarly in the case of $d^\ast\ms{P}_{\mathsf{M}'_{p,1}}$.

Our plan is now to use Corollary \ref{cor:fiber-product-of-shtukas} to glue the shtukas $a^\ast \ms{P}_{\eK_{p,1}}$ and $d^\ast \ms{P}_{\mathsf{M}'_{p,1}}$ in order to obtain a $\mc{B}_1'^c$-shtuka on $\ms{S}^\mf{d}_{\mathsf{L}'_{p,1}}(\mb{B}_1', \mb{Y}_1')_{\mc{O}_E}$ modeling the $\mc{B}_1'^c$-shtuka $\sP_{\mathsf{L}'_{p,1},E}$ living over the generic fiber $\Sh_{\mathsf{L}'_{p,1}}(\mb{B}_1', \mb{Y}_1')_{E'}$, see Proposition \ref{prop:generic-fiber-shtuka}.
\begin{lemma}\label{lem:torsor-isom}
    Let $\delta$ be the natural map $\mc{G}_1\to \mc{G}_1^{{\ab}}$, and let $N\colon {\mc{T}'_1}^{c}\to \mc{G}_1^\ab$ be the $($map induced by$)$ the norm. Then there is a $\mb{B}'_1(\A_f^p)$-equivariant isomorphism of $\mc{G}_1^{{\ab}}$-torsors
\begin{equation*}
    \theta\colon \delta_\ast a^\ast\ms{P}_{\mathsf{K}_{p,1}}\isomto N_\ast d^\ast \ms{P}_{\mathsf{M}_{p,1}}.
\end{equation*}
\end{lemma}

\begin{proof}
To construct such an isomorphism, it suffices construct compatible isomorphisms over $\ms{S}_{{\mathsf{L}'_1}^{p}}^\mf{d}(\mb{B}_1',\mb{Y}'_1)_{\mc{O}_E}$, for each neat compact open subgroup ${\mathsf{L}'_1}^{p}\subseteq \mb{B}'_1(\bb{A}_f^p)$. Furthermore, by \cite[Corollary 2.7.10]{PR2021}, it suffices to construct such compatible isomorphisms on $\Sh_{\mathsf{L}_{p,1}'{\mathsf{L}'_1}^{p}}(\mb{B}_1',\mb{Y}'_1)_E$. 

The shtukas $\ms{P}_{\mathsf{K}_{p,1},E}$ and $\ms{P}_{\mathsf{M}_{p,1}}$ are functorially constructed from the corresponding \'etale realization functors $\bb{P}_{\mathsf{K}_{p,1},E}$ and $\bb{P}_{\mathsf{M}_{p,1}}$, respectively, in a way compatible with pullbacks along morphisms of Shimura varieties and pushforwards along morphisms of groups. Standard compatabilities of \'etale realizations (c.f.\@, \cite[Equation (4.3.2)]{ImaiKatoYoucis}) imply that there are natural isomorphisms 
\begin{equation*}
    \delta_\ast a^\ast\bb{P}_{\mathsf{K}_{p,1}}\simeq (b\circ a)^\ast\bb{P}_{\mathsf{N}_{p,1}}\simeq
    (c\circ d)^\ast \bb{P}_{\mathsf{N}_{p,1}}\simeq N_\ast d^\ast \bb{P}_{\mathsf{M}_{p,1}},
\end{equation*}
from where the claim follows. From this we deduce the existence of canonical isomormorphisms
\begin{equation*}
    \delta_\ast a^\ast\ms{P}_{\mathsf{K}_{p,1}}\simeq  (b\circ a)^\ast\ms{P}_{\mathsf{N}_{p,1}}\simeq
    (c\circ d)^\ast \ms{P}_{\mathsf{N}_{p,1}}\simeq N_\ast d^\ast \ms{P}_{\mathsf{M}_{p,1}},
\end{equation*}
where $\ms{P}_{\mathsf{N}_{p,1}}$ is the unique model of $\ms{P}_{\mathsf{N}_{p,1},E}$ from Theorem \ref{thm:PR-conj-for-toral-type}. The resulting composition $\theta:\delta_\ast a^\ast \ms{P}_{\eK_{p,1}} \isomto N_\ast d^\ast \ms{P}_{\mathsf{M}_{p,1}}$ is the desired isomorphism.
\end{proof}

\begin{prop}
    There exists a $\mc{B}_1'^c$-shtuka $\sP_{\mathsf{L}'_{p,1}}$ on $\ms{S}_{\mathsf{L}'_{p,1}}^\mf{d}(\mb{B}_1', \mb{Y}_1')_{\mc{O}_E}$ which models $\sP_{\mathsf{L}'_{p,1},E}$.
\end{prop}
\begin{proof}

Using Lemma \ref{lem:torsor-isom}, we may apply Corollary \ref{cor:fiber-product-of-shtukas} and Proposition \ref{prop:Bc} to deduce the construction of a $\mc{B}_1'^c$-shtuka $\ms{P}_{\mathsf{L}'_{p,1}}$ on ${\ms{S}^\mf{d}_{\mathsf{L}'_{p,1}}(\mb{B}'_1,\mb{Y}'_1)_{\mc{O}_E}}$. By construction this shtuka has an equivariant action of $\mb{B}_1'(\A_f^p)$. As in the last paragraph, to show $\ms{P}_{\mathsf{L}'_{p,1}}$ models the ${\mc{B}_1'}^c$-shtuka $\ms{P}_{\mathsf{L}'_{p,1},E}$, it suffices to show that there is an isomorphism
\begin{equation*}
    \bb{P}_{\mathsf{K}_{p,1}}\times_\theta \bb{P}_{\mathsf{M}_{p,1}}\simeq \bb{P}_{\mathsf{L}'_{p,1}}.
\end{equation*}
Here we are applying Proposition \ref{prop:torsors-comparison}, which makes sense when one views these local systems as torsors as in \cite[\S2.1.1]{ImaiKatoYoucis}. But, in view of Proposition \ref{prop:Bc} this is an elementary calculation (cf.\@ \cite[\S3.4.4]{LoveringFCrystals}).
\end{proof}

\medskip

\subsubsection*{Step 3: Descending to $\ms{S}_{\mathsf{K}_p}^\mf{d}(\mb{G},\mb{X})_{\mc{O}_{E^\ur}}$} \label{subsub:Step-3-Proof}

We now show that there is a $\mb{G}(\A_f^p)$-equivariant $\mc{G}^c$-shtuka on $\ms{S}_{\mathsf{K}_p}^\mf{d}(\mb{G},\mb{X})_{E^\ur}$ which models $\ms{P}_{\mathsf{K}_p,E^\ur}$. As previously noted, this is sufficient to prove Theorem \ref{prop:construction-of-shtuka}. Let us begin by considering the $\mc{G}^c$-shtuka
\begin{equation*}
    \ms{P}'\defeq \ms{P}_{\mathsf{L}'_{p,1},\mc{O}_{E^\ur}}\times^{{\mc{B}_1'}^c}\mc{G}^c,
\end{equation*}
on $\ms{S}_{{\mathsf{L}'_1}^{p}}^\mf{d}(\mb{B}_1',\mb{Y}'_1)_{\mc{O}_{E^\ur}}$. 

Choose a connected component ${\ms{S}'}^{+}$ of $\ms{S}_{{\mathsf{L}'_1}^p}^\mf{d}(\mb{B}_1',\mb{Y}'_1)_{\mc{O}_{E^\ur}}$, and let $\ms{S}^+$ be the connected component of $\ms{S}_{\mathsf{K}_p}^\mf{d}(\mb{G},\mb{X})_{\mc{O}_{E^\ur}}$ containing its image under the map
\begin{equation}\label{eq:Step-3-map-of-Shim-vars}
    \ms{S}_{{\mathsf{L}'_1}^{p}}^\mf{d}(\mb{B}_1',\mb{Y}'_1)_{\mc{O}_{E^\ur}}\to \ms{S}_{\mathsf{K}_p}^\mf{d}(\mb{G},\mb{X})_{\mc{O}_{E^\ur}}.
\end{equation}
Set ${\ms{P}'}^{+}$ to be the pullback of $\ms{P}^+$ to ${\ms{S}'}^{+}$. 
\begin{lemma}
    The $\mc{G}^c$-shtuka ${\ms{P}'}^{+}$ descends uniquely to a $\mc{G}^c$-shtuka $\ms{P}^+$ on $\ms{S}^+$, which models the restriction of $\ms{P}_{\mathsf{K}_{p},E^\ur}$ to $\ms{S}^+_{E^\ur}$
\end{lemma}

\begin{proof}

As noted in the proof of Proposition \ref{prop:funct-for-ad-isom}, the map ${\ms{S}'}^{+}\to\ms{S}^{+}$ is a quotient by the group
\begin{equation*}
    \Delta=\ker(\ms{A}(\bm{\mc{G}}_1)^\circ\to \ms{A}(\bm{\mc{G}}))/\ker(\ms{A}(\bm{\mc{G}}_1)^\circ\to \ms{A}(\bm{\mc{B}}'_1)),
\end{equation*}
which acts through a finite quotient asn in \cite[\S E.6]{Kisin2010}. Thus, by \cite[Proposition 19.5.3]{SW2020}, it suffices to show that there is a $\Delta$-action on ${\ms{P}'}^{+}$ which gives a descent datum. Arguing as in \hyperref[subsub:Step-2-Construction]{\textbf{Step 2}}, it suffices to verify that there is a $\Delta$-descent datum on the $\mc{G}^c(\Z_p)$-local system ${\bb{P}'}^{+}$ on ${\ms{S}'}^{+}_{E^\ur}$, whose definition is, mutatis mutandis, the same as that of ${\ms{P}'}^{+}$. But, this is clear as the $\mc{G}^c(\Z_p)$-local system $\bb{P}'$ on $\Sh_{\mathsf{L}'_{p,1}}(\mb{B}'_1,\mb{Y}'_1)_{E^\ur}$ is the pullback of a $\mc{G}^c(\Z_p)$-local system on $\Sh_{\mathsf{K}_p}(\mb{G},\mb{X})_{E^\ur}$ by \cite[Equation (4.3.2)]{ImaiKatoYoucis}.

Thus we obtain a $\mc{G}^c$-shtuka $\ms{P}^+$ on $\ms{S}^+$. That this shtuka models the restriction of $\ms{P}_{\mathsf{K}_{p},E^\ur}$ to $\ms{S}^+_{E^\ur}$ is clear by design. Indeed, $\ms{P}^+$ is isomorphic to the $\mc{G}^c$-shtuka obtained as the descent of ${\ms{P}'}^{+}$ with respect to the $\Delta$-action which, as noted before, agrees with that of the restriction to $\ms{S}^+_{E^\ur}$ of the pullback of $\ms{P}_{\mathsf{K}_p,E^\ur}$ along the map in \eqref{eq:Step-3-map-of-Shim-vars}.
\end{proof}

Using the $\ms{A}(\bm{\mc{G}})$-equivariant bijection 
\begin{equation}
    \pi_0(\Sh_{\mathsf{K}_{p}}(\mb{G},\mb{X})_{{E^\ur}})\isomto \pi_0({\ms{S}}^\mf{d}_{\mathsf{K}_{p}}(\mb{G},\mb{X})_{\mc{O}_{E\ur}}),
\end{equation}
(see Lemma \ref{lem:pi0-omnibus} and Lemma \ref{lem:normal-components}) and \cite[Lemma 4.6.13]{KP2018}, the group $\ms{A}(\bm{\mc{G}})$ acts transitively on $\pi_0(\ms{S}_{\mathsf{K}_p}^\mf{d}(\mb{G},\mb{X})_{\mc{O}_{E^\ur}})$. Thus, for any component $\ms{T}^+$ of $\ms{S}_{\mathsf{K}_p}^\mf{d}(\mb{G},\mb{X})_{\mc{O}_{E^\ur}}$ we may choose some $g$ in $\ms{A}(\bm{\mc{G}})$ such that
\begin{equation*}
    [g]\colon \ms{S}^+\isomto \ms{T}^+.
\end{equation*}
We define 
\begin{equation*}
    \ms{P}_{\mathsf{K}_p,\mc{O}_{E^\ur},\ms{T}^+}\defeq [g]^\ast\ms{P}^+.
\end{equation*}
To see that this is canonically independent of the choice of $g$ we again note by \cite[Corollary 2.7.10]{PR2021} that it suffices to verify this on the generic fiber. There, the statement follows from the fact that $\ms{P}^+$ models the restriction of $\ms{P}_{\mathsf{K}_p,E^\ur}$ to $\ms{S}^+_{E^\ur}$, and the $\ms{A}(\bm{\mc{G}})$-equivariance of $\ms{P}_{\mathsf{K}_{p},E^\ur}$. Moreover, by the same logic we see that $\ms{P}_{\mathsf{K}_p,\mc{O}_{E^\ur},\ms{T}^+}$ models  the restriction of $\ms{P}_{\mathsf{K}_{p,1},E^\ur}$ to $\ms{T}^+_{E^\ur}$.

For each neat compact open subgroup $\mathsf{K}^p\subseteq \mb{G}(\A_f^p)$ and each connected component $\ms{T}_{\mathsf{K}^p}^+$ of $\ms{S}_{\mathsf{K}^p}^\mf{d}(\mb{G},\mb{X})_{\mc{O}_{E^\ur}}$ we may choose any connected component $\ms{T}^+$ of $\ms{S}_{\mathsf{K}_p}^\mf{d}(\mb{G},\mb{X})_{\mc{O}_{E^\ur}}$ mapping to it. The fact here that the map
\begin{equation*}
   \pi_0(\ms{S}_{\mathsf{K}_p}^\mf{d}(\mb{G},\mb{X})_{\mc{O}_{E^\ur}})\to  \pi_0(\ms{S}_{\mathsf{K}_p\mathsf{K}^p}^\mf{d}(\mb{G},\mb{X})_{\mc{O}_{E^\ur}})
\end{equation*}
is surjective follows from Lemma \ref{lem:pi0-omnibus}. 
\begin{lemma}
    The shtuka $\ms{P}_{\eK_p, \mc{O}_{E^\ur}, \ms{T}^+}$ descends along the map $\ms{T}^+ \to \ms{T}_{\eK^p}^+$ to a shtuka $\ms{P}_{\eK^p, \mc{O}_{E^\ur},\ms{T}^+_{\eK^p}}$ which models the restriction of $\ms{P}_{\eK^p, E^\ur}$ to $\ms{T}^+_{\eK^p,E^\ur}$. 
\end{lemma}
\begin{proof}
The map $\ms{T}^+\to \ms{T}^+_{\mathsf{K}^p}$ is a profinite \'etale cover, and thus by \cite[Proposition 19.5.3]{SW2020} to show that $\ms{P}_{\mathsf{K}_p,\mc{O}_{E^\ur},\ms{T}^+}$ descends along this map, it suffices to produce descent data. Again by \cite[Corollary 2.7.10]{PR2021}, this data can be obtained using the descent data for the restriction of $\ms{P}_{\mathsf{K}_{p,1},E^\ur}$ to $\ms{T}^+_{E^\ur}$ relative to the map $\ms{T}^+_{E^\ur}\to \ms{T}^+_{\mathsf{K}^p, E^\ur}$. 

It is clear by the definition of $\ms{P}_{\mathsf{K}_p,\mc{O}_{E^\ur},\ms{T}^+}$, \cite[Lemma 2.2.5]{Kisin2010}, and the equivariance of the map $\ms{T}^+\to \ms{T}^+_{\mathsf{K}^p}$ that this descent is independent of all choices, and hence we obtain $\ms{P}_{\mathsf{K}^p,\mc{O}_{E^\ur},\ms{T}^+_{\mathsf{K}^p}}$. We conclude by observing that $\ms{P}_{\mathsf{K}^p,\mc{O}_{E^\ur},\ms{T}^+_{\mathsf{K}^p}}$ models the restriction of $\ms{P}_{\mathsf{K}^p,E^\ur}$ to $\ms{T}^+_{\mathsf{K}^p,E^\ur}$ since, by definition, they have the same descent data relative to the map $\ms{T}^+_{E^\ur}\to \ms{T}^+_{\mathsf{K}^p,E^\ur}$.  
\end{proof}

Finally, we can construct the desired $\mc{G}^c$-shtuka on $\ms{S}_{\eK_p\eK^p}^\mf{d}(\mb{G}, \mb{X})_{\mc{O}_{E^\ur}}$.
\begin{cons}\label{cons:shtuka}

We set
\begin{equation*}
    \ms{P}_{\mathsf{K}^p,\mc{O}_{E^\ur}}=\bigsqcup_{\ms{T}_{\mathsf{K}^p}^+\in\pi_0(\ms{S}_{\eK_p\eK^p}^\mf{d}(\mb{G},\mb{X})_{\mc{O}_{E^\ur}})}\ms{P}_{\mathsf{K}^p,\mc{O}_{E^\ur},\ms{T}_{\mathsf{K}^p}^+}.
\end{equation*}
More precisely, let $k_E$ denote the residue field of $E$. Then, as $\pi_0(\ms{S}_{\eK_p\eK^p}^\mf{d}(\mb{G},\mb{X})_{\mc{O}_{E^\ur}})$ is finite, the set of objects $S$ of $\mb{Perf}_{k_E}$ such that $S\to (\ms{S}_{\eK_p\eK^p}^\mf{d}(\mb{G},\mb{X})_{\mc{O}_{E^\ur}})^{\lozenge /}$ factorizes through some $(\ms{T}_{\mathsf{K}^p}^+)^{\lozenge /}$ is a basis of $\mb{Perf}_{k_E}$, and on such an object we define the $\mc{G}^c$-shtuka over $S$ to be that determined by the (unique) map $S\to (\ms{T}_{\mathsf{K}^p}^+)^{\lozenge /}$.
\end{cons}

\begin{proof}[Proof of Theorem \ref{prop:construction-of-shtuka}] It is clear from our definition of $\ms{P}_{\mathsf{K}^p,\mc{O}_{E^\ur}}$ that the collection $\{\ms{P}_{\eK^p, \mc{O}_{E^\ur}}\}_{\eK^p}$ forms a compatible family of $\mc{G}^c$-shtukas over $\ms{S}_{\mathsf{K}^p}^\mf{d}(\mb{G},\mb{X})_{\mc{O}_{E^\ur}}$. Also, by design, $\ms{P}_{\mathsf{K}^p,\mc{O}_{E^ur}}$ models $\ms{P}_{\mathsf{K}^p,E^\ur}$ since it does so on each connected component. Theorem \ref{prop:construction-of-shtuka} follows.
\end{proof}

\subsection{The completion of the proof}

We are now prepared to prove Theorem \ref{thm:main}. Given Theorem \ref{thm:KPZ-model-properties} and Theorem \ref{prop:construction-of-shtuka} it remains to verify the following proposition. 
\begin{prop}\label{prop:complete-local-rings}
    For any neat compact open subgroup $\mathsf{K}^p\subseteq \mb{G}(\A_f^p)$, and $x$ in $\ms{S}_{\mathsf{K}^p}^\mf{d}(\mb{G},\mb{X})(\ov{k}_E)$, there exists an isomorphism of $v$-sheaves
\begin{equation*}
    \Theta_x\colon \left(\mc{M}^\mr{int}_{\mc{G}^c,b_x,{\mbbmus^c_h}}\right)^\wedge_{/x_0}\isomto \left(\ms{S}_{\mathsf{K}^p}^\mf{d}(\mb{G},\mb{X})^\wedge_{/x}\right)^\lozenge,
\end{equation*}
such that $\Theta_x^\ast(\ms{P}_{\mathsf{K}^p})$ is isomorphic to $\ms{P}^\mr{univ}$, with notation as in Definition \ref{Def:PRAxioms}.
\end{prop} 

The key to constructing such an isomorphism is the following result of Pappas--Rapoport.

\begin{prop}[{\cite[Proposition 5.3.1]{PR2021}}]\label{prop:shtuka-completion-ad-isom} Suppose that $f\colon (\mc{G}_1,\mbbmu_1,b_1)\to (\mc{G},b,\mbbmu)$ is an ad-isomorphism of local Shimura data, and that $x_1$ is an element of $\mc{M}^\mr{int}_{\mc{G}_1,b_1,{\mbbmus_1}}(\ov{k}_{E_1})$ with image $x$ in $\mc{M}^\mr{int}_{\mc{G},b,{\mbbmus}}(\ov{k}_{E})$ $($note that $\ov{k}_{E_1}=\ov{k}_E$$)$. Then the induced map
\begin{equation*}
    f\colon \left(\mc{M}^\mr{int}_{\mc{G}_1,b_1,{\mbbmus_1}}\right)^\wedge_{/x_1}\to \left(\mc{M}^\mr{int}_{\mc{G},b,{\mbbmus}}\right)^\wedge_{/x} \times_{\Spd(\mc{O}_E)} \Spd(\mc{O}_{E^\ur_1}),
\end{equation*}
is an isomorphism, and there is an identification
\begin{equation*}
    f^\ast(\ms{P}_{\mc{G}}^\mr{univ})\simeq \ms{P}_{\mc{G}_1}^\mr{univ}\times^{\mc{G}_1}\mc{G}.
\end{equation*}
\end{prop}

\begin{proof}[Proof of Proposition \ref{prop:complete-local-rings}]Consider the commutative diagram
\begin{equation}\label{eq:big-diag-main-proof-sec}
\begin{tikzcd}
	{\ms{S}_{\mathsf{M}'_{p,1}}(\mb{T}_1',\{h_{\mb{T}'_1}\})_{\mc{O}_E}} & {\ms{S}_{\mathsf{N}_{p,1}}(\mb{G}_1^\mr{ab},\mb{X}_1^\mr{ab})_{\mc{O}_E}} \\
	{\ms{S}^\mf{d}_{\mathsf{L}'_{p,1}}(\mb{B}'_1,\mb{Y}'_1)_{\mc{O}_E}} & {\ms{S}^\vn_{\mathsf{K}_{p,1}}(\mb{G}_1,\mb{X}_1)_{\mc{O}_E}} \\
	{\ms{S}^\mf{d}_{\mathsf{L}_{p}}(\mb{B},\mb{Y})} & {\ms{S}^\mf{d}_{\mathsf{K}_{p}}(\mb{G},\mb{X})}
	\arrow["a", from=2-1, to=2-2]
	\arrow[from=2-1, to=3-1]
	\arrow[from=3-1, to=3-2]
    \arrow["e",from=2-1, to=3-2]
	\arrow["b", from=2-2, to=1-2]
	\arrow["d", from=2-1, to=1-1]
	\arrow["c", from=1-1, to=1-2]
\end{tikzcd}
\end{equation}
from Proposition \ref{prop:big-integral-model-diagram}, where each map is equivariant for the relevant group homomorphism. Choose compact open subgroups 
\begin{equation*}
    {\mathsf{L}'}^{p}\subseteq \mb{B}_1'(\A_f^p),\quad \mathsf{K}^p\subseteq \mb{G}(\A_f^p),\quad {\mathsf{M}'}^{p}_1\subseteq \mb{T}'_1(\A_f^p),\ \text{ and } \ \mathsf{K}^{p}_1\subseteq\mb{G}_1(\A_f^p),
\end{equation*}
such that $(a,e)({\mathsf{L}'}^{p}_1)\subseteq \mathsf{K}_1^{p}\times {\mathsf{M}'_1}^{p}$ and $e({\mathsf{L}'}^{p}_1)\subseteq \mathsf{K}^p$. We then obtain the diagram
\begin{equation*}
   \begin{tikzcd}
	{\ms{S}_{{\mathsf{L}'_1}^{p}}^\mf{d}(\mathbf{B}_1',\mathbf{Y}'_1)_{\mathcal{O}_E}} & {\ms{S}_{\mathsf{K}_{1}^p}^\mf{d}(\mb{G}_1,\mb{X}_1)_{\mc{O}_E}\times \ms{S}_{{\mathsf{M}'_1}^p}(\mb{T}'_1,\{h_{\mb{T}'_1}\})_{\mc{O}_E}} \\
	{\ms{S}_{\mathsf{K}^p}^\mf{d}(\mb{G},\mb{X})} & {\ms{S}_{\mathsf{K}^p_{1}\times{\mathsf{M}'_1}^p}^\mf{d}(\mb{G}_1\times\mb{T}'_1,\mb{X}_1\times\{h_{\mb{T}'_1}\})_{\mc{O}_E},}
	\arrow["e"', from=1-1, to=2-1]
	\arrow["{(a,c)}", from=1-1, to=1-2]
	\arrow["\wr",from=1-2, to=2-2]
\end{tikzcd}
\end{equation*}
from \eqref{prop:big-integral-model-diagram} (excluding some terms) and Proposition \ref{prop:KPZ-for-products}.

Let us now fix a point $w$ of $\ms{S}^\mf{d}_{{\mathsf{L}'_1}^p}(\mb{B}_1',\mb{Y}_1')(\ov{k}_E)$. We set
\begin{equation*}
    x=e(w)\in {\ms{S}_{\mathsf{K}^p}^\mf{d}(\mb{G},\mb{X})}(\ov{k}_E),\qquad y=a(w)\in \ms{S}_{\mathsf{K}_{1}^p}^\mf{d}(\mb{G}_1,\mb{X}_1)_{\mc{O}_E}(\ov{k}_E),
\end{equation*}
and 
\begin{equation*}
    z=c(w)\in \ms{S}_{{\mathsf{M}'_1}^p}(\mb{T}'_1,\{h_{\mb{T}'_1}\})_{\mc{O}_E}(\ov{k}_E).
\end{equation*}
Since the maps
\begin{equation*}
    (\mb{B}'_1)^\der\to \mb{G}^\der\ \text{ and } \ (\mb{B}'_1)^\der\to (\mb{G}_1\times\mb{T}'_1)^\der
\end{equation*}
are isogenies, the maps 
\begin{equation*}
    \mb{B}'_1 \to \mb{G}\ \text{ and } \ \mb{B}'_1 \to \mb{G}_1\times\mb{T}'_1
\end{equation*}
are ad-isomorphisms by Proposition \ref{prop:ad-iso-equiv}. Thus, the maps $e$ and $(a,c)$ are finite \'etale by Proposition \ref{prop:funct-for-ad-isom}, and we obtain induced isomorphisms 
\begin{equation*}
    e\colon \wh{\mc{O}}_x\isomto,\wh{\mc{O}}_w \qquad (a,c)\colon \wh{\mc{O}}_y\otimes_{\mc{O}_{\breve{E}}}\wh{\mc{O}}_z\isomto \wh{\mc{O}}_w.
\end{equation*}
Here we use the obvious notation for the complete local rings of these Kisin--Pappas--Zhou models at these $\ov{k}_E$-points. We also use the fact that $\wh{\mc{O}}_z$ is isomorphic to $\mc{O}_{\breve{E}}$ (see \cite[Lemma 4.2]{DanielsToral}) to relate the local ring $(y,z)$ to the tensor product of local rings.

Now, by Theorem \ref{thm:PR-conj-for-Hodge-type} and Theorem \ref{thm:PR-conj-for-toral-type} we may choose isomorphisms
\begin{equation*}
    \Theta_y\colon \Big(\mc{M}^\mr{int}_{\mc{G}_1,b_y,{\mbbmus_y}}\Big)^\wedge_{/y_0}\isomto \Spf(\wh{\mc{O}}_y)^\lozenge \ \text{ and } \ \Theta_z\colon \Big(\mc{M}^\mr{int}_{{\mc{T}'_1}^c,b_z,{\mbbmus_z}}\Big)^\wedge_{/z_0}\isomto \Spf(\wh{\mc{O}}_z)^\lozenge,
\end{equation*}
such that $\Theta_y^\ast(\ms{P}_{\mathsf{K}^{p}_1})\simeq \ms{P}_y^\mr{univ}$ and $\Theta_z^\ast(\ms{P}_{\mathsf{M}^{'p}_1})\simeq \ms{P}_z^\mr{univ}$, where have used the obvious shortenings for the local cocharacters and universal shtukas. From these, we obtain an isomorphism
\begin{equation*}
    \Theta_y\times\Theta_z\colon \Big(\mc{M}^\mr{int}_{\mc{G}_1,b_y,{\mbbmus_y}}\Big)^\wedge_{/y_0}\times \Big(\mc{M}^\mr{int}_{{\mc{T}'_1}^c,b_z,{\mbbmus_z}}\Big)^\wedge_{/z_0}\isomto \Spf(\wh{\mc{O}}_y)^\lozenge\times\Spf(\wh{\mc{O}}_z)^\lozenge.
\end{equation*}

On the other hand, since ${B}_1'^c\to {G}_1\times{{T}'_1}^c$ is an ad-isomorphism, we see from Proposition \ref{prop:shtuka-completion-ad-isom} that we have an induced isomorphism
\begin{equation*}
    \Big(\mc{M}^\mr{int}_{\mc{B}_1'^c,b_w,{\mbbmus_w}}\Big)^\wedge_{/w_0}\isomto \Big(\mc{M}^\mr{int}_{\mc{G}_1\times{\mc{T}'_1}^{c},b_{(y,z)},{\mbbmus_{(y,z)}}}\Big)^\wedge_{/(y_0,z_0)}.
\end{equation*}
Using the fact that the formation of integral local Shimura varieties commutes with products in the obvious way, we obtain an induced isomorphism
\begin{equation*}
    (\alpha,\gamma)\colon \Big(\mc{M}^\mr{int}_{\mc{B}_1'^c,b_w,{\mbbmus_w}}\Big)^\wedge_{/w_0}\isomto \Big(\mc{M}^\mr{int}_{\mc{G}_1,b_{y},{\mbbmus_y}}\Big)^\wedge_{/y_0}\times \Big(\mc{M}^\mr{int}_{{\mc{T}'_1}^{c},b_{z},{\mbbmus_z}}\Big)^\wedge_{/z_0},
\end{equation*}
where we are again using the fact that $(\mc{M}^\mr{int}_{{\mc{T}'_1}^{c},b_{z},{\mbbmus_z}})^\wedge_{/z_0}$ is a point to pass to the completions.

We may then form the map
\begin{equation*}
    \Theta_w\defeq  (a,c)^{-1}\circ (\Theta_y\times\Theta_z)\circ(\alpha,\gamma)\colon \Big(\mc{M}^\mr{int}_{\mc{B}_1'^c,b_w,{\mbbmus_w}}\Big)^\wedge_{/w_0}\isomto \Spf(\wh{\mc{O}}_w)^\lozenge.
\end{equation*}
Let us observe that by design
\begin{equation}\label{eq:by-design}
    \Theta_w^\ast(\ms{P}_{{\mathsf{L}'_1}^p})\times^{\mc{B}_1'^c}(\mc{G}_1\times{\mc{T}'_1}^c)\simeq \alpha^\ast\ms{P}_{y}^\mr{univ}\times \gamma^\ast\ms{P}_z^\mr{univ}\simeq \ms{P}_w^\mr{univ}\times^{\mc{B}_1'^c}(\mc{G}_1\times{\mc{T}'_1}^c).
\end{equation}
But, since
\begin{equation*}
    \mc{B}_1'^c(\Z_p^\ur)=(\mc{G}_1(\Z_p^\ur)\times{\mc{T}'_1}^c(\Z_p^\ur)) \cap B^{'c}_1(\Q_p^\ur)
\end{equation*}
(as follows from Proposition \ref{prop:Bc}), \cite[Proposition 5.1.3]{PR2022} implies that the map of integral local Shimura varieties
\begin{equation}\label{eq:closed-imm}
    \mc{M}^\mr{int}_{\mc{B}_1'^c, b_w, {\mbbmus_w}} \to \mc{M}^\mr{int}_{\mc{G}_1\times{\mc{T}'_1}^{c},b_{(y,z)},{\mbbmus_{(y,z)}}}
\end{equation}
is a closed immersion. Because the map \eqref{eq:closed-imm} is given at the level of objects by the push forward of shtukas along $\mc{B}_1'^c \to \mc{G}_1 \times {\mc{T}'_1}^c$, it follows from \eqref{eq:by-design} that $\Theta_w^\ast(\ms{P}_{{\mathsf{L}'_1}^p})$ is isomorphic to $\ms{P}_w^\mr{univ}$.

On the other hand, since ${B}_1'^c\to {G}^c$ is an ad-isomorphism, we deduce from Proposition \ref{prop:shtuka-completion-ad-isom} an isomorphism
\begin{equation*}
    \varepsilon\colon \Big(\mc{M}^\mr{int}_{\mc{B}_1'^c,b_w,{\mbbmus_w}}\Big)^\wedge_{/w_0}\isomto \Big(\mc{M}^\mr{int}_{\mc{G},b_x,{\mbbmus_x}}\Big)^\wedge_{/x_0}.
\end{equation*}
We may then define
\begin{equation*}
    \Theta_x\defeq e\circ\Theta_w\circ\varepsilon^{-1}\colon \Big(\mc{M}_{\mc{G},b_x,{\mbbmus_x}}\Big)^\wedge_{/x_0}\isomto \Spf(\wh{\mc{O}}_x)^\lozenge.
\end{equation*}
That $\Theta_x(\ms{P}_{\mathsf{K}^p})$ is isomorphic to $\ms{P}^\mr{univ}_x$ is easy as
\begin{equation*}
    \begin{aligned} \Theta_x(\ms{P}_{\mathsf{K}^p}) &\simeq (\varepsilon^{-1})^\ast(\Theta_w^\ast(e^\ast \ms{P}_{\mathsf{K}^p})))\\ &\simeq (\varepsilon^{-1})^\ast(\Theta_w^\ast(\ms{P}_{{\mathsf{L}'_1}^p}\times^{\mc{B}_1'^c}\mc{G}^c))\\ &\simeq (\varepsilon^{-1})^\ast(\Theta_w^\ast(\ms{P}_{{\mathsf{L}'_1}^p})\times^{\mc{B}_1'^c}\mc{G}^c)\\ &\simeq (\varepsilon^{-1})^\ast(\ms{P}_w^\mr{univ}\times^{\mc{B}_1'^c}\mc{G}^c)\\ & \simeq \ms{P}^\mr{univ}_x.
    \end{aligned}
\end{equation*}
Here the first isomorphism holds by definition, the second holds by our construction of $\ms{P}_{\mathsf{K}^p}$ as in \hyperref[subsub:Step-3-Construction]{\textbf{Step 3}} of \S\ref{s:construction-of-shtuka}, the third holds from abstract nonsense, the fourth holds by our construction of $\Theta_w$, and the final isomorphism holds from the properties of $\varepsilon$ stated in Proposition \ref{prop:shtuka-completion-ad-isom}. 
\end{proof}

The proof of Theorem \ref{thm:main} now follows by combining Theorem \ref{thm:KPZ-model-properties} with Theorem \ref{prop:construction-of-shtuka} and Proposition \ref{prop:complete-local-rings}. 
\bibliographystyle{amsalpha}
\bibliography{Refs}

\providecommand{\bysame}{\leavevmode\hbox to3em{\hrulefill}\thinspace}
\providecommand{\MR}{\relax\ifhmode\unskip\space\fi MR }
\providecommand{\MRhref}[2]{%
  \href{http://www.ams.org/mathscinet-getitem?mr=#1}{#2}
}
\providecommand{\href}[2]{#2}
\begin{thebibliography}{DvHKZ24b}

\bibitem[ALY21]{ALY1P2}
Piotr Achinger, Marcin Lara, and Alex Youcis, \emph{Variants of the de {J}ong
  fundamental group}, arXiv preprint
  \href{http://arxiv.org/abs/2203.11750}{arXiv:2203.11750}, 2021.

\bibitem[BLR90]{BLR}
S.~Bosch, W.~L\"utkebohmert, and M.~Raynaud, \emph{N\'eron models}, Ergebnisse
  der {M}athematik und ihrer {G}renzgebiete, Springer-Verlag, Berlin, 1990.

\bibitem[Bor91]{BorelLAG}
Armand Borel, \emph{Linear algebraic groups}, second ed., Graduate Texts in
  Mathematics, vol. 126, Springer-Verlag, New York, 1991. \MR{1102012}

\bibitem[Bor98]{Borovoi}
Mikhail Borovoi, \emph{Abelian {G}alois cohomology of reductive groups}, Mem.
  Amer. Math. Soc. \textbf{132} (1998), no.~626, viii+50. \MR{1401491}

\bibitem[BS15]{BhattScholzeProetale}
Bhargav Bhatt and Peter Scholze, \emph{The pro-\'{e}tale topology for schemes},
  Ast\'{e}risque (2015), no.~369, 99--201. \MR{3379634}

\bibitem[BT84]{BTII}
F.~Bruhat and J.~Tits, \emph{Groupes r\'eductifs sur un corps local. {II}.
  {S}ch\'emas en groups. existence d'une donn\'ee radicielle valu\'ee}, Inst.
  Hautes \'Etudes Sci. Publ. Math. (1984), no.~60, 197--376.

\bibitem[Dan22]{DanielsToral}
P.~Daniels, \emph{Canonical integral models for {S}himura varieties of toral
  type}, arxiv preprint arxiv:2207.09513v2 (2022).

\bibitem[Del79]{DeligneModulaire}
Pierre Deligne, \emph{Vari\'{e}t\'{e}s de {S}himura: interpr\'{e}tation
  modulaire, et techniques de construction de mod\`eles canoniques},
  Automorphic forms, representations and {$L$}-functions ({P}roc. {S}ympos.
  {P}ure {M}ath., {O}regon {S}tate {U}niv., {C}orvallis, {O}re., 1977), {P}art
  2, Proc. Sympos. Pure Math., XXXIII, Amer. Math. Soc., Providence, R.I.,
  1979, pp.~247--289. \MR{546620}

\bibitem[DvHKZ24a]{DvHKZIgusa}
Patrick Daniels, Pol van Hoften, Dongryul Kim, and Mingjia Zhang, \emph{Igusa
  stacks for {S}himura varieties of {H}odge type}, preprint (2024).

\bibitem[DvHKZ24b]{DvHKZ}
\bysame, \emph{On a conjecture of {P}appas and {R}apoport}, preprint (2024).

\bibitem[FK18]{FujiwaraKato}
Kazuhiro Fujiwara and Fumiharu Kato, \emph{Foundations of rigid geometry. {I}},
  EMS Monographs in Mathematics, European Mathematical Society (EMS),
  Z\"{u}rich, 2018. \MR{3752648}

\bibitem[Gir71]{Giraud}
Jean Giraud, \emph{Cohomologie non ab\'{e}lienne}, Die Grundlehren der
  mathematischen Wissenschaften, Band 179, Springer-Verlag, Berlin-New York,
  1971. \MR{0344253}

\bibitem[Gle20]{Gleason2020}
I.~Gleason, \emph{Specialization maps for {S}cholze's category of diamonds},
  arXiv preprint arXiv:2012.05483 (2020).

\bibitem[Gle22]{Gleason2022}
\bysame, \emph{On the geometric connected components of moduli spaces of
  $p$-adic shtukas and local {S}himura varieties}, arxiv preprint
  arxiv:2107.03579 (2022), 33.

\bibitem[Gro66]{EGA4-3}
A.~Grothendieck, \emph{\'{E}l\'{e}ments de g\'{e}om\'{e}trie alg\'{e}brique.
  {IV}. \'{E}tude locale des sch\'{e}mas et des morphismes de sch\'{e}mas.
  {III}}, Inst. Hautes \'{E}tudes Sci. Publ. Math. (1966), no.~28, 255.
  \MR{217086}

\bibitem[Gro67]{EGA4-4}
\bysame, \emph{\'{E}l\'{e}ments de g\'{e}om\'{e}trie alg\'{e}brique. {IV}.
  \'{E}tude locale des sch\'{e}mas et des morphismes de sch\'{e}mas {IV}},
  Inst. Hautes \'{E}tudes Sci. Publ. Math. (1967), no.~32, 361. \MR{238860}

\bibitem[GW20]{GortzWedhorn}
Ulrich G\"{o}rtz and Torsten Wedhorn, \emph{Algebraic geometry {I}.
  {S}chemes---with examples and exercises}, Springer Studium
  Mathematik---Master, Springer Spektrum, Wiesbaden, [2020] \copyright 2020,
  Second edition [of 2675155]. \MR{4225278}

\bibitem[He16]{HeKRConj}
Xuhua He, \emph{Kottwitz-{R}apoport conjecture on unions of affine
  {D}eligne-{L}usztig varieties}, Ann. Sci. \'{E}c. Norm. Sup\'{e}r. (4)
  \textbf{49} (2016), no.~5, 1125--1141.

\bibitem[HL23]{Hamann-Lee}
Linus Hamann and Si~Ying Lee, \emph{Torsion vanishing for some {S}himura
  varieties}, arXiv e-prints (2023).

\bibitem[HR08]{HainesRapoport}
Thomas Haines and Michael Rapoport, \emph{On parahoric subgroups}, Advances in
  Mathematics \textbf{219} (2008), no.~1, 188--198.

\bibitem[IKY23]{ImaiKatoYoucis}
N.~Imai, H.~Kato, and A.~Youcis, \emph{The prismatic realization functor for
  {S}himura varieties of abelian type}, preprint (2023).

\bibitem[Kis10]{Kisin2010}
M.~Kisin, \emph{Integral models for {S}himura varieties of abelian type}, J.
  Amer. Math. Soc. \textbf{23} (2010), no.~4, 967--1012.

\bibitem[Kot84]{KotShtw}
Robert~E. Kottwitz, \emph{Shimura varieties and twisted orbital integrals},
  Math. Ann. \textbf{269} (1984), no.~3, 287--300. \MR{761308}

\bibitem[Kot90]{KottwitzAnnArbor}
\bysame, \emph{Shimura varieties and {$\lambda$}-adic representations},
  Automorphic forms, {S}himura varieties, and {$L$}-functions, {V}ol. {I}
  ({A}nn {A}rbor, {MI}, 1988), Perspect. Math., vol.~10, Academic Press,
  Boston, MA, 1990, pp.~161--209. \MR{1044820}

\bibitem[Kot97]{KottwitzIsocrystalsII}
\bysame, \emph{Isocrystals with additional structure. {II}}, Compositio Math.
  \textbf{109} (1997), no.~3, 255--339. \MR{1485921}

\bibitem[KP18]{KP2018}
M.~Kisin and G.~Pappas, \emph{Integral models of {S}himura varieties with
  parahoric level structure}, Publ. Math. Inst. Hautes \'Etudes Sci. (2018),
  no.~128, 121--218.

\bibitem[KP23]{KalethaPrasad}
Tasho Kaletha and Gopal Prasad, \emph{Bruhat-{T}its theory---a new approach},
  New Mathematical Monographs, vol.~44, Cambridge University Press, Cambridge,
  2023. \MR{4520154}

\bibitem[KSZ21]{KSZ}
M.~Kisin, S.W. Shin, and Y.~Zhu, \emph{The stable trace formula for {S}himura
  varieties of abelian type}, arxiv preprint arXiv:2110.05381 (2021).

\bibitem[KZ21]{KZ21}
M.~Kisin and R.~Zhou, \emph{Independence of $\ell$ for {F}robenius conjugacy
  classes attached to abelian varieties}, arxiv preprint arxiv:2103.09945
  (2021), 66.

\bibitem[Lan00]{Landvogt2000}
E.~Landvogt, \emph{Some functorial properties of the {Bruhat-Tits} building},
  J. Reine Angew. Math. \textbf{518} (2000), 213 -- 241.

\bibitem[Lov17a]{Lovering}
T.~Lovering, \emph{Integral canonical models for automorphic vector bundles of
  abelian type}, Algebra Number Theory \textbf{11} (2017), no.~8, 1837--1890.

\bibitem[Lov17b]{LoveringFCrystals}
Tom Lovering, \emph{Filtered f-crystals on {S}himura varieties of abelian
  type}, 2017.

\bibitem[LR87]{LanglandsRapoport}
R.~P. Langlands and M.~Rapoport, \emph{Shimuravariet\"{a}ten und {G}erben}, J.
  Reine Angew. Math. \textbf{378} (1987), 113--220. \MR{895287}

\bibitem[MFK94]{MumfordGIT}
D.~Mumford, J.~Fogarty, and F.~Kirwan, \emph{Geometric invariant theory}, third
  ed., Ergebnisse der Mathematik und ihrer Grenzgebiete (2) [Results in
  Mathematics and Related Areas (2)], vol.~34, Springer-Verlag, Berlin, 1994.
  \MR{1304906}

\bibitem[Mil05]{Milne2005}
J.~S. Milne, \emph{Introduction to {S}himura varieties}, Harmonic analysis, the
  trace formula, and {S}himura varieties, Clay Math. Proc., vol.~4, Amer. Math.
  Soc., 2005, pp.~265--378.

\bibitem[Mil13]{MilneModuli}
\bysame, \emph{Shimura varieties and moduli}, Handbook of moduli. {V}ol. {II},
  Adv. Lect. Math. (ALM), vol.~25, Int. Press, Somerville, MA, 2013,
  pp.~467--548. \MR{3184183}

\bibitem[Mil17]{Milne2017}
J.S. Milne, \emph{Algebraic groups: The theory of group schemes of finite type
  over a field}, vol. 170, Cambridge University Press, 2017.

\bibitem[Moo98]{Moonen1998}
B.~Moonen, \emph{Models of {S}himura varieties in mixed characteristics},
  Galois representations in arithmetic algebraic geometry ({D}urham, 1996),
  London Math. Soc. Lecture Note Ser., vol. 254, Cambridge Univ. Press, 1998,
  pp.~267--350.

\bibitem[MP16]{PeraSpin}
Keerthi Madapusi~Pera, \emph{Integral canonical models for spin {S}himura
  varieties}, Compos. Math. \textbf{152} (2016), no.~4, 769--824. \MR{3484114}

\bibitem[MS81]{MilneShih}
J.~S. Milne and Kuang-yen Shih, \emph{The action of complex conjugation on a
  {S}himura variety}, Ann. of Math. (2) \textbf{113} (1981), no.~3, 569--599.
  \MR{621017}

\bibitem[Pap22]{Pappas2022}
G.~Pappas, \emph{On integral models of {S}himura varieties}, Math. Ann. (2022),
  61.

\bibitem[PR21]{PR2021}
G.~Pappas and M.~Rapoport, \emph{$p$-adic shtukas and the theory of local and
  global {S}himura varieties}, arxiv preprint arXiv:2106.08270 (2021).

\bibitem[PR22]{PR2022}
\bysame, \emph{On integral local {S}himura varieties}, arxiv preprint
  arxiv:2204.02829 (2022), 56.

\bibitem[RV14]{RV2014}
M.~Rapoport and E.~Viehmann, \emph{Towards a theory of local {Shimura}
  varieties}, M{\"u}nster Journal of Mathematics \textbf{7} (2014), no.~1,
  273--326.

\bibitem[Sch13]{Scholze2013}
P.~Scholze, \emph{$p$-adic {H}odge theory for rigid-analytic varieties}, Forum
  Math. Pi. \textbf{1} (2013), 77.

\bibitem[Sch17]{Scholze2017}
\bysame, \emph{{\'E}tale cohomology of diamonds}, arXiv preprint
  arXiv:1709.07343 (2017).

\bibitem[{Sta}17]{stacks-project}
The {Stacks Project Authors}, \emph{{Stacks Project}},
  \url{http://stacks.math.columbia.edu}, 2017.

\bibitem[Ste65]{Steinberg}
Robert Steinberg, \emph{Regular elements of semisimple algebraic groups}, Inst.
  Hautes \'{E}tudes Sci. Publ. Math. (1965), no.~25, 49--80. \MR{180554}

\bibitem[SW20]{SW2020}
P.~Scholze and J.~Weinstein, \emph{Berkeley lectures on $p$-adic geometry:
  $(${AMS}-207$)$}, Annals of {M}athematical {S}tudies, Princeton University
  Press, 2020.

\bibitem[vHS24]{vHS}
Pol van Hoften and Jack Sempliner, \emph{On the {P}iatetski-{S}hapiro
  construction for integral models of {S}himura varieties}, preprint (2024).

\bibitem[Zha23]{Zhang}
Mingjia Zhang, \emph{A {PEL}-type {I}gusa stack and the $p$-adic geometry of
  {S}himura varieties}, arxive preprint arxiv:2309.05152 (2023).

\end{thebibliography}

\end{document}